\documentclass[11pt]{article}
\usepackage[utf8]{inputenc}

\usepackage{amsmath, amsthm, amsfonts,stmaryrd,amssymb}
\usepackage[top=1in, bottom=1.25in, left=1.25in, right=1.25in]{geometry}
\usepackage{mathrsfs}
\usepackage{dsfont}                  % Nice \mathscr fonts
\usepackage{mathtools}   
\usepackage{thmtools}
\usepackage{booktabs}
\usepackage{tikz}
\usepackage{dirtytalk}
\usepackage{todonotes}
\usepackage[shortlabels]{enumitem}
\usepackage{hyperref}
\usepackage{xcolor}
\hypersetup{
	colorlinks=true,
	linkcolor=blue,  
	urlcolor=cyan,
	pdftitle={},
	pdfauthor={},
}
\usepackage{verbatim}
\usepackage[capitalise]{cleveref}
\usepackage{bm}

\usepackage{biblatex}

\usepackage{nomencl}
\makenomenclature

\usepackage{pgfplots}
\pgfplotsset{compat=1.15}
\usepackage{mathrsfs}
\usetikzlibrary{arrows}
\newtheorem{theorem}{Theorem}[section]
\newtheorem{lemma}[theorem]{Lemma}
\newtheorem{corollary}[theorem]{Corollary}

\newtheorem{proposition}[theorem]{Proposition}

\theoremstyle{definition}
\newtheorem{definition}[theorem]{Definition}
\newtheorem{example}[theorem]{Example}
\newtheorem{remark}[theorem]{Remark}
\crefname{claim}{claim}{claims}

\def\R{{\mathbb R}}

\def\Ec{{\mathcal E}}
\def\Pc{{\mathcal P}}

\DeclareMathOperator*{\argmax}{arg\,max}
\DeclareMathOperator*{\argmin}{arg\,min}

\newcommand{\bracket}[1]{\langle #1 \rangle}
\newcommand{\abs}[1]{\lvert #1 \rvert}\newcommand{\norm}[1]{\lVert #1 \rVert}

\newcommand{\OT}{\mathcal{T}}
\newcommand{\ET}{\mathcal{T}}
\newcommand{\UT}{\mathcal{U}\mathcal{T}}
\newcommand{\Mp}{{\mathcal{M}_+}}
\newcommand{\Mm}{{\mathcal{M}_-}}

\newcommand{\Lc}{{\mathcal L}}
\newcommand{\Mc}{{\mathcal M}}
\newcommand{\Tc}{{\mathcal T}}

\newcommand{\eps}{\varepsilon}

\DeclareMathOperator{\dom}{dom}
\DeclareMathOperator{\KL}{KL}
\newcommand{\leP}{\ll_{\mathrm P}}
\newcommand{\leQ}{\ll_{\mathrm Q}}
\newcommand{\Leb}{{\mathcal L}}
\newcommand{\JKO}{\mathcal{S}}

\DeclareMathOperator{\dist}{dist}
\newcommand{\Omegax}{{\Omega}}
\newcommand{\Omegay}{{\Omega^*}}
\newcommand{\Rinf}{\R\cup\{+\infty\}}
\DeclareMathOperator{\supp}{supp}
\makeatletter
\let\save@mathaccent\mathaccent
\newcommand*\if@single[3]{%
  \setbox0\hbox{${\mathaccent"0362{#1}}^H$}%
  \setbox2\hbox{${\mathaccent"0362{\kern0pt#1}}^H$}%
  \ifdim\ht0=\ht2 #3\else #2\fi
  }
\newcommand*\rel@kern[1]{\kern#1\dimexpr\macc@kerna}
\newcommand*\widebar[1]{\@ifnextchar^{{\wide@bar{#1}{0}}}{\wide@bar{#1}{1}}}
\newcommand*\wide@bar[2]{\if@single{#1}{\wide@bar@{#1}{#2}{1}}{\wide@bar@{#1}{#2}{2}}}
\newcommand*\wide@bar@[3]{%
  \begingroup
  \def\mathaccent##1##2{%
    \let\mathaccent\save@mathaccent
    \if#32 \let\macc@nucleus\first@char \fi
    \setbox\z@\hbox{$\macc@style{\macc@nucleus}_{}$}%
    \setbox\tw@\hbox{$\macc@style{\macc@nucleus}{}_{}$}%
    \dimen@\wd\tw@
    \advance\dimen@-\wd\z@
    \divide\dimen@ 3
    \@tempdima\wd\tw@
    \advance\@tempdima-\scriptspace
    \divide\@tempdima 10
    \advance\dimen@-\@tempdima
    \ifdim\dimen@>\z@ \dimen@0pt\fi
    \rel@kern{0.6}\kern-\dimen@
    \if#31
      \overline{\rel@kern{-0.6}\kern\dimen@\macc@nucleus\rel@kern{0.4}\kern\dimen@}%
      \advance\dimen@0.4\dimexpr\macc@kerna
      \let\final@kern#2%
      \ifdim\dimen@<\z@ \let\final@kern1\fi
      \if\final@kern1 \kern-\dimen@\fi
    \else
      \overline{\rel@kern{-0.6}\kern\dimen@#1}%
    \fi
  }%
  \macc@depth\@ne
  \let\math@bgroup\@empty \let\math@egroup\macc@set@skewchar
  \mathsurround\z@ \frozen@everymath{\mathgroup\macc@group\relax}%
  \macc@set@skewchar\relax
  \let\mathaccentV\macc@nested@a
  \if#31
    \macc@nested@a\relax111{#1}%
  \else
    \def\gobble@till@marker##1\endmarker{}%
    \futurelet\first@char\gobble@till@marker#1\endmarker
    \ifcat\noexpand\first@char A\else
      \def\first@char{}%
    \fi
    \macc@nested@a\relax111{\first@char}%
  \fi
  \endgroup
}
\makeatother

\newcommand{\Ub}{{\widebar U}}

\title{A synthetic approach to comparison principles for variational problems, 
with applications to optimal transport}
\author{Flavien Léger, Maxime Sylvestre}

\date{\today}
\begin{document}

\maketitle

\begin{abstract}
    % We study variational problems on Banach spaces which involve submodular energies. We extend the notion of substitutability to this infinite dimensional setting and show that it is in duality with submodularity. These two notions allow us to derive comparison principles in an abstract fashion. We apply our results to the optimal transport, entropic optimal transport and unbalanced optimal transport problems. We then derive comparison principles on the Kantorovich, Schrödinger and unbalanced potentials. We also prove comparison principles for the associated JKO schemes.
    We develop a synthetic, variational framework for deriving comparison principles in infinite-dimensional Banach spaces. Unlike traditional approaches that rely on the regularity of minimizers and Euler--Lagrange equations, our method exploits the order-theoretic structure of the energy. Central to our analysis is the notion of submodularity and its convex dual, substitutability, which we extend here to the infinite-dimensional setting. We prove a duality theorem establishing that a convex functional is submodular if and only if its conjugate is substitutable.
    We apply these results to problems in optimal transport, and derive comparison principles for Kantorovich potentials in standard, entropic, and unbalanced settings without requiring regularity assumptions on the cost or domain. Finally, we prove that general transport costs are substitutable, yielding comparison principles for JKO schemes driven by internal energies.
\end{abstract}
\setcounter{tocdepth}{3}
\tableofcontents

\section{Introduction}

% Comparison principles are cornerstones of the theory of partial differential equations (PDEs) and calculus of variations. 
Comparison principles are fundamental tools in the study of partial differential equations (PDEs) and variational problems.
Classically, the notion that ``ordered data leads to ordered solutions'' allows one to establish uniqueness of solutions, construct barriers for regularity, or derive pointwise bounds on solutions. For a variational problem of the form
\begin{equation}\label{eq:intro-var-pb}
    \min_{u\in X} \, \Ec(u,f),
\end{equation}
a comparison principle typically says that if the parameters satisfy $f_1 \leq f_2$ (in the pointwise order), then the respective unique minimizers satisfy $u(f_1) \leq u(f_2)$. The standard route to proving such results in the continuum setting is differential: one derives the Euler--Lagrange equations (e.g., a quasilinear elliptic PDE), assumes sufficient regularity of the data and the domain to ensure the existence of smooth solutions, and applies a maximum principle. 

In this paper, we propose a fundamentally different, ``synthetic'' approach. We derive comparison principles purely from the variational structure of the problem, relying on specific order-theoretic properties of the energy. This approach offers two distinct advantages. First, it frees us from regularity assumptions. By bypassing the differential formulation entirely, we obtain results for non-smooth quantities, for general measures that may lack absolute continuity, and over infinite-dimensional domains. Second, it allows us to formulate comparison principles even when uniqueness of solutions fails. 
% Our method does not establish uniqueness itself; that must be obtained by other means, such as the strict convexity of the energy. 
Instead of comparing singletons, we derive orderings between \textit{solution sets} $S_i(f) \coloneqq \argmin \Ec(\cdot, f)$.
% However, when minimizers are not unique, we derive orderings between \textit{solution sets} $S(f) \coloneqq \argmin \Ec(\cdot, f)$. 
We establish comparisons of the form
\begin{equation}\label{eq:intro-cp-nonunique}
    f_1\leq f_2\implies S(f_1) \ll S(f_2),
\end{equation}
where the relation $\ll$ denotes one of two orderings: \emph{P-dominance} or \emph{Q-dominance}. For instance, P-dominance says that for any $u_1 \in S(f_1)$ and $u_2 \in S(f_2)$, the pointwise infimum $u_1 \wedge u_2$ lies in $S(f_1)$ while the pointwise supremum $u_1 \vee u_2$ belongs to $S(f_2)$.
% where the relation $\ll$ may be one of seve(P-dominance) implies that for any $u_1 \in S(f_1)$ and $u_2 \in S(f_2)$, the pointwise infimum $u_1 \wedge u_2$ remains in $S(f_1)$ while the pointwise supremum $u_1 \vee u_2$ moves to $S(f_2)$. 
We emphasize that our framework does not establish uniqueness of solutions; this property must be obtained by other means, such as the strict convexity of the energy. However, whenever uniqueness holds, our set-valued results naturally collapse to the standard pointwise comparison principle $u_1 \leq u_2$.
% When solutions are known to be unique, this collapses to the standard comparison $u_1 \leq u_2$. 

The machinery driving these results rests on two dual concepts: the classical notion of \emph{submodularity} and the recently introduced notion of \emph{substitutability}.

\subsection{Submodularity and substitutability}

The notion of submodularity has a long history, appearing in various fields under different names. 
In potential theory, Choquet studied submodularity (under the name $2$-alternating capacities) to extend the concept of capacity from compact to analytic sets \cite{choquet}.
% In potential theory, Choquet studied submodularity (termed $2$-alternating capacities) to establish regularity properties for potentials on non-compact sets (the capacitability theorem).
In economics and operations research, submodularity is the engine behind the theory of \emph{monotone comparative statics} \cite{Topkis1978,MilgromShannon1994,TopkisBook1998}.
Pioneered by Topkis in the context of lattice programming, this theory establishes conditions under which optimal decisions may increase or decrease in response to changes in the problem data. We adapt this lattice-theoretic mechanism to derive comparison principles for functionals on Banach spaces. Submodularity  is also a central concept in combinatorial optimization \cite{FujishigeBook2005,jon,Bach2013}, often regarded there as the discrete analogue of convexity \cite{Lovasz1983, Murota2003}, where it allows for the efficient minimization of set functions.

In the context of a space of functions $X$ equipped with a pointwise order (such as a space of continuous functions, or an $L^p$ space), a functional $E$ is called \emph{submodular} if, for all $\phi_1, \phi_2 \in X$,
\begin{equation*}\label{eq:intro-submodular}
    E(\phi_1 \wedge \phi_2)+E(\phi_1 \vee\phi_2)\leq E(\phi_1)+E(\phi_2).
\end{equation*}
Here, $\wedge$ and $\vee$ correspond to the pointwise minimum and maximum of two functions. 
While typically associated with discrete problems, submodularity is pervasive in the calculus of variations. For instance, the Dirichlet energy $\phi \mapsto \int_\Omega \abs{\nabla \phi}^2$ is submodular on $H^1(\Omega)$ (in fact, it satisfies the inequality with equality, see \Cref{ex:dirichlet}). Similarly, the perimeter functional exhibits this property; notably, Chambolle and Darbon exploit this submodularity to derive comparison principles for perimeter-minimizing sets in the context of image analysis \cite{Chambolle2009}. Other examples include fractional Sobolev seminorms, and concave functions of a sum (e.g., $u \mapsto g(\int_\Omega u)$ for a concave $g$), see \Cref{sec:examples-submodular}. In this work, we review these properties within the unified framework of Banach lattices, making these tools accessible for applications in analysis. 

% It turns out that submodularity alone is insufficient for many applications, particularly those arising from submodular minimization via convex duality.
% It is a known result that the convex conjugate of a submodular function is \emph{supermodular} (inequality \eqref{eq:intro-submodular} holds reversed). Unfortunately, supermodularity does not generally yield comparison principles for minimization problems.
% To resolve this, we study \emph{substitutability}, a condition which was recently identified as the conjugate notion to submodularity in a finite-dimensional context by Galichon, Hsieh, and Sylvestre \cite{GalichonHsiehSylvestre2024}. A discrete counterpart to substitutability appeared in the work of Mutora under the name $M^\sharp$-convexity \cite{Murota2003}. 
% In the same direction \cite{Chen2020} proposes a notion close to substitutability under regularity assumptions. 
% In both cases monotone comparative statics are derived in a finite dimensional setting. In this paper, we instead build upon \cite{GalichonHsiehSylvestre2024} and extend the approach to the infinite-dimensional setting: specifically, P-dominance (submodularity) in the primal space corresponds to Q-dominance (substitutability) in the dual space. This duality result (\cref{cor:submodular-exchangeable}) allows us to detect the presence of comparison principles in a variational problem by inspecting its convex dual. 

It turns out that submodularity alone is insufficient for many applications, particularly those arising from submodular minimization via convex duality. It is a known result that the convex conjugate of a submodular function is \emph{supermodular} (i.e. inequality \eqref{eq:intro-submodular} holds reversed). Unfortunately, supermodularity does not generally yield comparison principles for minimization problems. 
To resolve this, we study \emph{substitutability}. Recently introduced in a finite-dimensional context by Galichon, Hsieh, and Sylvestre~\cite{GalichonHsiehSylvestre2024}, this condition proves to be the correct conjugate notion to submodularity. Precursors to this idea exist in the finite-dimensional literature: Murota studied a discrete counterpart known as M$^\natural$-convexity~\cite{Murota2003}, while Chen and Li~\cite{Chen2020} proposed a similar notion under regularity assumptions. While these works derived monotone comparative statics in finite dimensions, we build here upon~\cite{GalichonHsiehSylvestre2024} to extend the framework to the infinite-dimensional setting. We establish that P-dominance (submodularity) in the primal space corresponds to Q-dominance (substitutability) in the dual space. This duality result (\cref{cor:submodular-exchangeable}) allows us to detect the presence of comparison principles in a variational problem purely by inspecting the structure of its convex dual.

\subsection{Applications to optimal transport}

We demonstrate the power of our approach by applying it to certain optimal transport problems. We obtain two main classes of results:

\subsubsection*{Comparison of Kantorovich Potentials} 

We show that the dual optimal transport functional is submodular (\Cref{lemma:K-submodular}). From this observation, we derive a general comparison principle for Kantorovich potentials (\Cref{thm:comp-principle-potentials}). Our result allows for a comparison where the source measures $\mu_1, \mu_2$ are ordered on a subset $U \subset\Omega$, while the potentials satisfy a boundary condition on $\Omega \setminus U$. 
This formulation recovers the comparison principle for the Monge--Ampère equation but requires no regularity of the cost or the domain. Furthermore, because our approach relies solely on the submodular structure of the dual problem, it applies equally to \emph{entropic optimal transport} and \emph{unbalanced optimal transport} (\Cref{thm:entropic-ot-comp}, \Cref{thm:uot-pot}), settings where PDE methods are significantly more complex or unavailable.

To be precise, let $\Omegax$ and $\Omegay$ be two compact metric spaces and let $c$ be a continuous function over $\Omegax \times \Omegay$. We obtain results for three major classes of transport costs $\Tc(\mu,\nu)$:
\begin{enumerate}[label=(\Roman*), leftmargin=*]
    \item \textbf{Standard Optimal Transport:}
    \begin{equation*}
        \Tc_c(\mu,\nu) = \inf_{\pi \in \Pi(\mu,\nu)} \int_{\Omegax \times \Omegay} c(x,y) \,d\pi(x,y).
    \end{equation*}
    
    \item \textbf{Entropic Optimal Transport:}
    \begin{equation*}
        \ET_{c,\eps}(\mu,\nu) = \inf_{\pi \in \Pi(\mu,\nu)} \int_{\Omegax \times \Omegay} c(x,y) \,d\pi(x,y) + \eps \KL(\pi \mid \alpha \otimes \beta).
    \end{equation*}
    (with $\eps > 0$ and reference measures $\alpha$, $\beta$).
    
    \item \textbf{Unbalanced Optimal Transport:} 
    \begin{equation*}
        \UT_{h,c}(\mu,\nu) = \inf_{\pi \in \Mp(\Omegax\times\Omegay)} H_{h_0,\mu}(\pi_0) + H_{h_1,\nu}(\pi_1)+ \int_{\Omegax \times \Omegay} c(x,y) \,d\pi(x,y).
    \end{equation*} 
     (where $H_{h_i,m}$ are convex internal energy functionals defined in \Cref{def:internal-energy}).
\end{enumerate}

We delay precise definitions to \cref{sec:cp-exch}.

\begin{theorem}[Theorems \ref{thm:comp-principle-potentials}, \ref{thm:entropic-ot-comp} and \ref{thm:uot-pot}]\label{thm:intro-ot-potentials}    
    Let $\mu_1,\mu_2$ be two measures on $\Omegax$ and let $\nu$ be a  measure on $\Omegay$. 
    For $i=1,2$, let $\phi_i \in \Phi(\mu_i,\nu)$, where $\Phi$ denotes the set of dual optimizers (Kantorovich potentials) corresponding to the chosen cost $\Tc$.
    
    Let $U$ be a Borel subset of $\Omegax$. Suppose that $\mu_1\leq\mu_2$ on $U$ and that $\phi_1\leq\phi_2$ on $\Omegax\setminus U$. Then:
    \begin{equation}\label{eq:intro-thm:comp-principle-potentials}
        \phi_1\wedge\phi_2\in\Phi(\mu_1,\nu) 
        \quad\text{and}\quad
        \phi_1\vee\phi_2\in\Phi(\mu_2,\nu).
    \end{equation}
    Additionally, $\phi_1 \leq \phi_2$ on the support of $(\mu_2-\mu_1)$.
\end{theorem}

\begin{remark}
    The conclusion \eqref{eq:intro-thm:comp-principle-potentials} implies an ordering of the \emph{solution sets}. If at least one of the two potentials is known to be unique up to an additive constant, then \eqref{eq:intro-thm:comp-principle-potentials} can be strengthened into a standard comparison of the form 
\[
\phi_1\leq\phi_2 \text{ on } \Omega. 
\]
\end{remark}

To the best of our knowledge, these results are new at this level of generality. 
In the related but distinct setting of the standard JKO scheme (see also the next section), Jacobs, Kim, and Tong established a comparison principle for dual potentials~\cite{Jacobs2020TheP}. We note that their analysis also exploits monotonicity properties akin to submodularity (specifically in their Lemma 4.1). However, their arguments rely on the specific coupling enforced by the JKO optimality conditions: namely, the convex duality relation linking the Kantorovich potential to the derivative of the internal energy. 
Our variational framework complements this by isolating the lattice structure of the transport cost itself, when the specific energy-potential relation is absent. 

\subsubsection*{Comparison of JKO Flows}

The Jordan--Kinderlehrer--Otto (JKO) scheme is a minimizing movement scheme used to construct and simulate gradient flows in the space of probability measures \cite{Jordan1998}. 
Let $\Omegax$ and $\Omegay$ be two compact metric spaces and let $c$ be a continuous function over $\Omegax \times \Omegay$. Let $\Tc$ denote one of the three classes of transport costs introduced above. 
Given a convex ``internal energy'' functional $E\colon\Mp(\Omega)\to \R$, consider the minimization problem
    \begin{equation}\label{eq:intro-ot-min-problem}
    \min_{\nu\in\Mp(\Omegay)} \OT(\mu,\nu)+E(\nu).
    \end{equation}
Here $\Mp(\Omegay)$ denotes the set of positive Borel measures on $\Omegay$.
We prove the following result.

\begin{theorem}[\cref{thm:ot-cp}]\label{thm:intro-ot}
    % Let $\Omegax$ and $\Omegay$ be two compact metric spaces and let $c$ be a continuous function over $\Omegax\times\Omegay$. Denote by $\Mp(\Omegax)$, $\Mp(\Omegay)$ the set of positive Borel measures on $\Omegax$ and $\Omegay$ respectively. 
    Let $\mu_1,\mu_2\in \Mp(\Omega)$ and suppose that \eqref{eq:intro-ot-min-problem} admits a unique minimizer $\nu_i$ with data $\mu_i$, $i=1,2$. Then 
    \begin{equation*} \label{eq:intro-thm-ot-1}
        \mu_1\leq\mu_2 \implies \nu_1\leq\nu_2.
    \end{equation*}
\end{theorem}

Beyond the standard benefits of comparison principles, there are two specific consequences in the JKO setting. First, the ordering established at the discrete level persists in the continuous-time limit, therefore validating the comparison principle for the associated gradient flow PDEs. Second, because the scheme preserves total mass, a classical result of Crandall and Tartar \cite{CrandallTartar1980} implies non-expansivity of the flow in the Total Variation (TV) distance (see \cref{cor:tv-contraction}). In translation-invariant settings, this TV contraction further implies that the bounded variation norm of the solution decreases over time.

Recently, Jacobs, Kim, and Tong established the TV contraction principle (called $L^1$ contraction there) for standard optimal transport \cite{Jacobs2020TheP}. Their proof relies on the existence of optimal maps, necessitating costs that are differentiable and satisfy the twist condition. Let us also mention the earlier work \cite{AlexanderKimYao2014} which proves comparison principles for Rényi entropies and the quadratic cost. 
In complement to these two works, we derive comparison principles for the standard optimal transport with general ground cost, the entropic transport and the unbalanced transport, by identifying that the costs $\Tc(\mu, \nu)$ are \emph{substitutable} with respect to their arguments (\Cref{lemma:Tc-exchangeable}). 
% This property holds for continuous costs on compact metric spaces without any twist or regularity assumptions. Consequently, we obtain a comparison principle for the JKO scheme (\Cref{thm:ot-cp}) that covers standard, entropic, and unbalanced transport in a single stroke.

Importantly, the enlarged class of costs we consider is far from arbitrary; rather, these functionals model distinct physical regimes and computational schemes that cannot be captured by standard quadratic transport.

\textbf{Standard optimal transport with general costs.}
Varying the ground cost $c(x,y)$ beyond the squared Euclidean distance drives distinct physical evolutions. For instance, costs growing as $\abs{x-y}^q$ with $q \neq 2$ generate doubly nonlinear evolution equations \cite{Agueh2005}, such as the $p$-Laplacian flow, which model turbulent filtration and non-Newtonian fluids. 
Generalized ground costs such as Bregman divergences can also be used for computational purposes, as efficient surrogates for Riemannian distances \cite{rankin2024jkoschemesgeneraltransport}. 

\textbf{Entropic Optimal Transport.} 
The entropic JKO scheme was first proposed as a computationally efficient proxy for standard Wasserstein gradient flows \cite{Peyre2015,CarlierDuvalPeyreSchmitzer2017}.
Theoretical properties of these ``entropic gradient flows'' are now actively investigated in their own right \cite{baradat2025usingsinkhornjkoscheme}, see also the related works \cite{hardion2025gradientflowspotentialenergies,agarwal2024iteratedschrodingerbridgeapproximation}.

% \textbf{Unbalanced Optimal Transport} relaxes the mass conservation constraint, allowing for the modeling of systems with growth and decay \cite{LieroMielkeSavare2018,ChizatPeyreSchmitzerVialard2018JFA}. In the JKO framework, this cost naturalizes the study of \emph{reaction-diffusion equations}, where diffusion (transport) competes with the creation or destruction of mass (reaction) \cite{LieroMielkeSavare2016,ChizatPeyreSchmitzerVialard2018FoCM,KondratyevMonsaingeonVorotnikov2016}. This structure appears in models of tumor growth and population dynamics. 
\textbf{Unbalanced Optimal Transport} relaxes the mass conservation constraint \cite{LieroMielkeSavare2018,ChizatPeyreSchmitzerVialard2018JFA,ChizatPeyreSchmitzerVialard2018FoCM,KondratyevMonsaingeonVorotnikov2016}, allowing for the modeling of certain systems with growth and decay. This structure appears in models of tumor growth \cite{LieroMielkeSavare2016,PerthameQuirosVazquez2014} and population dynamics \cite{KondratyevMonsaingeonVorotnikov2016}. These gradient flows have also been recently used in statistical learning, to perform inclusive (i.e. reverse) KL minimization for Bayesian inference \cite{zhu2024inclusiveklminimizationwassersteinfisherrao} and learn Gaussian Mixture Models \cite{MR4804827}.

\subsection{Organization of the paper}

The paper is organized as follows. 
In \Cref{sec:submodularity-exchangeability}, we introduce the necessary background on Banach lattices and rigorously define submodularity and substitutability in infinite dimensions. 
We also describe how these notions allow us to derive synthetic comparison principles.
\Cref{sec:duality} establishes the duality between submodularity and substitutability. The section ends with concrete examples of submodular and substitutable functionals. 
Finally, \Cref{sec:cp-exch} presents the applications to optimal transport, detailing the comparison principles for Kantorovich potentials and the JKO scheme.

\section{Submodularity and substitutability}\label{sec:submodularity-exchangeability}

Comparison principles for infinite-dimensional problems involve a diverse range of partially ordered functional spaces, such as $L^p$ spaces, Sobolev spaces, spaces of continuous functions, measure spaces, and so on.
\emph{Banach lattices} turn out to be a convenient framework to unify these different settings: we give some basic definitions, properties and examples in \Cref{sec:banach-lattices}. 
In \Cref{sec:def-submodular} we review the concept of a \emph{submodular function} defined on a general Banach lattice, and in \Cref{sec:def-exchangeable} we extend the notion of \emph{substitutability} recently introduced in~\cite{GalichonHsiehSylvestre2024} to the infinite-dimensional setting.

\subsection{Banach lattices}\label{sec:banach-lattices}

Standard textbook references on Banach lattices include \cite{ZaanenBook1997,SchaeferBook1974,MeyerNiebergBook1991}. In this section we will only introduce some very elementary material which will be used in the rest of the paper.
Recall that a \emph{partial ordering} $\leq$ on a set $X$ is a binary relation which is reflexive ($\phi\leq\phi$), transitive ($\phi_1\leq\phi_2$ and $\phi_2\leq\phi_3$ imply $\phi_1\leq\phi_3$) and antisymmetric ($\phi_1\leq\phi_2$ and $\phi_2\leq\phi_1$ imply $\phi_1=\phi_2$). 
When $(X,\leq)$ is a partially ordered set and $A$ is a nonempty subset of $X$, we say that $\phi_1\in X$ is the \emph{supremum} of $A$ if it is an upper bound of $A$ (i.e.\ $\phi\leq\phi_1$ for all $\phi\in A$) and any other upper bound $\phi_2$ of $A$ satisfies $\phi_1\leq\phi_2$; in other words $\phi_1$ is the least upper bound of $A$. Similarly, $\phi_1$ is the infimum of $A$ if it is the greatest lower bound of $A$.

\begin{definition}[Lattice]
    A partially ordered set $(X,\leq)$ is a \emph{lattice} if for any two $\phi_1,\phi_2\in X$, the subset $\{\phi_1,\phi_2\}$ admits a supremum, denoted by $\phi_1\vee\phi_2$, as well as an infimum, denoted by $\phi_1\wedge\phi_2$.
\end{definition}

\begin{definition}[Banach lattice]
    Let $(X,\lVert\cdot\rVert)$ be a Banach space equipped with a partial ordering $\leq$. Denote $\phi^+=\phi\vee 0$, $\phi^-=-(\phi\wedge 0)$, and $\abs\phi=\phi^+ + \phi^-$. We say that $X=(X,\lVert\cdot\rVert,\leq)$ is a \emph{Banach lattice} if 
    \begin{enumerate}[(i)]
        \item $(X,\leq)$ is a lattice;
        
        \item the ordering is compatible with the vector space operations in the sense that $\phi_1\leq\phi_2$ implies $\phi_1+\phi_3\leq\phi_2+\phi_3$, and $\phi_1\leq\phi_2$ implies $\lambda \phi_1\leq\lambda\phi_2$, for all $\phi_1,\phi_2,\phi_3\in X$ and all $\lambda\in\R$, $\lambda\geq 0$; 
        
        \item the norm is compatible with the ordering in the sense that $\abs{\phi_1} \leq \abs{\phi_2} $ implies $\Vert \phi_1 \Vert \leq \Vert \phi_2 \Vert$ for all $\phi_1,\phi_2\in X$.
    \end{enumerate}
\end{definition}

Let us now state some elementary properties of lattice operations which will be used throughout this section.

\begin{lemma}\label{lemma:lattice-basic-prop}
    Let $X$ be a Banach lattice, and let $\phi_1,\phi_2\in X$. Then 
    \begin{enumerate}[(i)]
        \item $\phi_1+(\phi_2\vee\phi_3)=(\phi_1+\phi_2)\vee(\phi_1+\phi_3)$;
        \label{item:lemma:lattice-basic-prop-sum}
        
        \item $\phi_1\vee\phi_2+\phi_1\wedge\phi_2=\phi_1+\phi_2$;
        
        \item $\phi_1\vee\phi_2-\phi_1\wedge\phi_2=\abs{\phi_1-\phi_2}$;
        
        \item $\phi_1\leq\phi_2$ if and only if $\phi_1^+\leq\phi_2^+$ and $\phi_1^-\geq\phi_2^-$;
        \label{item:lemma:lattice-basic-prop-sum-pp-np}
        
        \item $-\phi_2\leq\phi_1\leq \phi_2$, $\phi_2\geq 0$ implies $\abs{\phi_1}\leq \phi_2$.\label{item:lemma:lattice-basic-prop-tv-bound}
    \end{enumerate}
\end{lemma}
\begin{proof}
    See \cite[Chapter 2]{ZaanenBook1997}, or \cite[Chapter II]{SchaeferBook1974}.
\end{proof}

Let us now give examples of Banach lattices.

\begin{example}
    Consider the Banach space $C(K)$ of real-valued continuous functions on a compact Hausdorff space $K$ equipped with the supremum norm $\norm{\phi}=\sup_{x\in K}\abs{\phi(x)}$. Consider the pointwise ordering on $C(K)$ defined by $\phi_1\leq \phi_2$ if $\phi_1(x)\leq\phi_2(x)$ for all $x\in K$. Then $(X,\norm{\cdot},\leq)$ is a Banach lattice, see e.g.\ \cite[Section 1.1]{MeyerNiebergBook1991} or \cite[Example 15.5(i)]{ZaanenBook1997}. The $\wedge$ and $\vee$ operators are given by the pointwise min and max, 
    \begin{align*}        
        (\phi_1\wedge\phi_2)(x)&=\min(\phi_1(x),\phi_2(x)),\\
    (\phi_1\vee\phi_2)(x)&=\max(\phi_1(x),\phi_2(x)).\\
    \end{align*}
\end{example}

\begin{example}\label{ex:lp-spaces-banach-lattice}
    Consider the Banach space $L^p(\R^n)$, $1\leq p\leq +\infty$, of all measurable functions $\phi\colon \R^n\to\R$ such that $\int_{\R^n} \abs{\phi(x)}^p \,dx < +\infty$ (or $\operatorname{esssup}_{x\in \R^n} \abs{\phi(x)} < +\infty$ if $p=+\infty$, where $\operatorname{esssup}$ denotes the essential supremum) equipped with the norm $\norm{\phi}_p=\left(\int_{\R^n} \abs{\phi(x)}^p dx\right)^{1/p}$ (or $\norm{\phi}_\infty=\operatorname{esssup}_{x\in \R^n} \abs{\phi(x)}$ if $p=+\infty$). Consider the pointwise ordering defined by $\phi_1\leq \phi_2$ if $\phi_1(x)\leq \phi_2(x)$ for almost every $x\in \R^n$. Then $(L^p(\R^n),\norm{\cdot}_p,\leq)$ is a Banach lattice, see~\cite[Example 15.5(ii)]{ZaanenBook1997} or \cite[Section 1.1]{MeyerNiebergBook1991}. The $\wedge$ and $\vee$ operators are given by the almost everywhere pointwise min and max.

    More generally, the Banach spaces $L^p(\mu)$ where $\mu$ is a positive measure on, say, a Polish space $\Omega$, are Banach lattices.
\end{example}

Another important class of Banach lattices is spaces of signed measures. Below we start by recalling some standard facts about signed measures, and point to e.g.\ \cite{BogachevBook2007} for a standard reference on the subject. 

\begin{example}\label{ex:signed-measures}

Let $(S,\Sigma)$ be a measurable space. A \emph{finite signed measure} on $(S,\Sigma)$ is a countably additive function $\mu\colon\Sigma\to\R$ such that $\mu(\emptyset)=0$; since in this paper we do not consider measures with infinite values we will often simply write signed measure instead of finite signed measures. A \emph{positive} measure is a signed measure taking only nonnegative values. The \emph{restriction} of a signed measure $\mu$ to a subset $B\in\Sigma$ is the signed measure $\mu|_B$ on $(S,\Sigma)$ defined by $\mu|_B(A)=\mu(A\cap B)$ for all $A\in\Sigma$. $\mu$ is said to be \emph{concentrated on} $B\in\Sigma$ if $\mu=\mu|_B$. Two signed measures $\mu,\nu$ are \emph{mutually singular} if they can be concentrated on disjoint sets.

Let $\mu$ be a signed measure. 
The \textbf{Hahn decomposition theorem} says that there exists a partition $S = S_+ \cup S_-$ such that $\mu|_{S_+} \geq 0$ and $\mu|_{S_-} \leq 0$, $S_+,S_-\in\Sigma$ .
The \textbf{Jordan decomposition theorem} says that there exist unique mutually singular positive measures $\mu^+$ and $\mu^-$ such that $\mu=\mu^+-\mu^-$. These can be taken as $\mu^+=\mu|_{S_+}$ and $\mu^-=-\mu|_{S_-}$, or equivalently defined as $\mu^+(B)=\sup\{\mu(A) : A\subset B, A\in \Sigma\}$ and $-\mu^-(B)=\inf\{\mu(A) : A\subset B, A\in \Sigma\}$.
The total variation of a signed measure $\mu$ is the positive measure defined by $\abs\mu=\mu^++\mu^-$. The variation norm of $\mu$ is the finite real number $\norm\mu=\abs\mu(S)$.

The set $\Mc(S)$ of finite signed measures over $(S,\Sigma)$ can be naturally given the structure of a real vector space. Endowed with the total variation norm, it is a Banach space. The zero element is the \emph{null measure}, denoted by $0$, which assigns the zero value to every measurable subset of $S$. Consider the partial order $\leq$ on $\Mc(S)$ defined by 
\begin{equation}\label{eq:order-measures}
    \mu\leq\nu\quad\text{if } \mu(A)\leq \nu(A) \text{ for every } A\in\Sigma.
\end{equation}
Then equipped with $\leq$, $\Mc(S)$ is a Banach lattice \cite[Section 1.1]{MeyerNiebergBook1991}. The $\wedge$ and $\vee$ operators are given by 
\begin{equation}\label{eq:formula-vee-wedge-measures}
    (\mu \vee \nu) (A) = \sup_{A=B\cup C}  \mu(B) + \nu(C), \quad (\mu \wedge \nu) (A) = \inf_{A=B\cup C}  \mu(B) + \nu(C),
\end{equation}
where the supremum and infimum above run over $B,C\in\Sigma$ satisfying $B\cup C=A$ and $B\cap C=\emptyset$.
The positive and negative parts $\mu^+$ and $\mu^-$ defined above then coincide with the Banach lattice objects $\mu^+=\mu\vee 0$ and $\mu^-=-(\mu\wedge 0)$. 
\end{example}

The examples of continuous functions and signed measures described above can be connected by duality. Indeed, when $X$ is a Banach lattice, its dual space $X^*$ (the set of continuous linear forms) can be naturally equipped with the \emph{dual order}, still denoted by $\leq$ and defined for $\mu,\nu\in X^*$ by 
\begin{equation}\label{eq:dual-order}
    \mu\leq\nu \quad\text{if} \,\bracket{\mu,\phi}\leq \bracket{\nu,\phi} \text{ for any } \phi\in X, \phi\geq 0.
\end{equation}
Recall also that the dual norm is given by $\norm{\mu}\coloneqq \sup\{\bracket{\mu,\phi} : \phi\in X,\norm{\phi}\leq 1\}$. Then:

\begin{proposition}[{\cite[Prop. 1.3.7]{MeyerNiebergBook1991},\cite[Chap. 11, Prop 5.5.]{SchaeferBook1974}}]
    Equipped with the dual norm and the dual order, $X^*$ is a Banach lattice.
\end{proposition}

To relate continuous functions and signed measures using the above result, let $\Omega$ be a compact metric space, and denote by $C(\Omega)$ the Banach lattice of continuous functions over $\Omega$ and by $\Mc(\Omega)$ the Banach lattice of finite signed Borel measures over $\Omega$. By the Riesz representation theorem, $\Mc(\Omega)$ can be identified with the dual Banach space of $C(\Omega)$, via the mapping of $\mu\in \Mc(\Omega)$ to $(\phi\mapsto\int_\Omega \phi\,d\mu)\in C(\Omega)^*$ \cite[Theorem 7.10.4.]{BogachevBook2007}. In fact, we have 

\begin{theorem}
    Let $\Omega$ be a compact metric space. Then $\Mc(\Omega)$ and $C(\Omega)^*$ can be identified as Banach lattices.
\end{theorem}
\begin{proof}
    Since $\Mc(\Omega)$ and $C(\Omega)^*$ can be identified as Banach spaces by the Riesz representation theorem, all is left to show is that the two orderings \eqref{eq:order-measures} and \eqref{eq:dual-order} respectively defined on $\Mc(\Omega)$ and $C(\Omega)^*$ coincide. By linearity it is sufficient to show that a signed measure $\mu\in\Mc(\Omega)$ is positive if and only if $\int_\Omega \phi\,d\mu \geq 0$ for every nonnegative continuous function $\phi$. The ``only if'' part is immediate. The ``if'' part follows from the following considerations. First, since a compact metric space is complete and separable, every Borel measure is Radon \cite[Theorem 7.1.7]{BogachevBook2007}. Second, the evaluation of Radon measures can be approximated by continuous functions \cite[Lemma 7.2.8]{BogachevBook2007}: for any Borel set $A\subset \Omega$ and any $\eps>0$ there exists a continuous function $\phi\colon\Omega\to[0,1]$ such that $\abs{\mu(A)-\int_\Omega\phi\,d\mu}<\eps$. This implies the bound $\mu(A)>-\eps$ for all $\eps>0$, thus the desired result. 
\end{proof}

\subsection{Submodularity} \label{sec:def-submodular}

Throughout the text, $X=(X,\norm{\cdot},\leq)$ denotes a Banach lattice and $\wedge$ and $\vee$ denote the corresponding inf and sup operations. 

\begin{definition}
A function $E\colon X\to\R\cup\{+\infty\}$ is \emph{submodular} if for every $\phi_1,\phi_2 \in X$ we have
\begin{equation}\label{eq:def-submodular}
     E(\phi_1 \wedge \phi_2)+E(\phi_1 \vee \phi_2) \leq E(\phi_1)+E(\phi_2).
\end{equation}

\end{definition}

When minimizing a submodular function, we therefore see that $\phi_1 \wedge \phi_2$ and $\phi_1 \vee \phi_2$ provide competitors to $\phi_1$ and $\phi_2$ which produce a lower total cost. Note also that submodularity does not depend on the norm used on $X$. We shall refer to \eqref{eq:def-submodular} as the \emph{submodularity inequality}.

For sufficiently smooth functions in finite dimensions, submodularity can be characterized by a differential condition. The next result is set on the standard lattice $\R^n$ with the coordinate order : $x\leq y$ if and only if $x_i\leq y_i$ for every $i=1,\dots, n$. This is a lattice order for which the lattice operations are given by the coordinate-wise minimum and maximum:  
\begin{equation*}
    (x \wedge y)_i = \min(x_i,y_i), \quad (x \vee y)_i = \max(x_i,y_i).
\end{equation*}

\begin{lemma}\label{lemma:cross-derivatives}
    Take $X = (\R^n,\leq )$. Then a function $E\in C^2(\R^n)$ is submodular if and only if 
    \begin{equation*}
        \frac{\partial^2 E}{\partial x_i\partial x_j}(x)\leq 0\quad\text{for all } i\neq j, \text{ and all }x\in\R^n.
    \end{equation*}
\end{lemma}
\begin{proof}
Assume first that $E$ is submodular. Take $1 \leq i\neq j \leq n$ and denote by $e_i,e_j$ the corresponding elements of the canonical basis of $\R^n$. Given $x \in \mathbb{R}^n$, we have 
\begin{equation}\label{eq:proof-lemma:cross-derivatives}
    \frac{\partial^2 E}{\partial x_i\partial x_j}(x) = \lim_{\eps \to 0^+} \frac{E(x+\eps(e_i+e_j)) + E(x) - E(x+\eps e_i) - E(x+\eps e_j)}{\eps^2}.
\end{equation}
Note that $x+\eps(e_i+e_j)=(x+\eps e_i) \vee (x+\eps e_j)$ and $x=(x+\eps e_i) \wedge (x+\eps e_j)$ when $i \neq j$ and $\eps>0$. By submodularity of $E$, we deduce nonpositivity of the right-hand side of \eqref{eq:proof-lemma:cross-derivatives}.

Assume now that $\frac{\partial^2 E}{\partial x_i\partial x_j}(x) \leq 0$ for all $i \neq j$ and $x\in \R^n$. Let $a,b \in \R^n$ be such that $a\leq b$ and $a_i = b_i$, then
\begin{equation*}
    \frac{\partial E}{\partial x_i}(a) - \frac{\partial E}{\partial x_i}(b) = \int_0^1 \sum_{j\neq i} \frac{\partial^2 E}{\partial x_i\partial x_j}(a+t(b-a))(b_j-a_j) dt \leq 0
\end{equation*}
since $\frac{\partial^2 E}{\partial x_i\partial x_j}(x) \leq 0$ for all $x$, $b_j \geq a_j$ for $j\neq i$ and $a_i = b_i$. Let $x,y \in \R^n$, then 
\begin{align*}
    E(x\vee y) - E(x) &= \int_0^1 \langle \nabla E(x + t(x\vee y -x)), x\vee y - x\rangle dt\\
    &= \int_0^1 \langle \nabla E(x + t(y-x\wedge y)), y-x\wedge y\rangle dt\\
    &\leq \int_0^1 \langle \nabla E(x\wedge y + t(y-x\wedge y)), y-x\wedge y\rangle dt\\
    &= E(y) - E(x\wedge y).
\end{align*}
The second equality holds because $x+y = x\wedge y + x\vee y$, the inequality holds by the claim above because $x\wedge y + t(y-x\wedge y) \leq x + t(y-x\wedge y)$ and $(x\wedge y + t(y-x\wedge y))_i = (x + t(y-x\wedge y))_i$ when $(y - x\wedge y)_i > 0$. Rearranging the terms exactly gives submodularity of $E$.
\end{proof}

A prototypical example of a submodular function in $\R^n$ is then given by a quadratic function with negative cross-derivatives.

\begin{example}
    Let $A \in \R^{n\times n}$ be a symmetric matrix with non-positive off-diagonal entries: $A_{ij} \leq 0$ for every $i \neq j$. Then the function $E\colon (x\in \R^n) \mapsto \frac{1}{2} x^T A x$ is submodular. This is a direct consequence of \cref{lemma:cross-derivatives}.
\end{example}

It turns out to be useful to break the symmetry between $\phi_1$ and $\phi_2$ in \eqref{eq:def-submodular} to define a relation between functions defined on $X$.

\begin{definition}\label{def:P-order}
    Consider two functions $E_1,E_2\colon X\to\R\cup\{+\infty\}$. We say that $E_1$ is P-dominated by $E_2$, and write $E_1\leP E_2$, if for every $\phi_1,\phi_2 \in X$ we have
    \begin{equation}\label{eq:def-P-order}
         E_1(\phi_1 \wedge \phi_2)+E_2(\phi_1 \vee \phi_2) \leq E_1(\phi_1)+E_2(\phi_2).
    \end{equation}
\end{definition}

We use the notation $E_1\leP E_2$  because P-dominance does not define a partial order in general.  
When specializing \eqref{eq:def-P-order} to indicator functions of sets (i.e. functions only taking the values $0$ and $+\infty$) we obtain a relation on sets.

\begin{definition}\label{def:P-order-sets}
    Consider two sets $A_1,A_2 \subset X$. We say that $A_1$ is P-dominated by $A_2$, and write $A_1\leP A_2$, if for every $\phi_1 \in A_1, \phi_2 \in A_2$ we have $\phi_1\wedge \phi_2 \in A_1$ and $\phi_1 \vee \phi_2 \in A_2$. 
\end{definition}

This notion is also known as the strong set order, or the Veinott order \cite[Section 2]{Topkis1978}. An elementary example of P-dominance which will be useful for our applications involves the duality bracket, and in some sense reverses inequalities. Recall that the dual Banach space $X^*$ is equipped with its natural Banach lattice structure, see~\eqref{eq:dual-order}.

\begin{lemma}\label{lemma:duality-bracket-P-dominance}
    Let $\mu_1,\mu_2 \in X^\ast$ be such that $\mu_1 \leq \mu_2$. Then $\langle \mu_2, \cdot \rangle$ is P-dominated by $\langle \mu_1, \cdot \rangle$, i.e.\ $\langle \mu_2, \cdot \rangle \leP \langle \mu_1, \cdot \rangle$.
\end{lemma}

\begin{proof}
    Let $\phi_1,\phi_2 \in X$. By definition of the dual order we have 
    \begin{equation*}
        \langle \mu_1, (\phi_1\vee\phi_2)  - \phi_1\rangle  \leq \langle \mu_2, (\phi_1\vee\phi_2)  - \phi_1\rangle,
    \end{equation*}
    since $(\phi_1\vee\phi_2) - \phi_1 \geq 0$. Using that  $(\phi_1\vee\phi_2) - \phi_1 = \phi_2 - (\phi_1\wedge\phi_2) $ the above inequality can be rewritten as
    \begin{equation*}
        \langle \mu_1, (\phi_1\vee\phi_2)  - \phi_1\rangle  \leq \langle \mu_2, \phi_2 - (\phi_1\wedge\phi_2) \rangle.
    \end{equation*}
    Finally, rearranging terms gives $\langle \mu_2, \cdot \rangle \leP \langle \mu_1, \cdot \rangle$.
\end{proof}

Another elementary observation is that certain monotonicity properties can be expressed in terms of P-dominance. Recall that a function $E$ is \emph{order-preserving} if $E(\phi_1)\leq E(\phi_2)$ whenever $\phi_1\leq\phi_2$. 

\begin{lemma}
    A function $E\colon X\to\R\cup\{+\infty\}$ is order-preserving if and only if $E \leP 0$, where $0$ denotes the identically zero function. 
\end{lemma}
\begin{proof}
    It is clear that an order-preserving function satisfies $E \leP 0$. Conversely, if $E(\phi_1\wedge\phi_2)\leq E(\phi_1)$, choosing $\phi_2\leq\phi_1$ implies the order-preserving inequality.
\end{proof}

The next result shows that P-dominance is stable under addition. Therefore the various examples described above (submodular functions, duality brackets, order-preserving functions) and in \cref{sec:examples-submodular} below can be easily combined into interesting functionals. 

\begin{lemma}\label{lemma:sum-stable-P}    
    Consider $E_i\colon X\to\R\cup\{+\infty\}$ for $1\leq i\leq 4$. If
    $E_1\leP E_3$ and $E_2\leP E_4$ then $E_1+E_2\leP E_3+E_4$. In particular, if $E_1$ and $E_2$ are submodular then $E_1+E_2$ is submodular.
\end{lemma}
\begin{proof}
    This is a direct consequence of \Cref{def:P-order} applied to the pairs $(E_1,E_3)$ and $(E_2,E_4)$.
\end{proof}
By definition, submodularity is a particular case of P-dominance. Conversely, P-dominance may arise in the presence of joint submodularity. Here, joint submodularity always refers to the product lattice structure. When $X$ and $Y$ are Banach lattices, we always endow $X\times Y$ with the natural Banach lattice structure given by the usual product norm and the following product order: given $(\phi_1,\psi_1)$ and $(\phi_2,\psi_2)\in X\times Y$, $(\phi_1,\psi_1)\leq (\phi_2,\psi_2)\iff \phi_1\leq\phi_2$ and $\psi_1\leq\psi_2$. Then $(\phi_1,\psi_1)\wedge(\phi_2,\psi_2)=(\phi_1\wedge\phi_2, \psi_1\wedge\psi_2)$ and $(\phi_1,\psi_1)\vee(\phi_2,\psi_2)=(\phi_1\vee\phi_2, \psi_1\vee\psi_2)$.

\begin{lemma}\label{lemma:P-order-joint-subm}
    Let $X$ and $Y$ be two Banach lattices and consider a function $E\colon X\times Y\to\R\cup\{+\infty\}$ which is jointly submodular. Suppose that $\phi_1,\phi_2\in X$ are such that $\phi_1\leq\phi_2$. Then $E(\phi_1,\cdot)\leP E(\phi_2,\cdot)$. 
\end{lemma}
\begin{proof}
    Let $\phi_1, \phi_2\in X$ and $\psi_1,\psi_2\in Y$. By joint submodularity of $E$ we have 
    \begin{equation}\label{eq:proof-lemma:P-order-joint-subm-1}
    E(\phi_1\wedge\phi_2,\psi_1\wedge\psi_2)+
    E(\phi_1\vee\phi_2,\psi_1\vee\psi_2)\leq
    E(\phi_1,\psi_1)+E(\phi_2,\psi_2).
    \end{equation}
    Now assume that $\phi_1\leq\phi_2$. Then $\phi_1\wedge\phi_2=\phi_1$ and $\phi_1\vee\phi_2=\phi_2$, so that \eqref{eq:proof-lemma:P-order-joint-subm-1} turns into the desired inequality.
\end{proof}

P-dominance is also preserved by certain minimization procedures.

\begin{lemma}\label{lemma:P-order-min}
    Let $X$ and $Y$ be two Banach lattices and consider two functions $F_1,F_2\colon X\times Y\to \R\cup\{+\infty\}$ bounded from below and jointly submodular. Define for $i=1,2$ the functions $E_i\colon X\to\R\cup\{+\infty\}$ by 
    \[
    E_i(\phi)=\inf_{\psi\in Y} F_i(\phi,\psi).
    \]
    Suppose that $F_1\leP F_2$. Then $E_1\leP E_2$.
    
    As a direct consequence, if $F_1$ is jointly submodular then the marginal function $E_1$ is also submodular.
\end{lemma}

\begin{proof}
    Let $\phi_1,\phi_2\in X$ and $\psi_1,\psi_2\in Y$. Since $F_1\leP F_2$ we have 
    \begin{equation}\label{eq:lemma:P-order-min-1}
        F_1(\phi_1\wedge\phi_2,\psi_1\wedge\psi_2)+
    F_2(\phi_1\vee\phi_2,\psi_1\vee\psi_2)\leq
    F_1(\phi_1,\psi_1)+F_2(\phi_2,\psi_2).
    \end{equation}
    The left-hand side of \eqref{eq:lemma:P-order-min-1} can be bounded below by the respective infima of $F_1$ and $F_2$ over $Y$, giving 
    \[
    E_1(\phi_1\wedge\phi_2)+E_2(\phi_1\vee\phi_2)\leq F_1(\phi_1,\psi_1)+F_2(\phi_2,\psi_2).
    \]
    Since $\psi_1$ and $\psi_2$ are arbitrary this in turn implies $E_1(\phi_1\wedge\phi_2)+E_2(\phi_1\vee\phi_2)\leq E_1(\phi_1)+E_2(\phi_2)$.
\end{proof}

The next two results form the backbone of our approach to the comparison principles based on submodular energies. The first one transfers P-dominance from functions to their minimizers.

\begin{proposition}\label{prop:topkis}
    Let $E_1,E_2\colon X\to\R\cup\{+\infty\}$. If $E_1 \leP E_2$ then $\argmin E_1 \leP \argmin E_2$. 
\end{proposition} 

\begin{proof}
    If either of $ \argmin E_1$, $ \argmin E_2$ is empty the result holds vacuously. Otherwise, take $\phi_1 \in \argmin E_1$ and $\phi_2 \in \argmin E_2$. Since $E_1\leP E_2$ we have
    \begin{equation*}
        E_1(\phi_1 \wedge \phi_2) + E_2(\phi_1 \vee \phi_2) \leq E_1(\phi_1) + E_2(\phi_2).
    \end{equation*}
    Since $E_1(\phi_1) \leq E_1(\phi_1 \wedge \phi_2)$ and $E_2(\phi_2) \leq E_2(\phi_1 \vee \phi_2)$ we find that  $\phi_1 \wedge \phi_2 \in \argmin E_1$ and $\phi_1 \vee \phi_2 \in \argmin E_2$.
\end{proof}

This comparison between optimizers of energies is an extension of the celebrated Topkis' theorem in economics \cite{Topkis1978}. 
Finally, the next result says that P-dominance reduces to the order on $X$ for singletons.

\begin{lemma}\label{lemma:P-singletons}
    Let $\phi_1,\phi_2 \in X$. Then $\{\phi_1\} \leP \{\phi_2\}$ holds if and only if $\phi_1 \leq \phi_2$.
\end{lemma}

\begin{proof}
    From $\{\phi_1\} \leP \{\phi_2\}$ we deduce that $\phi_1 \wedge \phi_2 \in \{\phi_1\}$, or equivalently $\phi_1 \leq \phi_2$. The converse holds trivially since $\phi_1 \leq \phi_2$ implies that $\phi_1 \wedge \phi_2 = \phi_1$ and $\phi_1 \vee \phi_2 = \phi_2$.
\end{proof}

\subsubsection{Examples of submodular functions}\label{sec:examples-submodular}

In this section we provide several examples of submodular functions. First, we recall that a function $E$ valued in $\R \cup \{+\infty\}$ is \emph{proper} if it is not identically $+\infty$, and its domain $\dom E$ is the set of points where it is finite.
We start with ``elementary functions'' which can be combined into more complicated energies. 
One such important basic building block is given by the following example.

\begin{lemma}\label{lemma:convexdiff_submodular}
    Let $f \colon \R \to \R \cup \{+\infty\}$ be a proper convex function. Then the function 
    \begin{align*}
            E \colon \R^2 &\to \R\cup \{+\infty\} \\
            (a,b) &\mapsto f(a-b)
    \end{align*}
    is proper convex and submodular.
\end{lemma}
\begin{proof}
    $E$ is clearly a proper convex function over $\R^2$. Given $(a_1,b_1), (a_2, b_2)\in\dom E$ we then want to show that 
    \begin{equation}\label{eq:proof-lemma:convexdiff_submodular-1}
        E(a_1\wedge a_2,b_1\wedge b_2)+E(a_1\vee a_2,b_1\vee b_2) \leq E(a_1,b_1)+E(a_2,b_2).
    \end{equation}
    First, we suppose without loss of generality that $a_1\leq a_2$. Second, note that if $b_1\leq b_2$ then \eqref{eq:proof-lemma:convexdiff_submodular-1} is immediate. We therefore suppose that $b_2\leq b_1$. As a direct consequence, we have the inequality 
    \begin{equation*}
        a_1-b_1\leq (a_1\wedge a_2)-(b_1\wedge b_2)\leq a_2-b_2.
    \end{equation*}
    We deduce the existence of some $\lambda\in [0,1]$ such that $(a_1\wedge a_2)-(b_1\wedge b_2)=(1-\lambda)(a_1-b_1) + \lambda (a_2-b_2)$. By convexity of $f$ we have 
    \begin{equation*}
        f((a_1\wedge a_2)-(b_1\wedge b_2))\leq (1-\lambda) f(a_1-b_1) + \lambda f(a_2-b_2).
    \end{equation*}
    This is the same as 
    \begin{equation}\label{eq:proof-lemma:convexdiff_submodular-2}
        E(a_1\wedge a_2,b_1\wedge b_2)\leq (1-\lambda) E(a_1,b_1) + \lambda E(a_2,b_2).
    \end{equation}
    To conclude, note that $a_1\wedge a_2 + a_1\vee a_2 = a_1+a_2$ and $b_1\wedge b_2 + b_1\vee b_2 = b_1+b_2$ so that $(a_1\vee a_2)-(b_1\vee b_2)=\lambda (a_1-b_1) + (1-\lambda)(a_2-b_2)$. By convexity of $f$ we find 
    \begin{equation}\label{eq:proof-lemma:convexdiff_submodular-3}
        E(a_1\vee a_2,b_1\vee b_2)\leq \lambda E(a_1,b_1) + (1-\lambda) E(a_2,b_2).
    \end{equation}
    Combining \eqref{eq:proof-lemma:convexdiff_submodular-2} and \eqref{eq:proof-lemma:convexdiff_submodular-3} gives \eqref{eq:proof-lemma:convexdiff_submodular-1}. 
\end{proof}

\begin{remark}
In the case where $f\in C^2(\R)$, the submodularity of $F$ directly follows from \cref{lemma:cross-derivatives} since 
\[
\frac{\partial^2 E}{\partial a \,\partial b}=-f''(a-b)\leq 0.
\]
\end{remark}

Another source of submodularity comes from concave functions of a sum. We state this general principle in the following form.

\begin{lemma}\label{lemma:concave-submodular}
    Consider a Polish space $\Omega$ equipped with a positive measure $\mu$ and let $X=L^1(\mu)$. Let $f\colon \R\to\R\cup\{-\infty\}$ be a concave function and define $E\colon (\phi\in L^1(\mu))\mapsto f\big(\int_\Omega \phi\,d\mu\big)$. Then $E$ is submodular.
\end{lemma}

\begin{proof}
    Fix $\phi_1, \phi_2\in L^1(\mu)$. By \cref{ex:lp-spaces-banach-lattice}, $L^1(\mu)$ is a Banach lattice and therefore $\phi_1\wedge\phi_2,\phi_1\vee\phi_2\in L^1(\mu)$. Since $\phi_1\wedge\phi_2\leq\phi_1\leq \phi_1 \vee\phi_2$ we have $\int_\Omega (\phi_1\wedge\phi_2)\,d\mu \leq \int_\Omega \phi_1\,d\mu \leq \int_\Omega (\phi_1\vee\phi_2)\,d\mu$, and there must exist some $\lambda\in [0,1]$ such that 
    \[
        \int_\Omega \phi_1\,d\mu = (1-\lambda) \int_\Omega (\phi_1\wedge\phi_2)\,d\mu +\lambda \int_\Omega (\phi_1\vee\phi_2)\,d\mu.
    \]
    By concavity of $f$ we deduce 
    \begin{equation}\label{eq:proof-lemma:concave-submod}
        f\left(\int_\Omega \phi_1\,d\mu\right) \geq (1-\lambda) f\left(\int_\Omega (\phi_1\wedge\phi_2)\,d\mu\right) + \lambda f\left(\int_\Omega (\phi_1\vee\phi_2)\,d\mu \right).
    \end{equation}
    This inequality may contain infinities, in which case the convention $0\cdot(-\infty)=0$ should be adopted \cite[Section 4]{Rockafellar1970}.
    Note that $\int_\Omega (\phi_1\wedge \phi_2)\,d\mu + \int_\Omega (\phi_1\vee\phi_2)\,d\mu$ is equal to $\int_\Omega \phi_1\,d\mu+\int_\Omega \phi_2\, d\mu$, thus
    \begin{equation*}
        \int_\Omega \phi_2\,d\mu = \lambda \int_\Omega (\phi_1\wedge\phi_2)\,d\mu +(1-\lambda) \int_\Omega (\phi_1\vee\phi_2)\,d\mu.
    \end{equation*}
    Using the concavity of $f$ grants once again an inequality similar to \Cref{eq:proof-lemma:concave-submod}, exchanging the roles of $\phi_1$ and $\phi_2$ and changing $\lambda$ into $1-\lambda$. Finally by summation we have
    \begin{equation*}
        f\left(\int_\Omega \phi_1\,d\mu\right) + f\left(\int_\Omega \phi_2\, d\mu\right)\geq f\left(\int_\Omega (\phi_1\wedge\phi_2)\,d\mu\right) + f\left(\int_\Omega (\phi_1\vee\phi_2)\,d\mu \right).
    \end{equation*}

\end{proof}
\begin{example}
    As a direct consequence of \Cref{lemma:concave-submodular}, we obtain the submodularity of the log-laplace transform  $ (\phi \in C(\Omega)) \mapsto \log\left(\int e^\phi d\mu \right)$, where $\Omega$ is a Polish space and $\mu$ is a positive measure over $\Omega$. 
    Indeed since the exponential is strictly increasing, it commutes with all the lattice operations. 
    % The positivity of the exponential ensures that the submodularity inequality is empty if one of the integrals diverge.
\end{example}

In the next example, the set $X=\R^I$ for an arbitrary set $I$ is merely a lattice rather than a Banach lattice since it is not endowed with any topology. As alluded to earlier, this poses no problem to consider submodularity which only depends on the lattice operations. Here the ordering on $\R^I$ is the natural pointwise ordering, and the lattice operations $\wedge$ and $\vee$ are given respectively by the pointwise minimum and maximum.

\begin{lemma}\label{lemma:max-submodular}
    Let $I$ be an arbitrary index set and consider the function
    \begin{align*}
        E\colon \R^I &\to \R\cup\{+\infty\}\\
        (a_i)_{i\in I} &\mapsto \sup_{i\in I} a_i.
    \end{align*}
    Then $E$ is submodular.
\end{lemma}

\begin{proof}
    Let $(a_i)_{i\in I},(b_i)_{i\in I}\in \R^I$. From $\min(a_i,b_i)\leq a_i$ we find $\sup_{i\in I}\min(a_i,b_i) \leq \sup_{i\in I} a_i$, i.e.\ $E(a\wedge b)\leq E(a)$. Exchanging $a$ and $b$ gives $E(a\wedge b)\leq E(b)$, and therefore 
    \begin{equation}\label{eq:proof-lemma:max-submodular-1}
        E(a\wedge b)\leq \min(E(a),E(b)). 
    \end{equation}
    On the other hand,
    \begin{equation}\label{eq:proof-lemma:max-submodular-2}
        E(a\vee b)=\sup_{i\in I}\max(a_i,b_i)=\max(\sup_{i\in I}a_i, \sup_{i\in I} b_i)=\max(E(a),E(b)).
    \end{equation}
    Summing \eqref{eq:proof-lemma:max-submodular-1} and \eqref{eq:proof-lemma:max-submodular-2} gives the desired inequality.
\end{proof}

The examples presented until now are fundamental and can be used to show submodularity of functionals in a wide range of functional spaces. We will now focus on two examples in the Sobolev spaces $H^1(U)$ and $W^{s,p}(U)$ for $s \in (0,1)$, $p\in (1,\infty)$, where $U$ is an open bounded subset of $\R^n$ with regular boundary. The first such example is the Dirichlet energy, one of the simplest examples of a submodular functional on Sobolev spaces. 

\begin{example}\label{ex:dirichlet}
This example takes place over $H^1(U)$, which is not a Banach lattice due to the lack of compatibility between the order and the norm. However, $H^1(U)$ is a sublattice of the Banach lattice $L^2(U)$, in the sense that the inf and sup of two $H^1$ functions are still in $H^1(U)$ \cite[Theorem 2.1.11]{Ziemer}. For $\phi \in H^1(U)$, define the Dirichlet energy of $\phi$ as 
\begin{equation*}
    E(\phi) = \int_U \abs{\nabla \phi}^2.
\end{equation*}
Then $E$ is submodular. Indeed, let $\phi_1,\phi_2\in H^1(U)$. Then the weak derivatives of $\phi_1\wedge\phi_2$ and $\phi_1\vee\phi_2$ are given by \cite[Corollary 2.1.8]{Ziemer}
\begin{equation*}
    \nabla(\phi_1\wedge\phi_2)(x)=\begin{cases}
        \nabla\phi_1(x)&\quad\text{if } \phi_1(x)\leq\phi_2(x),\\
        \nabla\phi_2(x)&\quad\text{if } \phi_1(x)>\phi_2(x),     
    \end{cases}
\end{equation*}
\begin{equation*}
    \nabla(\phi_1\vee\phi_2)(x)=\begin{cases}
        \nabla\phi_2(x)&\quad\text{if } \phi_1(x)\leq\phi_2(x),\\
        \nabla\phi_1(x)&\quad\text{if } \phi_1(x)>\phi_2(x).    
    \end{cases}
\end{equation*}
The above formulas hold for a.e. $x\in U$. Therefore
\begin{align*}
    E(\phi_1 \wedge \phi_2) + E(\phi_1 \vee \phi_2) &= \int_U \abs{\nabla (\phi_1\wedge\phi_2)}^2 + \abs{\nabla (\phi_1\vee\phi_2)}^2\\
    & = \int_{\{\phi_1 \leq \phi_2\}} \abs{\nabla \phi_1}^2 + \abs{\nabla \phi_2}^2 + \int_{\{\phi_1 > \phi_2\}} \abs{\nabla \phi_2}^2 + \abs{\nabla \phi_1}^2 \\
    &= E(\phi_1) + E(\phi_2).
\end{align*}
We deduce that $E$ is submodular. In fact the submodularity inequality is here an equality, which shows both $E$ and $-E$ to be submodular.
\end{example}

The second example is the fractional $p$-Laplacian energy (see \cite{DINEZZA2012521} for more details).

\begin{example}\label{ex:fractional}
    Given $\phi \in W^{s,p}(U)$ with $0<s<1$ and $1<p<+\infty$, the fractional $p$-Laplacian energy is defined as
    \begin{equation*}
        E(\phi) = \int_U \int_U \frac{\abs{\phi(x)-\phi(y)}^p}{\abs{x-y}^{n+sp}}\,dx\,dy.
    \end{equation*}
    $E$ is submodular as a consequence of \cref{lemma:convexdiff_submodular}. Indeed, convexity of $t \mapsto \abs{t}^p$ implies the submodularity of the functions $F_{x,y} \colon W^{s,p}(U) \to \R\cup\{+\infty\}$ defined, for almost every $x,y \in U$ such that $x\neq y$, by $F_{x,y}(\phi) = \frac{\abs{\phi(x)-\phi(y)}^p}{\vert x-y \vert^{n+sp}}$. Since submodularity is preserved by summation (or integration), we deduce that $E$ is submodular.
\end{example}

\subsection{Substitutability}\label{sec:def-exchangeable}

Substitutability naturally arises as the conjugate notion to submodularity in the context of convex duality. 
Recall that the dual space $X^*$ of a Banach lattice $(X,\norm{\cdot},\leq)$ is itself a Banach lattice equipped with the dual norm and the dual order \eqref{eq:dual-order}, which we still denote by $\leq$.
The convex conjugate of a proper function $E\colon X\to\R\cup\{+\infty\}$ is the function $E^*\colon X^*\to\R\cup\{+\infty\}$ defined by $E^*(\mu)=\sup_{\phi\in X}\bracket{\phi,\mu}-E(\phi)$ \cite{EkelandTemamBook}. 
Given a Banach lattice $Y$ and $\mu,\nu\in Y$, define the \emph{order interval} by 
\begin{equation}\label{eq:def-order-interval}
    [\mu,\nu]=\{m\in Y : \mu\leq m\leq \nu\}.
\end{equation}
The definition of a substitutable function is then motivated by the following result, proved in \Cref{sec:duality}.

\begin{theorem}\label{thm:p-order-q-order-dual}
    Let $X$ be a Banach lattice and $E_1,E_2 \colon X \to \R\cup \{+\infty\}$ two proper convex and lower semicontinuous functions. Then $E_1 \leP E_2$ if and only if for every $\mu_1,\mu_2\in X^*$, and every $t_{21}\in [0,(\mu_2-\mu_1)^+]$, there exists $t_{12}\in [0,(\mu_1-\mu_2)^+]$ such that 
    \begin{equation}\label{eq:thm-p-order-q-order-dual}
        E_1^*(\mu_1+t_{21}-t_{12})+E_2^*(\mu_2-t_{21}+t_{12}) \leq E_1^*(\mu_1)+E_2^*(\mu_2).
    \end{equation}
\end{theorem}

We use \eqref{eq:thm-p-order-q-order-dual} to define a binary relation $\leQ$, Q-dominance. 
We define it on arbitrary Banach lattices $Y$, but when used in duality with submodularity, for some of our results and our applications to optimal transport in \Cref{sec:cp-exch} it will exclusively be used on dual Banach lattices $Y=X^*$. 

For the remainder of this section, $X$ and $Y$ denote Banach lattices.

\begin{definition}
    Consider two functions $F_1,F_2\colon Y \to \R\cup\{+\infty\}$. We say that $F_2$ is Q-dominated by $F_1$, and write $F_2\leQ F_1$, if for any $\mu_1,\mu_2 \in Y$ and every $t_{21}\in[0 ,(\mu_2-\mu_1)^+]$, there exists $t_{12}\in[0,(\mu_1-\mu_2)^+]$ such that
    \begin{equation*}
        F_1(\mu_1+t_{21}-t_{12})+F_2(\mu_2-t_{21}+t_{12}) \leq F_1(\mu_1)+ F_2(\mu_2).
    \end{equation*}    
\end{definition}

Similar to the situation between P-dominance and submodularity, a substitutable function is defined as a function Q-dominated by itself.

\begin{definition}\label{def:exchangeability}
    We say that $F\colon Y\to\R\cup\{+\infty\}$ is substitutable if $F\leQ F$.
\end{definition}

It is also useful to consider these notions for sets.

\begin{definition}\label{def:Q-order}
    Consider two sets $B_1,B_2 \subset Y$. We say that $B_2$ is Q-dominated by $B_1$, and write $B_2\leQ B_1$, if for every $\mu_1 \in B_1, \mu_2 \in B_2$ and every $t_{21}\in[0 ,(\mu_2-\mu_1)^+]$, there exists $t_{12}\in[0,(\mu_1-\mu_2)^+]$ such that
    \begin{equation*}
        \mu_1+t_{21}-t_{12} \in B_1, \quad \mu_2-t_{21}+t_{12} \in B_2.
    \end{equation*}    
    We say that $B\subset Y$ is substitutable if $B\leQ B$.
\end{definition}

Here is an economic interpretation of Q-dominance: suppose that you need to produce goods $\omega\in \Omega$ in quantity $\mu$, where $\mu$ is a positive measure over $\Omega$. Suppose moreover that you are using two factories to manufacture them. If $F_i$ denotes the manufacturing cost of factory $i$ ($i=1,2$), then dividing the production into $\mu=\mu_1+\mu_2$ and assigning factory $i$ to manufacture $\mu_i$ results in a total cost $F_1(\mu_1)+F_2(\mu_2)$. Then $F_2\leQ F_1$ says that factory $1$ is in some sense more \emph{efficient} than factory $2$: for every way that factory $2$ might shift some of its output onto factory $1$,
factory $1$ can reassign its own resources and send some  production back to factory $2$
in such a way that the total cost of producing everything goes down. In other words, no matter what “production shock” factory $2$ imposes on factory $1$, factory $1$ can always counter-reallocate so as to reduce overall costs.

As for P-dominance, the duality bracket gives an elementary example of Q-dominance. In the next lemma, the notation $\bracket{\cdot,\phi}$ stands for the evaluation function $\mu\mapsto \bracket{\mu,\phi}$. 

\begin{lemma}\label{lemma:duality-bracket-Q-dominance}
    Let $\phi_1,\phi_2 \in X$ be such that $\phi_1 \leq \phi_2$. Then $\langle \cdot, \phi_2 \rangle$ is Q-dominated by $\langle \cdot, \phi_1 \rangle$ in $X^*$.
\end{lemma}

\begin{proof}
    Let $\mu_1,\mu_2\in X^*$ and $t_{21}\in [0,(\mu_2-\mu_1)^+]$. Set $t_{12} = 0 \in [0,(\mu_1-\mu_2)^+]$. We have 
    \begin{equation*}
        \langle t_{21}-t_{12}, \phi_1-\phi_2 \rangle = \langle t_{21}, \phi_1-\phi_2\rangle \leq 0,
    \end{equation*}
    because $\phi_1\leq\phi_2$ and $t_{21}\geq 0$. Therefore, adding $\langle \mu_1, \phi_1\rangle + \langle \mu_2, \phi_2\rangle$ to both sides of the previous inequality gives
    \begin{equation*}
        \langle \mu_1+t_{21}-t_{12}, \phi_1\rangle + \langle \mu_2-t_{21}+t_{12}, \phi_2\rangle \leq \langle \mu_1, \phi_1\rangle +\langle \mu_2, \phi_2\rangle.
    \end{equation*}
\end{proof}

Our next result shows that substitutability is a stronger notion than convexity, at least on dual Banach lattices. We start with a result on order intervals.

\begin{lemma}\label{lemma:prop-segments-ast}
    Let $\mu,\nu \in X^\ast$ be such that $\mu \leq \nu$ (otherwise $[\mu,\nu] = \emptyset$). Then we have the following two properties:
    \begin{enumerate}[(i)]
        \item The subset $[\mu,\nu]$ is convex and compact for the weak-$\ast$ topology.\label{item:lemma:prop-segments-ast-i}
        \item For any $\phi \in X$, 
        \begin{equation*}
            \max_{m \in [\mu,\nu]} \langle m,\phi \rangle = \langle \nu, \phi^+\rangle - \langle \mu,\phi^-\rangle .
        \end{equation*}        \label{item:lemma:prop-segments-ast-ii}
    \end{enumerate}
\end{lemma}

In (ii) we can write ``max'' instead of ``sup'' since for the weak-$*$ topology, $m\mapsto \bracket{m,\phi}$ is continuous and $[\mu,\nu]$ is compact (by (i)). Recall also that $\phi^+=\phi\vee 0$ and $\phi^-=-(\phi\wedge 0)$.

\begin{proof}[Proof of \cref{lemma:prop-segments-ast}]
We start with the proof of \ref*{item:lemma:prop-segments-ast-i}. Convexity of $[\mu,\nu]$ follows from the compatibility of the ordering with the vector space operations. Moreover, by definition of the ordering on $X^\ast$ and its compatibility with the topology,  $[\mu,\nu]$ is closed. It is also bounded because by \cref{lemma:lattice-basic-prop}\ref*{item:lemma:lattice-basic-prop-sum-pp-np} any $m \in [\mu,\nu]$ satisfies $m^+ \leq \nu^+$ and $m^-\leq \mu^-$, and  thus $\vert m \vert \leq \nu^+ + \mu^-$. Since $X^\ast$ is a Banach lattice we get a bound on the norm, $\Vert m \Vert \leq \Vert \nu^+ + \mu^- \Vert$. Compactness in the weak-$\ast$ topology of $[\mu,\nu]$ then follows from the Banach--Alaoglu--Bourbaki theorem.

To prove \ref*{item:lemma:prop-segments-ast-ii}, take $\phi \in X$. The upper bound on the maximum comes directly from the definition of the dual order. It remains to show the existence of an element of $[\mu,\nu]$ which attains the upper bound. If $\phi \geq 0$ then $\nu$ attains the upper bound. Similarly if $\phi \leq 0$ then $\mu$ attains the upper bound. Now assume that $\phi^+,\phi^- \neq 0$, then $(\phi^+,\phi^-)$ is a linearly independent family and we can define the linear operator $m_0\colon \alpha \phi^+ + \beta \phi^- \mapsto \alpha \langle \nu, \phi^+\rangle + \beta \langle \mu,\phi^-\rangle$. It satisfies $\langle m_0,f\rangle \leq \langle \nu, f^+\rangle - \langle \mu,f^-\rangle$ for all $f \in \mathrm{Vect}(\phi^+,\phi^-)$, and note that the right-hand side defines a subadditive operator over all of $X$. Thus by the Hahn--Banach theorem there exists $m \in X^\ast$ which extends $m_0$ and satisfies $\langle m,f\rangle \leq \langle \nu, f^+\rangle - \langle \mu,f^-\rangle$. In particular $m \in [\mu,\nu]$ and it attains the upper bound by definition of $m_0$.
\end{proof}
We will also need the lower semicontinuity of the positive part for the weak-$\ast$ topology on $X^\ast$.
\begin{lemma}\label{lem:lsc-positive-part}
    Let $(\mu_n)$ be a sequence in $X^\ast$ converging weak-$\ast$ to $\mu$. Then for every $\phi \in X$ such that $\phi \geq 0$, we have
    \begin{equation*}
        \liminf_{n\to +\infty} \,\langle (\mu_n)^+,\phi\rangle \geq \langle \mu^+,\phi\rangle.
    \end{equation*}
\end{lemma}
\begin{proof}
    Given $m\in X^\ast$ and $\phi \in X$ such that $\phi \geq 0$, \cite[Chapter II \textsection 4.2,Corollary 1]{SchaeferBook1974} combined with \cite[Chapter II, Prop. 5.5]{SchaeferBook1974} ensure that $\langle m^+,\phi\rangle = \sup_{\psi \in [0,\phi]} \langle m,\psi\rangle$. Thus $m\mapsto \langle m^+,\phi\rangle$ is lower semicontinuous for the weak-$\ast$ topology as a supremum of such functions.
\end{proof}

\begin{proposition}
    Let $X$ be a Banach lattice and $F\colon X^\ast \to \mathbb{R}\cup\{+\infty\}$ be a proper, weak-$*$ lower semicontinuous function. If $F$ is substitutable then it is convex. 
\end{proposition}

\begin{proof}
    Since $F$ is lower semicontinuous, it is sufficient to prove that it is midpoint convex. Let $\mu_1,\mu_2\in X^\ast$ and set $m = \frac{\mu_1+\mu_2}{2}$. We introduce the level set 
    \begin{equation*}
        S = \left\{\nu \in X^\ast : \nu \in [\mu_1\wedge \mu_2,\mu_1\vee \mu_2], F(\nu) + F(2m-\nu) \leq F(\mu_1)+F(\mu_2) \right\}.
    \end{equation*}
    First, we claim that finding an element $\nu \in S$ such that $\nu \leq m$ is sufficient to prove midpoint convexity. Indeed, in that case $\nu \leq 2m-\nu$, and the substitutability property applied to $\nu$, $2m-\nu$ and $t_{21}=m-\nu$ gives necessarily $t_{12}=0$ and the inequality $F(m)+F(m)\leq F(\nu)+F(2m-\nu)$. Using that $\nu\in S$ then implies the desired midpoint inequality.

    In order to establish the existence of such a point $\nu$, first observe that $S$ is not empty since $\mu_1,\mu_2 \in S$, and that $S$ is weak-$\ast$ closed by lower semicontinuity of $F$. Moreover $S$ is compact for the weak-$\ast$ topology by \cref{lemma:lattice-basic-prop}\ref{item:lemma:lattice-basic-prop-tv-bound}. Let $\nu_0 \in S$. 
    Applying the substitutability property to $2m-\nu_0$, $\nu_0$ and $t_{21}=(\nu_0-m)^+$, we find that there exists $t_{12}\in [0,2(m-\nu_0)^+]$ such that $\nu_1\coloneqq \nu_0 - (\nu_0-m)^+ + t_{12}$ verifies $F(\nu_1)+F(2m-\nu_1)\leq F(\nu_0)+F(2m-\nu_0)$. We deduce that $\nu_1\in S$. Now, observe that
    \begin{equation*}
        \nu_1 - m = \nu_0 - m - (\nu_0-m)^+ + t_{12} \in [-(\nu_0-m)^-,(\nu_0-m)^-].
    \end{equation*}
    This implies $\abs{\nu_1-m} \leq (\nu_0-m)^-$ as shown in \cref{lemma:lattice-basic-prop}\ref{item:lemma:lattice-basic-prop-tv-bound}. 
    
    By iterating this process, we can construct a sequence $(\nu_k)_k$ of elements of $S$ satisfying $\abs{\nu_k-m} \leq (\nu_{k-1}-m)^-$. Decomposing $\abs{\nu_k-m} =(\nu_k-m)^+ + (\nu_k-m)^-$ and pairing with an arbitrary $\phi\geq 0$, we obtain a telescoping sum and the estimate
    \begin{equation*}
        \sum_{k=1}^\infty \bracket{(\nu_k-m)^+,\phi}  \leq \bracket{(\nu_0-m)^-,\phi}.
    \end{equation*}
    We deduce that $\bracket{(\nu_k-m)^+, \phi} \to 0$ as $k \to\infty$.

    Since $S$ is compact for the weak-$\ast$ topology, we may pass to a subsequence (still denoted by $(\nu_k)$) that weak-$\ast$ converges to some $\nu\in S$. 
    % Set $s_n = \sum_{k=1}^n (\nu_k-m)^+$ then
    % \begin{equation*}
    %     0 \leq s_n \leq \sum_{k=1}^n (\nu_{k-1}-m)^- - (\nu_k-m)^- \leq (\nu_0-m)^- - (\nu_n - m)^- \leq (\nu_0-m)^-.
    % \end{equation*}
    % Thus $s_n \in [0,(\nu_0-m)^-]$ which is compact by \cref{lemma:prop-segments-ast} . Thus a relabeled subsequence $(s_{n_k})_k$ converges weak-$\ast$ towards $s \in X^\ast$. Then for $\phi \in X$ such that $\phi \geq 0$
    % \begin{equation*}
    %      0 \leq \langle (\nu_{n_k}-m)^+ , \phi\rangle \leq \langle s_{n_{k+1}} - s_{n_k},\phi\rangle \to 0,
    % \end{equation*}
    % because the positive parts are greater than or equal to $0$ for the order on $X^\ast$.
    % Moreover b
    By \cref{lem:lsc-positive-part}, we have for any $\phi \geq 0$
    \begin{equation*}
        \langle (\nu-m)^+,\phi \rangle \leq \lim_k \langle (\nu_{k}-m)^+ , \phi\rangle = 0.
    \end{equation*}
    Since $\phi \geq 0$ is arbitrary, we have shown that $(\nu-m)^+ \leq 0$, or equivalently $\nu \leq m$.
    % This holds for any $h \geq 0$ thus we have shown that $(\nu-m)^+ \leq 0$ which implies that $\nu \leq m$. Finally since $(\nu-m)^- \in [0,(2m-\nu-\nu)^+]$ there is $t \in [0,2(\nu-m)^+] = \{0\}$ such that $\bar{\nu} = \nu +t - (\nu-m)^- \in S$ by substitutability of $F$. Since $t=0$ and $\nu \geq m$ we have that $\bar{\nu} = m$ which concludes the proof of $m \in S$.
\end{proof}

Similar to \Cref{lemma:P-order-joint-subm} stated in the world of submodularity, relations of the form $F_2\leQ F_1$ may come from joint substitutability. Recall that $X\times Y$ is equipped with the product Banach lattice structure, as in \Cref{sec:def-submodular}.

\begin{lemma}\label{lemma:Q-order-joint-exch}
    Let $X$ and $Y$ be two Banach lattices and consider a function $F\colon X\times Y\to\R\cup\{+\infty\}$ which is jointly substitutable. If $\mu_1,\mu_2\in X$ satisfy $\mu_1\leq \mu_2$ then $F(\mu_2,\cdot)\leQ F(\mu_1,\cdot)$. 
\end{lemma}
\begin{proof}
    Fix $\mu_1\leq\mu_2\in X$, $\nu_1,\nu_2\in Y$ and $t_{21}\in [0,(\nu_2-\nu_1)^+]$. Denote  $\gamma_i=(\mu_i,\nu_i)$ for $i=1,2$. We first extend $t_{21}$ from $Y$ to $X\times Y$ into $\tau_{21}=(0_{X},t_{21})$ and note that $0\leq\tau_{21}\leq (\gamma_2-\gamma_1)^+$. Since $F$ is substitutable there exists $\tau_{12}\in X\times Y$ such that 
    \begin{equation}\label{eq:proof-lemma:Q-order-joint-exch-1}
         \tau_{12}\in [0,(\gamma_1-\gamma_2)^+],
    \end{equation}
    and 
    \begin{equation}\label{eq:proof-lemma:Q-order-joint-exch-2}
        F(\gamma_1+\tau_{21}-\tau_{12})+F(\gamma_2-\tau_{21}+\tau_{12})\leq F(\gamma_1)+F(\gamma_2).
    \end{equation}

    Using that $\mu_1\leq\mu_2$, we deduce from \eqref{eq:proof-lemma:Q-order-joint-exch-1} that the first component of $\tau_{12}$ vanishes. Therefore $\tau_{12}=(0_{X}, t_{12})$ for some $0\leq t_{12}\leq (\nu_1-\nu_2)^+$. As a consequence, \eqref{eq:proof-lemma:Q-order-joint-exch-2} takes the form 
    \begin{equation}
        F(\mu_1,\nu_1+t_{21}-t_{12})+F(\mu_2,\nu_2-t_{21}+t_{12})\leq F(\mu_1,\nu_1)+F(\mu_2,\nu_2).
    \end{equation}
    This is the desired inequality.
\end{proof}

Q-dominance is also preserved by certain minimization procedures.

\begin{lemma}\label{lemma:Q-order-min}
    Let $X$ and $Y$ be two Banach lattices and consider two functions $G_1,G_2\colon X\times Y\to\R\cup\{+\infty\}$ bounded from below and jointly substitutable. Define $F_i\colon X\to\R\cup\{+\infty\}$ by $F_i(\mu)=\inf_{\nu\in Y} G_i(\mu,\nu)$ for $i=1,2$. If $G_1\leQ G_2$ then $F_1\leQ F_2$. 
\end{lemma}
\begin{proof}
    Let $\mu_1,\mu_2\in X$ and $t_{21}\in [0,(\mu_2-\mu_1)^+]$. Let $\nu_1,\nu_2\in Y$ and  denote  $\gamma_i=(\mu_i,\nu_i)$ for $i=1,2$. Extend $t_{21}$ from $X$ to $\tau_{21}\in X\times Y$ in such a way that $0\leq\tau_{21}\leq (\gamma_2-\gamma_1)^+$, for instance taking $\tau_{21}=(t_{21},0_{Y})$. Since $G_1\leQ G_2$ there exists $\tau_{12}\in X\times Y$ such that $0\leq \tau_{12}\leq (\gamma_1-\gamma_2)^+$
    and 
    \begin{equation}\label{eq:proof-lemma:Q-order-min-2}
        G_1(\gamma_1+\tau_{21}-\tau_{12})+G_2(\gamma_2-\tau_{21}+\tau_{12})\leq G_1(\gamma_1)+G_2(\gamma_2).
    \end{equation}

    Let $t_{12}$ denote the first component of $\tau_{12}$. Then $0\leq t_{12}\leq (\mu_1-\mu_2)^+$ and the left-hand side of \eqref{eq:proof-lemma:Q-order-min-2} can be bounded below by the respective infima of $G_1$ and $G_2$ over $Y$, resulting in 
    \begin{equation}
        F_1(\mu_1+t_{21}-t_{12}) + F_2(\mu_2-t_{21}+t_{12}) \leq G_1(\mu_1,\nu_1)+G_2(\mu_2,\nu_2).
    \end{equation}
    Since $\nu_1$ and $\nu_2$ are arbitrary we obtain $F_1(\mu_1+t_{21}-t_{12}) + F_2(\mu_2-t_{21}+t_{12}) \leq F_1(\mu_1)+F_2(\mu_2)$. 
\end{proof}

Finally, the next two results are the main ingredients to prove comparison principles using substitutability. They work the same way as \Cref{prop:topkis} and \Cref{lemma:P-singletons} in the previous section. The first one allows us to compare solutions of variational problems.

\begin{proposition}\label{prop:argmin-exch}
    Let $F_1,F_2\colon Y\to\R\cup\{+\infty\}$ be two functions on $Y$. If $F_1 \leQ F_2$ then $\argmin F_1 \leQ \argmin F_2$.
\end{proposition}
    
\begin{proof}
    To be consistent with the notation introduced before, let us reverse the roles of $F_1$ and $F_2$ and show that $F_2 \leQ F_1 \implies\argmin F_2 \leQ \argmin F_1$. 
    If either $\argmin F_1$ or $\argmin F_2$ is empty the result holds vacuously. Otherwise, take $\mu_1 \in \argmin F_1$ and $\mu_2 \in \argmin F_2$. Since $F_2\leQ F_1$, for any $0 \leq t_{21} \leq (\mu_2 - \mu_1)^+$ there exists $0 \leq t_{12} \leq (\mu_1-\mu_2)^+$ such that
    \begin{equation*}
        F_1(\mu_1 +t_{21}-t_{12}) + F_2(\mu_2 - t_{21}+t_{12}) \leq F_1(\mu_1) + F_2(\mu_2).
    \end{equation*}    
    Thus $\mu_1 +t_{21}-t_{12} \in \argmin F_1$ and $\mu_2 - t_{21}+t_{12}\in \argmin F_2$, proving the desired result. 
\end{proof}

The second result connects Q-dominance for singletons and the lattice order on $Y$.

\begin{lemma}\label{lemma:Q-singletons}
    Let $\mu_1,\mu_2\in Y$. If $\{\mu_1\} \leQ \{\mu_2\}$ then $\mu_1 \leq \mu_2$.
\end{lemma}

\begin{proof}
    Again to be consistent with the notation introduced above, we exchange the roles of $\mu_1$ and $\mu_2$ in this proof. If $\{\mu_2\} \leQ \{\mu_1\}$ then choosing $t_{21}\coloneqq(\mu_2-\mu_1)^+$ there exists $0 \leq t_{12} \leq (\mu_1-\mu_2)^+$ such that $\mu_1+t_{21}-t_{12} \in \{\mu_1\}$. In other words, $t_{21}=t_{12}$ which means that $(\mu_2-\mu_1)^+\leq (\mu_1-\mu_2)^+=(\mu_2-\mu_1)^-$. This implies $\mu_2-\mu_1\leq 0$. 
\end{proof}
Unlike P-dominance, Q-dominance is not compatible with addition in general. 
This incompatibility
motivates us to introduce a strengthening of substitutability, \emph{total substitutability}. 

\begin{definition}
    A function $H\colon Y\to\R\cup\{+\infty\}$ is \emph{totally substitutable} if for every $\mu_1,\mu_2\in Y$, and any $\mu_1',\mu_2' \in [\mu_1 \wedge \mu_2,\mu_1 \vee \mu_2]$ such that $\mu_1'+\mu_2'=\mu_1+\mu_2$, we have
    \begin{equation*}
        H(\mu_1')+H(\mu_2') \leq H(\mu_1)+H(\mu_2).
    \end{equation*}
    We also define a set $B\subset Y$ to be totally substitutable when $\iota_B$ is totally substitutable. 
\end{definition}

Simple examples of substitutable functions will be given in \Cref{sec:ex-exchangeable}. The main appeal of total substitutability lies the following result.

\begin{proposition}\label{prop:sum-total-exch}
    Suppose that $F_1,F_2\colon Y\to\R\cup\{+\infty\}$ satisfy $F_1\leQ F_2$ and that $H\colon Y\to\R\cup\{+\infty\}$ is totally substitutable. Then $F_1+H\leQ F_2+H$.
\end{proposition}
\begin{proof}
    To be consistent with the notation introduced before, we reverse once more the roles of $F_1$ and $F_2$ and show that $F_2 \leQ F_1 \implies F_2+H \leQ F_1+H$. Fix $\mu_1,\mu_2\in Y$ and $t_{21}\in [0,(\mu_2-\mu_1)^+]$. Since $F_2\leQ F_1$ there exists $t_{12}\in [0,(\mu_1-\mu_2)^+]$ such that $\mu_1'\coloneqq \mu_1+t_{21}-t_{12}$ and $\mu_2'\coloneqq \mu_2+t_{12}-t_{21}$ verify
    \begin{equation}\label{eq:proof-prop-sum-total-exch-1}
        F_1(\mu_1')+F_2(\mu_2')\leq F_1(\mu_1)+F_2(\mu_2).
    \end{equation}
    By construction of $t_{12}$ and $t_{21}$ we have that $\mu_1',\mu_2'\in[\mu_1\wedge\mu_2,\mu_1\vee\mu_2]$. Total substitutability of $H$ then implies 
    \begin{equation}\label{eq:proof-prop-sum-total-exch-2}
        H(\mu_1')+H(\mu_2')\leq H(\mu_1)+H(\mu_2).
    \end{equation}
    Combining \eqref{eq:proof-prop-sum-total-exch-1} and \eqref{eq:proof-prop-sum-total-exch-2} gives the desired result.
\end{proof}

As a direct consequence of \Cref{prop:sum-total-exch}, we find 
\begin{corollary}
    Let $F$ and $H\colon Y\to\R\cup\{+\infty\}$ respectively denote a substitutable and totally substitutable function. Then $F+H$ is substitutable.
\end{corollary}

\begin{remark}[On domains of definition]\label{rem:domain-def}
    A simple but practical use of totally substitutable sets concerns domains of definition. Let $D$ be a totally substitutable subset of a Banach lattice $Y$. Suppose that $F\colon D\to\R\cup\{+\infty\}$ satisfies the definition of a substitutable function when taking $\mu_1, \mu_2\in D$. Since $D$ is totally substitutable, $\mu_1'\coloneqq\mu_1+t_{21}-t_{12}$ and $\mu_2'\coloneqq\mu_2-t_{21}+t_{12}$ belong to $D$ for any $t_{21}\in[0,(\mu_2-\mu_1)^+]$ and $t_{12}\in[0,(\mu_2-\mu_1)^-]$. Therefore $F$ can readily be extended to a substitutable function $\tilde F$ defined on all of $Y$ by setting $\tilde F(\mu)=+\infty$ whenever $\mu\notin D$. 
\end{remark}

A concrete application of the preceding remark is for substitutable functions defined over spaces of measures which are naturally only defined over positive measures, see \cref{ex:proba-exchangeable}.

When we consider totally substitutable functions over Banach lattices of measures, it will be useful to reformulate the order interval condition in terms of a Radon--Nikodym density.

\begin{lemma}\label{lemma:total-exch-measures-density}
    Let $(S,\Sigma)$ be a measurable space, let $\mu_1,\mu_2$ be two finite signed measures on $(S,\Sigma)$ and take $\mu_1'\in [\mu_1\wedge\mu_2,\mu_1\vee\mu_2]$. Then there exists a function $h\in L^1(\abs{\mu_1-\mu_2})$, valued in $[0,1]$, such that $d\mu_1'=(1-h)\,d\mu_1+h \,d\mu_2$.
\end{lemma}
\begin{proof}
    Set $\sigma=\mu_1'-\mu_1$ and $\nu=\mu_2-\mu_1$. By \cref{lemma:lattice-basic-prop}\ref{item:lemma:lattice-basic-prop-sum} we have 
    $\mu_1\vee\mu_2=\mu_1+(\mu_2-\mu_1)^+$ and $\mu_1\wedge\mu_2=\mu_1-(\mu_2-\mu_1)^-$
    . 
    The inequality $\mu_1'\leq\mu_1\vee\mu_2$ can therefore be written as $\sigma\leq\nu^+$, which by \cref{lemma:lattice-basic-prop}\ref{item:lemma:lattice-basic-prop-sum-pp-np} implies 
    \begin{equation}\label{eq:proof-lemma:total-exch-measures-density-1}
        \sigma^+\leq\nu^+.
    \end{equation}
    Similarly, $\mu_1'\geq \mu_1\wedge\mu_2$ can be written as $\sigma\geq -\nu^-$, which in turn implies 
    \begin{equation}\label{eq:proof-lemma:total-exch-measures-density-2}
        \sigma^-\leq\nu^-.
    \end{equation}
    Let $S=S_+\cup S_-$ be a Hahn decomposition for $\nu$ (see \cref{ex:signed-measures}). By \eqref{eq:proof-lemma:total-exch-measures-density-1}--\eqref{eq:proof-lemma:total-exch-measures-density-2} we see that $\sigma^+(S_-)=0$ and $\sigma^-(S_+)=0$; therefore $S_+\cup S_-$ is also a Hahn decomposition for $\sigma$. Using the Radon--Nikodym theorem, define 
    \[
    h_1=\frac{d\sigma^+}{d\nu^+}\,\,\text{on $S_+$ \quad and\quad }h_2=\frac{d\sigma^-}{d\nu^-}\,\,\text{on $S_-$}.
    \]
    By \eqref{eq:proof-lemma:total-exch-measures-density-1}--\eqref{eq:proof-lemma:total-exch-measures-density-2}, $h_1$ and $h_2$ can be chosen valued in $[0,1]$. Moreover since $S_+$ and $S_-$ are disjoint, $h_1$ and $h_2$ can be combined into a function $h$ defined over $S$, which satisfies
    \[
    h = \frac{d\sigma}{d\nu}\text{ on }S.
    \]
    Note that $h$ is valued in $[0,1]$, and that $h\in L^1(\abs\nu)$ since $h_1\in L^1(\nu^+)$ and $h_2\in L^1(\nu^-)$. This concludes the proof.
\end{proof}

\subsubsection{Examples of substitutable sets and functions} \label{sec:ex-exchangeable}

We start with some simple considerations. We recall that given a Polish space $\Omega$, $\Mc(\Omega)$, $\Mp(\Omega)$ and $\Pc(\Omega)$ denote respectively the sets of finite signed measures, finite positive measures and probability measures on $\Omega$. In the following examples, we equip $\Mc(\Omega)$ with its natural Banach lattice structure, see \cref{ex:signed-measures}.

\begin{example}\label{ex:proba-exchangeable}
    Let $\Omega$ be a Polish space. Then (i) $\Mp(\Omega)$ is a totally substitutable subset of $\Mc(\Omega)$, and (ii) $\Pc(\Omega)$ is a substitutable subset of $\Mc(\Omega)$. 
    
    For (i), observe that if $\mu_1,\mu_2$ are two positive measures, then $\mu_1\wedge\mu_2$ is positive too. Therefore any $\mu_1',\mu_2'\geq \mu_1\wedge\mu_2$ are positive. This shows $\Mp(\Omega)$ to be totally substitutable. 
    
    For (ii), fix $\mu_1,\mu_2 \in \mathcal{P}(\Omega)$ and $t_{21} \in [0,(\mu_2-\mu_1)^+]$. Here we may as well suppose that $\mu_1\neq\mu_2$, which implies $(\mu_2-\mu_1)^+ \neq 0$. Let $t_{12} = \frac{t_{21}(\Omega)}{(\mu_2-\mu_1)^+(\Omega)}(\mu_1-\mu_2)^+$. We have $t_{12}\in [0,(\mu_1-\mu_2)^+]$ because $0 \leq t_{21}(\Omega) \leq (\mu_2-\mu_1)^+(\Omega) $. Let $\mu_1'=\mu_1 +t_{21}- t_{12}$, a positive measure by (i). Note that $(\mu_1-\mu_2)^+(\Omega)=(\mu_1-\mu_2)(\Omega)+(\mu_2-\mu_1)^+(\Omega)=(\mu_2-\mu_1)^+(\Omega)$ since $\mu_1$ and $\mu_2$ have mass $1$. Therefore $t_{12}(\Omega)=t_{21}(\Omega)$, which shows that  $\mu_1'\in\Pc(\Omega)$. By a similar argument, $\mu_2'\coloneqq \mu_2 - t_{21} + t_{12}\in\Pc(\Omega)$.
\end{example}

In particular, when a functional $F$ defined over the positive measures satisfies the requirements for substitutability when restricted to positive measures, it can be simply extended to a substitutable function over all signed measures by assigning the value $+\infty$ to measures which are not positive. See \cref{rem:domain-def}.

Let us now turn our attention to an important class of totally substitutable functions, the convex internal energies, also called \emph{Csiszár divergences} or \emph{relative entropy functionals}. We follow the presentation of \cite{AGS2008,LieroMielkeSavare2018}.

\begin{definition}\label{def:entropy-function}
    A function $f\colon [0,+\infty)\to [0,+\infty]$ is called an \emph{entropy function} if it is proper, convex, lower semicontinuous, and has superlinear growth, in the sense that $f(s)/s\to +\infty$ as $s\to+\infty$.
\end{definition}

\begin{definition}[Internal energy]\label{def:internal-energy}
    Let $\Omega$ be a Polish space. Fix $m\in\Mp(\Omega)$. Define the \emph{internal energy} $H_{f,m}\colon \Mc(\Omega)\to \R\cup\{+\infty\}$ associated with $f$ and $m$ by
    \begin{equation*}
        H_{f,m}(\mu)=\begin{cases}
            \displaystyle\int_\Omega f\Big(\frac{d\mu}{dm}\Big)\,dm\quad&\text{if $\mu$ is positive and absolutely continuous w.r.t. $m$,}\\
            \quad+\infty & \text{otherwise.}
        \end{cases}
    \end{equation*}
\end{definition}

When $m$ is the null measure, we have $H_{f,0}(\mu)=0$ if $\mu$ is the null measure and $+\infty$ otherwise, see \cite[Sect. 2.4]{LieroMielkeSavare2018}.
While $H_{f,m}$ is only finite over positive measures, we define it here over all signed measures to better integrate it in our framework. Note that doing this does not interfere with substitutability properties, see \cref{rem:domain-def}.

\begin{lemma}\label{lemma:internal_energy}
    Given a Polish space $\Omega$, let $H_{f,m}$ be the internal energy functional associated with a reference measure $m\in\Mp(\Omega)$ and an entropy function $f$. Then $H_{f,m}$ is totally substitutable.
    \end{lemma}
\begin{proof}
    Let $\mu_1,\mu_2\in\Mc(\Omega)$, and take $\mu_1',\mu_2'\in[\mu_1\wedge\mu_2,\mu_1\vee\mu_2]$ such that $\mu_1'+\mu_2'=\mu_1+\mu_2$. Let us show that 
    \begin{equation}\label{eq:proof-lemma:internal_energy}
        H_{f,m}(\mu'_1)+H_{f,m}(\mu'_2) \leq H_{f,m}(\mu_1)+H_{f,m}(\mu_2).
    \end{equation}    
    If $m$ is the null measure then the right-hand side is finite only when $\mu_1=\mu_2=0$ in which case $\mu_1'=\mu_2'=0$ and \eqref{eq:proof-lemma:internal_energy} trivially holds.

    Let us now suppose that $m$ is non-null. In view of the definition of $H_{f,m}$, we may as well take $\mu_1$ and $\mu_2$ positive and absolutely continuous with respect to $m$. By \cref{lemma:total-exch-measures-density} there exists a function $h\in L^1(\abs{\mu_1-\mu_2})$ valued in $[0,1]$ such that $d\mu_1'=(1-h)d\mu_1+h\,d\mu_2$. Then $d\mu_2'=h\,d\mu_1+(1-h)d\mu_2$. Clearly, $\mu_1'$ and $\mu_2'$ are positive and absolutely continuous with respect to $m$. By convexity of $f$,
    \begin{align*}
    H_{f,m}(\mu'_1)&=\int_\Omega f\left((1-h(x))\frac{d\mu_1}{dm}(x) + h(x)\frac{d\mu_2}{dm}(x)\right)dm(x) \\
    &\leq \int_\Omega \Big((1-h(x))f\left(\frac{d\mu_1}{dm}(x)\right) + h(x) f\left(\frac{d\mu_2}{dm}(x)\right)\Big)\,dm(x) .
    \end{align*}
    This expression may contain infinities, in which case the convention $0\cdot(+\infty)=0$ should be adopted. Using a similar argument for $\mu'_2$, we find that 
    \[
    H_{f,m}(\mu'_2)\leq \int_\Omega \Big(h(x)\,f\left(\frac{d\mu_1}{dm}(x)\right) + (1-h(x)) \,f\left(\frac{d\mu_2}{dm}(x)\right)\Big)\,dm(x).
    \]
    Summing both resulting inequalities we obtain \eqref{eq:proof-lemma:internal_energy}.
    
\end{proof}

\subsection{Duality between submodularity and substitutability}\label{sec:duality}

The two main results in this section prove that submodularity and substitutability are conjugate properties for convex duality. We fix a Banach lattice $X$ and consider duality in $X\times X^*$: the convex conjugate of a proper function $E\colon X\to\R\cup\{+\infty\}$ is the function $E^*\colon X^*\to\R\cup\{+\infty\}$ defined by $E^*(\mu)=\sup_{\phi\in X}\bracket{\phi,\mu}-E(\phi)$, while the convex conjugate of a proper function $F\colon X^*\to\R\cup\{+\infty\}$ is the function $F^*\colon X\to\R\cup\{+\infty\}$ defined by $F^*(\phi)=\sup_{\mu\in X^*}\bracket{\phi,\mu}-F(\mu)$ \cite[Chapter I]{EkelandTemamBook}.

\begin{theorem}\label{thm:duality-submodular-exch}
    Let $E_1,E_2\colon X\to\R\cup\{+\infty\}$ be two proper convex functions such that $E_1\leP E_2$. Then $E_2^*\leQ E_1^*$. In particular, if $E\colon X\to\R\cup\{+\infty\}$ is a submodular proper convex function then $E^*$ is substitutable.
\end{theorem}

\begin{theorem}\label{thm:duality-exch-submodular}
    Let $F_1,F_2\colon X^\ast\to\R\cup\{+\infty\}$ be two proper functions such that $F_2\leQ F_1$. Then $(F_1)^* \leP (F_2)^*$. In particular if $F\colon X^*\to\R\cup\{+\infty\}$ is a proper substitutable function, then $F^*\colon X\to\R\cup\{+\infty\}$ is submodular.  
\end{theorem}

As a direct consequence of the preceding two theorems we obtain:

\begin{corollary}\label{cor:submodular-exchangeable}
    Let $E\colon X\to\R\cup\{+\infty\}$ be a proper lower semicontinuous convex function. Then $E^*$ is substitutable if and only if $E$ is submodular.
\end{corollary}

We now set out to prove Theorems \ref{thm:duality-submodular-exch} and \ref{thm:duality-exch-submodular}. 
Both can be deduced from a property of the duality bracket which could be described as a type of joint submodularity--substitutability.

\begin{lemma}\label{lemma:duality-bracket}
    Let $\phi_1,\phi_2\in X$ and $\mu_1,\mu_2\in X^*$. There exists $t_{12}\in [0,(\mu_1-\mu_2)^+]$ such that for every $t_{21}\in [0,(\mu_2-\mu_1)^+]$,
    \begin{equation} \label{eq:lemma:duality-bracket}
        \bracket{\mu_1+t_{21}-t_{12},\phi_1}+
        \bracket{\mu_2-t_{21}+t_{12},\phi_2}
        \leq
        \bracket{\mu_1,\phi_1\wedge\phi_2}+
        \bracket{\mu_2,\phi_1\vee\phi_2}.
    \end{equation}
\end{lemma}

Note that contrary to the case for substitutable functions, $t_{12}$ is independent of $t_{21}$ in \Cref{lemma:duality-bracket}. However $t_{12}$ depends on $\phi_1,\phi_2$, and $\mu_1,\mu_2$. 

\begin{proof}[Proof of \Cref{lemma:duality-bracket}]
    Fix $\phi_1,\phi_2\in X$, $\mu_1,\mu_2\in X^*$. We want to find $t_{12}\in [0,(\mu_1-\mu_2)^+]$ such that \eqref{eq:lemma:duality-bracket} holds for every $t_{21}\in [0,(\mu_2-\mu_1)^+]$, 
    i.e.\ rearranging terms and denoting $u=\phi_1-\phi_2$, such that for all $t_{21}\in [0,(\mu_2-\mu_1)^+]$,
    \begin{equation}
        \bracket{\mu_1,u^+}
        +\bracket{t_{21},u}\leq \bracket{\mu_2,u^+}
        +\bracket{t_{12},u}.
    \end{equation}
    The above inequality can be written as
    \begin{equation}
        \bracket{(\mu_1-\mu_2)^+,u^+}
        +\bracket{t_{21},u^+}
        \leq \bracket{(\mu_2-\mu_1)^+,u^+}
        +\bracket{t_{12},u}
        +\bracket{t_{21},u^-}.
    \end{equation}
    By \Cref{lemma:prop-segments-ast}, choosing $t_{12}\in\argmax_{m\in [0,(\mu_1-\mu_2)^+]}\bracket{m,u}$ gives $\bracket{t_{12},u}=\bracket{(\mu_1-\mu_2)^+,u^+}$. To conclude, note that every $t_{21}\in [0,(\mu_2-\mu_1)^+]$ satisfies $\bracket{t_{21},u^+}\leq \bracket{(\mu_2-\mu_1)^+,u^+}$ and $\bracket{t_{21},u^-}\geq 0$. 
\end{proof}

With the help of \Cref{lemma:duality-bracket} we may now prove Theorems \ref{thm:duality-submodular-exch} and \ref{thm:duality-exch-submodular}.

\begin{proof}[Proof of \Cref{thm:duality-submodular-exch}]
    Fix $\mu_1,\mu_2\in X^*$ and $t_{21}\in [0,(\mu_2-\mu_1)^+]$. Let $K=[0,(\mu_1-\mu_2)^+]$. By \Cref{lemma:prop-segments-ast}(i) $K$ is a weak-$*$ compact subset of $X^*$. Let $\phi_1,\phi_2\in X$. By \Cref{lemma:duality-bracket} we have 
    \begin{equation*}
        \min_{s\in K}\bracket{\mu_1+t_{21}-s,\phi_1} + \bracket{\mu_2-t_{21}+s,\phi_2} \leq \bracket{\mu_1,\phi_1\wedge\phi_2}
        +\bracket{\mu_2,\phi_1\vee\phi_2}.
    \end{equation*}    
    This inequality in fact holds uniformly over $t_{21}\in [0,(\mu_2-\mu_1)^+]$, but we don't use this here. Using that $E_1 \leP E_2$, we have 
    \begin{align*}
        &\min_{s\in K}\bracket{\mu_1+t_{21}-s,\phi_1} -E_1(\phi_1)+ 
        \bracket{\mu_2-t_{21}+s,\phi_2}-E_2(\phi_2) \\
        \leq  &
        \bracket{\mu_1,\phi_1\wedge\phi_2}-E_1(\phi_1\wedge\phi_2)
        +\bracket{\mu_2,\phi_1\vee\phi_2} - E_2(\phi_1\vee\phi_2)\\
        \leq & E_1^*(\mu_1) + E_2^*(\mu_2).
    \end{align*}    
    Therefore 
    \begin{equation*}
        \sup_{\phi_1,\phi_2\in X}\min_{s\in K}\bracket{\mu_1+t_{21}-s,\phi_1} -E_1(\phi_1)+ 
        \bracket{\mu_2-t_{21}+s,\phi_2}-E_2(\phi_2) \leq E_1^*(\mu_1) + E_2^*(\mu_2).
    \end{equation*}
    Define $\phi=\phi_1\oplus\phi_2$ and $L(s,\phi)=\bracket{\mu_1+t_{21}-s,\phi_1}+\bracket{\mu_2-t_{21}+s,\phi_2}-E_1(\phi_1)-E_2(\phi_2)$. The function $L$ is convex and lower semicontinuous in $s$ on the compact set $K$, and concave in $\phi$. Therefore we have a minimax theorem, see e.g. \cite[Theorem 3.1]{SimonsMinimaxBook}. We obtain 
    \begin{equation*}
        \min_{s\in K}\sup_{\phi_1,\phi_2\in X}\bracket{\mu_1+t_{21}-s,\phi_1} -E_1(\phi_1)+ 
        \bracket{\mu_2-t_{21}+s,\phi_2}-E_2(\phi_2) \leq E_1^*(\mu_1) + E_2^*(\mu_2).
    \end{equation*}
    We deduce the existence of some $t_{12}=s\in K$ such that 
    \begin{equation*}
        E_1^*(\mu_1+t_{21}-t_{12})+
        E_2^*(\mu_2-t_{21}+t_{12})
        \leq E_1^*(\mu_1) + E_2^*(\mu_2).
    \end{equation*}
\end{proof}

\begin{proof}[Proof of \Cref{thm:duality-exch-submodular}]
    Fix $\phi_1,\phi_2\in X$. Let $\mu_1, \mu_2\in X^*$. Applying \Cref{lemma:duality-bracket} to $\tilde\mu_1\coloneqq -\mu_1$ and $\tilde\mu_2\coloneqq -\mu_2$, there exists $\tilde t_{12}\in [0,(\tilde\mu_1-\tilde\mu_2)^+]=[0,(\mu_2-\mu_1)^+]$ such that the inequality 
    \begin{equation}\label{eq:proof-lemma:duality-exch-submodular-1}
        \bracket{\tilde\mu_1+s-\tilde t_{12},\phi_1}+
        \bracket{\tilde\mu_2-s+\tilde t_{12},\phi_2}
        \leq
        \bracket{\tilde\mu_1,\phi_1\wedge\phi_2}+
        \bracket{\tilde\mu_2,\phi_1\vee\phi_2}
    \end{equation}    
    holds for every $s\in [0,(\tilde\mu_2-\tilde\mu_1)^+]=[0,(\mu_1-\mu_2)^+]$. Let us rewrite \eqref{eq:proof-lemma:duality-exch-submodular-1} as 
    \begin{equation}\label{eq:proof-lemma:duality-exch-submodular-2}
        \bracket{\mu_1,\phi_1\wedge\phi_2}+
        \bracket{\mu_2,\phi_1\vee\phi_2} 
        \leq 
        \bracket{\mu_1-s+\tilde t_{12},\phi_1}+
        \bracket{\mu_2+s-\tilde t_{12},\phi_2}.
    \end{equation}   
    Using now that $F_2\leQ F_1$, by choosing $t_{21}\coloneqq \tilde t_{12}\in [0,(\mu_2-\mu_1)^+]$ we know that there exists $t_{12}\in [0,(\mu_1-\mu_2)^+]$ such that 
    \begin{equation}\label{eq:proof-lemma:duality-exch-submodular-3}
        F_1(\mu_1+t_{21}-t_{12}) + 
        F_2(\mu_2-t_{21}+t_{12})\leq
        F_1(\mu_1)+F_2(\mu_2).
    \end{equation}
    Combining \eqref{eq:proof-lemma:duality-exch-submodular-3} and \eqref{eq:proof-lemma:duality-exch-submodular-2} with $s=t_{12}$, we obtain  
    \begin{align*}\label{eq:proof-lemma:duality-exch-submodular-4}
        & \bracket{\mu_1,\phi_1\wedge\phi_2}-F_1(\mu_1)+
        \bracket{\mu_2,\phi_1\vee\phi_2}-F_2(\mu_2) \\
        \leq &\bracket{\mu_1-t_{12}+t_{21},\phi_1}-F_1(\mu_1+t_{21}-t_{12})+
        \bracket{\mu_2+t_{12}-t_{21},\phi_2}-F_2(\mu_2-t_{21}+t_{12}).        
    \end{align*}   
    Since the terms in $X^*$ which appear in the right-hand side match exactly two by two, we bound by the supremum over $X^*$ and obtain
    \begin{equation*}
        \bracket{\mu_1,\phi_1\wedge\phi_2}-F_1(\mu_1)+
        \bracket{\mu_2,\phi_1\vee\phi_2}-F_2(\mu_2) \leq  F_1^*(\phi_1)+F_2^*(\phi_2).
    \end{equation*}
    Since this last inequality holds for arbitrary $\mu_1$ and $\mu_2\in X^*$, we deduce the desired result.
\end{proof}

\section{Comparison principles for optimal transport}\label{sec:cp-exch}

This section contains applications of submodularity and substitutability to optimal transport. 
We present comparison principles for the standard optimal transport problem, the entropic transport problem, and the unbalanced transport problem, in general settings: infinite-dimensional source and target spaces, plainly continuous cost function, arbitrary source and target measures with no regularity assumptions.

\subsection{Basic setting, definitions and notations}\label{sec:cp-basic-setting}

We consider two compact metric spaces $\Omegax$ and $\Omegay$, and a 
continuous function $c\colon \Omegax\times\Omegay\to\R$. 
Let $C(A)$, $\Mc(A)$, $\Mp(A)$ and $\Pc(A)$ respectively denote the sets of continuous functions, finite signed Borel measures, finite positive Borel measures, and Borel probability measures on $A$.

\paragraph{Supports of measures.}
Let $\mu\in\Mp(\Omega)$. We recall that there exists an open set $U\subset\Omega$ such that for any open set $O \subset\Omega$, $\mu(O)=0$ if and only if $O\subset U$. Therefore $U$ is the largest open set $O$ such that $\mu(O)=0$. The complement of $U$ is the \emph{support} of $\mu$, denoted by $\supp\mu$. It is the smallest closed set $C$ on which $\mu$ is concentrated. We say that $\mu$ has \emph{full support} if $\supp\mu=\Omega$. 

Consider now a signed measure $\sigma\in\Mc(\Omega)$. The support of $\sigma$ is defined as the support of its total variation $\abs\sigma$. Let $B$ be a Borel set. $B$ is called a \emph{positive set} for $\sigma$ if $\sigma(A)\geq 0$ for all Borel sets $A\subset B$. This is equivalent to $\inf_{A\subset B}\sigma(A)=0$, where the infimum runs over Borel sets $A$, since $\sigma(\emptyset)=0$. By the expression of the negative part, we find that $B$ is a positive set if and only if $\sigma^-(B)=0$. 

We will often write ``$\sigma \geq 0$ on $B$'' to say that $B$ is a positive set for $\sigma$. Frequently, $\sigma$ is given as the difference of two measures $\mu_1$ and $\mu_2$; then we write that ``$\mu_1\leq\mu_2$ on $B$'' to express that $B$ is a positive set for $\mu_2-\mu_1$. 

The preceding definitions and observations imply that given $\sigma\in\Mc(\Omega)$, there exists a largest open set $U$ which is a positive set for $\sigma$; it is given by $U=\Omega\setminus(\supp\sigma^-)$.

\paragraph{Standard optimal transport.} 
Given two measures $\mu\in\Mp(\Omegax)$ and $\nu\in\Mp(\Omegay)$, the optimal transport cost is defined by 
\begin{equation}\label{eq:ot-problem}
    \Tc_c(\mu,\nu) = \inf_{\pi\in\Pi(\mu,\nu)} \int_{\Omegax\times\Omegay} c(x,y)\,d\pi(x,y).
\end{equation}
Here $\Pi(\mu,\nu)$ denotes the set of measures $\pi\in\Mp(\Omegax\times\Omegay)$ having first marginal $\mu$ and second marginal $\nu$. When $\mu$ and $\nu$ have  different total masses, the constraint in \eqref{eq:ot-problem} is empty and we set $\Tc_c(\mu,\nu)=+\infty$. We refer to \cite{VillaniBook2009,SantambrogioBook} for a standard exposition of the subject.

The dual formulation of optimal transport takes the form 
\begin{equation}\label{eq:ot-dual-phi-psi-formulation}
    \Tc_c(\mu,\nu) = \sup_{(\phi,\psi)\in A} \int_\Omegax \phi\,d\mu - \int_\Omegay \psi\,d\nu,
\end{equation}
where $A$ consists of all functions $\phi\in C(\Omegax), \psi\in C(\Omegay)$ such that 
\begin{equation}\label{eq:constraint-dual-ot}
    \forall x\in \Omegax,\forall y\in\Omegay,\quad \phi(x)-\psi(y)\leq c(x,y).
\end{equation}
In view of \eqref{eq:constraint-dual-ot} and the maximization over $\psi$, we may as well in \eqref{eq:ot-dual-phi-psi-formulation} take $\psi$ to be the \emph{$c$-transform} of $\phi$ defined by 
\begin{equation}\label{eq:def-c-transform}
    \phi^c(y)=\sup_{x\in \Omegax}\phi(x)-c(x,y).
\end{equation}
This is a continuous function over $\Omegay$ since it inherits the modulus of continuity of $c$ \cite[Box 1.8]{SantambrogioBook}. Then $\psi$ may be eliminated from \eqref{eq:ot-dual-phi-psi-formulation} leading to the alternative dual formulation 
\begin{equation}\label{eq:ot-dual-formulation}
    \Tc_c(\mu,\nu) = \sup_{\phi\in C(\Omegax)} \int_\Omegax \phi\,d\mu - \int_\Omegay \phi^c\,d\nu.
\end{equation}
Since $\Omegax$, $\Omegay$ are compact and $c$ is continuous, not only is there no duality gap so that the formulas \eqref{eq:ot-dual-phi-psi-formulation} and \eqref{eq:ot-dual-formulation} hold, but in fact maximizers always exist \cite[Theorem 1.39]{SantambrogioBook}. Therefore in \eqref{eq:ot-dual-phi-psi-formulation} and \eqref{eq:ot-dual-formulation} we may write ``max'' instead of ``sup''.

\paragraph{Entropic optimal transport.} 
Given $\eps>0$ and two non-null fixed reference measures $\alpha\in\Mp(\Omegax)$, $\beta\in\Mp(\Omegay)$, consider for $\mu\in\Mp(\Omegax)$ and $\nu\in\Mp(\Omegay)$ the entropic optimal transport problem
\begin{equation}\label{eq:reg-ot}
    \ET_{c,\eps}(\mu,\nu) = \inf_{\pi\in\Pi(\mu,\nu)} \int_{\Omegax\times\Omegay} c(x,y)\,d\pi(x,y) + \eps \KL(\pi\mid\alpha\otimes\beta).
\end{equation}
Given $f(s)\coloneqq s\log(s)-s$, the Kullback--Leibler divergence is defined by 
\begin{equation*}
    \KL(m_1 \mid m_2)=\begin{cases}
    \displaystyle\int f\Big(\frac{d m_1}{dm_2}\Big)\,dm_2&\quad\text{if $m_1$ is absolutely continuous w.r.t. $m_2$},\\
    +\infty&\quad\text{otherwise}.
    \end{cases}
\end{equation*}
When $\mu(\Omegax) \neq \nu(\Omegay)$, the infimum in \eqref{eq:reg-ot} runs over an empty set and we set $\ET_{c,\eps}(\mu,\nu)=+\infty$. For an introduction to entropic optimal transport, see \cite{Leonard2013ASO,nutz2021introduction}. 

Similar to the standard optimal transport case, there exist two dual formulations. The first one takes the form 
\begin{equation}\label{eq:eot-dual}
    \ET_{c,\eps}(\mu,\nu) = \sup_{\phi\in C(\Omegax), \psi\in C(\Omegay)} \int_\Omegax \phi\,d\mu - \int_\Omegay \psi\,d\nu - \eps\iint_{\Omegax\times\Omegay} e^{[\phi(x)-\psi(y)-c(x,y)]/\eps}\,d\alpha  d\beta.
\end{equation}
As before, $\psi$ may be eliminated from \eqref{eq:eot-dual} resulting in the second dual formulation
\begin{equation}\label{eq:eot-semidual}
    \ET_{c,\eps}(\mu,\nu) = \sup_{\phi\in C(\Omegax)} \int_\Omegax \phi\,d\mu - \int_\Omegay \Lc_{c,\eps}(\phi)\,d\nu + \eps\KL(\nu \mid \beta).
\end{equation}
Here $\Lc_{c,\eps}\colon C(\Omegax)\to C(\Omegay)$ is the ``soft $c$-transform'' operator defined by
\begin{equation}\label{eq:def-soft-c-transform}
    \Lc_{c,\eps}(\phi)(y) = \eps\log\Big(\int_\Omegax e^{[\phi(x)-c(x,y)]/\eps}\,d\alpha(x)\Big).
\end{equation}
It maps continuous functions to continuous function by the dominated convergence theorem.

\paragraph{Unbalanced optimal transport.}
Standard references for unbalanced optimal transport include 
\cite{LieroMielkeSavare2018,LieroMielkeSavare2016,LieroMielkeSavare2016,ChizatPeyreSchmitzerVialard2018JFA,KondratyevMonsaingeonVorotnikov2016}. 

The unbalanced optimal transport problem replaces the hard marginal constraints on transport plans with soft penalizations by internal energies (\cref{def:internal-energy}).

Let $h_0,h_1\colon[0,+\infty)\to[0,+\infty)$ be two entropy functions (\cref{def:entropy-function}). Here we assume for simplicity that the $h_i$ are finite-valued. Let $H_{h_i,\cdot}$ denote the associated internal energies. Given non-null $\mu\in\Mp(\Omegax)$ and $\nu\in\Mp(\Omegay)$, the unbalanced optimal transport problem is given by 
\begin{equation*}
    \UT_{h,c}(\mu,\nu) = \inf_{\pi \in \mathcal{M}_+(\Omegax\times \Omegay)} H_{h_0,\mu}(\pi_0) + H_{h_1,\nu}(\pi_1) + \int_{\Omegax\times \Omegay} c(x,y) \,d \pi(x,y).
\end{equation*}
Here $\pi_0$ and $\pi_1$ denote respectively the first and second marginals of $\pi$. 

Let $h_i^*\colon \R\to\R$ denote the convex conjugate of $h_i$. It is finite-valued by superlinearity of $h_i$, and nondecreasing since the domain of $h$ contains only nonnegative real numbers.  The unbalanced optimal transport problem then admits the dual formulation
\begin{equation}\label{eq:UOT-dual}
    \UT_{h,c}(\mu,\nu) = \sup_{(\phi,\psi)\in A} \int_\Omegax -h_0^*(-\phi(x))\,d\mu(x)-\int_\Omegay h_1^*(\psi(y))\,d\nu(y).
\end{equation}
Here 
\[
A=\{(\phi,\psi)\in C(\Omegax)\times C(\Omegay) : \forall x\in\Omegax,\forall y\in\Omegay,\quad\phi(x)-\psi(y)\leq c(x,y)\}
\]
denotes the same constraint set as for standard optimal transport. 
Since $h_1^*$ is nondecreasing, $\psi$ can be again eliminated from \eqref{eq:UOT-dual}, yielding the second dual formulation 
\begin{equation*}
    \UT_{h,c}(\mu,\nu) = \sup_{\phi\in C(\Omega)} \int_\Omegax -h_0^*(-\phi(x))\,d\mu(x)-\int_\Omegay h_1^*(\phi^c(y))\,d\nu(y).
\end{equation*}
Here $\phi^c$ denotes the $c$-transform \eqref{eq:def-c-transform}. By \cite[Theorem 4.14]{LieroMielkeSavare2018}, a maximizer $\phi$ always exists.

\subsection{Comparison principles for potentials}\label{sec:ot-cp-potnentials}

\subsubsection{Standard optimal transport}

Given $\mu\in\Pc(\Omegax)$ and $\nu\in\Pc(\Omegay)$, let us denote by $\Phi_c(\mu,\nu)$ the set of solutions to the dual formulation of the optimal transport problem~\eqref{eq:ot-dual-formulation} (the \emph{Kantorovich potentials}),
\begin{equation*}
    \Phi_c(\mu,\nu) = \left\{\phi\in C(\Omegax) : \int_\Omegax \phi\,d\mu - \int_\Omegay \phi^c\,d\nu  = \Tc_c(\mu,\nu)\right\}.
\end{equation*}
Under the assumptions that $c$ be continuous and $\Omegax,\Omegay$ compact this set is nonempty, i.e.\ a solution $\phi$ always exists \cite[Prop. 1.11]{SantambrogioBook}.

The main result of this section is the following comparison principle for Kantorovich potentials.

\begin{theorem}[Comparison principle for Kantorovich potentials]\label{thm:comp-principle-potentials}
    Let $\Omegax,\Omegay$ be two compact metric spaces and let $c\in C(\Omegax\times\Omegay)$. 
    Let $\mu_1,\mu_2$ be two probability measures on $\Omegax$ and let $\nu$ be a probability measure on $\Omegay$. For $i=1,2$, take $\phi_i\in\Phi_c(\mu_i,\nu)$.
    Let $U$ be a Borel subset of $\Omegax$ and suppose that $\mu_1\leq\mu_2$ on $U$ and that $\phi_1\leq\phi_2$ on $\Omegax\setminus U$. Then 
    \begin{equation}\label{eq:thm:comp-principle-potentials}
        \phi_1\wedge\phi_2\in\Phi_c(\mu_1,\nu) 
        \quad\text{and}\quad
        \phi_1\vee\phi_2\in\Phi_c(\mu_2,\nu).
    \end{equation}
    Additionally, $\phi_1 \leq \phi_2$ on the support of $\mu_2-\mu_1$.
\end{theorem}

We recall that $\mu_1\leq \mu_2$ on $U$ means that $\mu_1(A)\leq \mu_2(A)$ for every Borel set $A\subset U$, and that the support of the signed Borel measure $\mu_2-\mu_1$ is defined as the support of its total variation $\abs{\mu_2-\mu_1}$.  

We delay the proof of \cref{thm:comp-principle-potentials} until the end of the section to make some remarks and derive some corollaries.
First, let us discuss the relevance of controlling measures on $U$ and potentials on $\Omega\setminus U$ from a modeling perspective. 
The dual problem of optimal transport can be naturally interpreted as a principal--agent problem. Suppose that a principal wishes to guide a population of agents, distributed as $\nu$, 
toward a distribution of goods, described by a measure $\mu$. To achieve this, the principal sets a price $\phi(x)$ for each item $x$ within the set of goods $U\subset\Omegax$ \emph{that they control}; prices in the outside region $\Omegax\setminus U$ are considered fixed. Then each agent $y\in\Omegay$, acting in their own self-interest, chooses an item $x\in \Omegax$ that maximizes their personal payoff $\phi(x)-c(x,y)$. 
The principal's objective is then to design the price function $\phi$ on $U$ such that the resulting distribution of choices made by the agents is precisely $\mu$ on $U$. 
Therefore in that setting, the complement $\Omegax\setminus U$ may be thought of as the boundary to $U$, and the condition $\phi_1\leq\phi_2$ on $\Omegax\setminus U$ as a boundary condition. See also the case of the Monge--Ampère equation where $
\Omega$ is the closure of the open set $U$, in \Cref{cor:monge-ampere-comp} below.

As a first corollary to \cref{thm:comp-principle-potentials}, taking $\mu_1=\mu_2$ and $U=\Omega$ we obtain that $\Phi_c(\mu,\nu)$ is a sublattice of $C(\Omega)$.

\begin{corollary}
    Let $\mu\in\Pc(\Omegax)$ and $\nu\in\Pc(\Omegay)$. Then $\Phi_c(\mu,\nu)$ is closed under $\wedge$ and $\vee$.
\end{corollary}

Since $\Phi_c(\mu,\nu)$ is also closed under the addition of a constant, two elements $\phi_1,\phi_2\in\Phi_c(\mu,\nu)$ generate a whole class of Kantorovich potentials $(\phi_1+C_1)\wedge(\phi_2+C_2)$, $(\phi_1+C_1)\vee(\phi_2+C_2)$.

Let us now look at the case where at least one of the Kantorovich potentials is unique up to an additive constant: $\Phi_c(\mu_i,\nu)=\{\phi_i+C : C \in \R\}$ for some $\phi_i\in C(\Omegax)$ and $i=1$ or $2$. This situation commonly occurs when transport maps are unique. Here is a typical setting: 
\begin{equation}\label{eq:unique-potential}
    \begin{minipage}[c]{0.85\textwidth}
        $\Omegax,\Omegay$ are compact subsets of $\R^n$, $\Omegax$ is connected and has Lebesgue negligible boundary, $c(x,y)=\abs{x-y}^2$, $\mu_i$ has full support and is absolutely continuous with respect to the Lebesgue measure. 
    \end{minipage}
\end{equation}
When \eqref{eq:unique-potential} holds, $\phi_i$ is unique up to an additive constant \cite[Theorem 1.17]{SantambrogioBook}. The following result shows that in such a case we obtain the standard form of a comparison principle.

\begin{corollary}\label{cor:ot-std-cp}
    Let $\mu_i\in\Pc(\Omegax)$, $\nu\in\Pc(\Omegay)$, and $\phi_i\in\Phi_c(\mu_i,\nu)$. Suppose that at least one of the solutions $\phi_i$ is unique up to an additive constant. Consider a Borel set $U$ with $U\neq \Omega$ such that $\mu_1\leq\mu_2$ on $U$ and $\phi_1\leq\phi_2$ on $\Omega\setminus U$. Then $\phi_1\leq\phi_2$ on $\Omega$. 
\end{corollary}
\begin{proof}
    Suppose that $\phi_1$ is the unique Kantorovich potential. Then $\Phi_c(\mu_1,\nu)=\{\phi_1+C : C\in \R\}$. By \cref{thm:comp-principle-potentials}, we have $\phi_1\wedge\phi_2\in\Phi_c(\mu_1,\nu)$. Therefore $\phi_1\wedge\phi_2=\phi_1+C$ for some constant $C$. Evaluating this equality at some $x\in\Omega\setminus U$ where $(\phi_1\wedge\phi_2)(x)=\phi_1(x)$, we see that $C=0$. Therefore $\phi_1\wedge\phi_2=\phi_1$, which means that $\phi_1\leq\phi_2$. A similar reasoning can be conducted when the unique Kantorovich potential is $\phi_2$. 
\end{proof}

\begin{remark}
    A standard motivation to prove comparison principles is that they imply uniqueness of solutions. Our framework operates differently. When solutions are known to be unique \emph{by other means}, our method provides a comparison principle in the standard form ``$\mu_1\leq\mu_2\implies\phi_1\leq\phi_2$''. When solutions are not unique, \eqref{eq:thm:comp-principle-potentials} still provides a type of ordering of the solution sets $\Phi_c(\mu_1,\nu)$ and $\Phi_c(\mu_2,\nu)$. 
\end{remark}

The next result is a \textbf{maximum principle} for the difference of two Kantorovich potentials.

\begin{corollary}[Maximum principle for Kantorovich potentials]\label{cor:max-principle}
    Let $\mu_i\in\Pc(\Omegax)$ with $\mu_1\neq\mu_2$, $\nu\in\Pc(\Omegay)$, and $\phi_i\in\Phi_c(\mu_i,\nu)$. Suppose that at least one of the solutions $\phi_i$ is unique up to an additive constant. 
    Then 
    \begin{equation}\label{eq:cor:max-principle}
        \max_\Omega(\phi_1-\phi_2)=\max_{\supp(\mu_1-\mu_2)^+}(\phi_1-\phi_2).
    \end{equation}
\end{corollary}

From \cref{sec:cp-basic-setting} and using that $(\mu_1-\mu_2)^+=(\mu_2-\mu_1)^-$, we see that $U\coloneqq \Omega\setminus \supp(\mu_1-\mu_2)^+$ is the largest open set on which $\mu_1\leq\mu_2$. Therefore \eqref{eq:cor:max-principle} states that the maximum of $\phi_1-\phi_2$ is guaranteed to be attained at least on the complement of this set. By exchanging the roles of the indices $1$ and $2$ and multiplying both sides by $-1$, we see that \eqref{eq:cor:max-principle} also implies the minimum principle 
\begin{equation}\label{eq:cor:min-principle}
    \min_\Omega(\phi_1-\phi_2)=\min_{\supp(\mu_2-\mu_1)^+}(\phi_1-\phi_2),
\end{equation}

\begin{proof}[Proof of \cref{cor:max-principle}]
    Since $\mu_2-\mu_1$ has zero total mass but is not the null measure, its negative part $(\mu_2-\mu_1)^-$ is non null. Otherwise, by the Jordan decomposition $\mu_2-\mu_1=(\mu_2-\mu_1)^+ -(\mu_2-\mu_1)^-$, $(\mu_2-\mu_1)^+$ would be a positive measure with zero mass making it null, in turn making $\mu_2-\mu_1$ null. Therefore $(\mu_2-\mu_1)^-$ has nonempty support, and note that $(\mu_2-\mu_1)^-=(\mu_1-\mu_2)^+$. Let $U=\Omega\setminus \supp(\mu_2-\mu_1)^-$, an open and strict subset of $\Omega$. Recall that $U$ is a positive set for $\mu_2-\mu_1$. 

    Since $\phi_1$ and $\phi_2$ are continuous over the compact set $\Omega$, there exists a constant $C\in\R$ such that 
    \[
        \phi_1(x)\leq\phi_2(x)+C \text{ for all $x\in\Omega\setminus U$}.
    \]
    The smallest such constant is $C\coloneqq \sup_{\Omega\setminus U}(\phi_1-\phi_2)$. This ``$\sup$'' is in fact a ``$\max$'' since $\Omega\setminus U$ is compact as a closed subset of $\Omega$, and $\phi_1,\phi_2$ are continuous. Note also that $\phi_2+C\in\Phi_c(\mu_2,\nu)$ since Kantorovich potentials are defined up to a constant. By \cref{cor:ot-std-cp} we find that $\phi_1\leq\phi_2+C$ over all of $\Omega$, i.e.\ $\sup_\Omega(\phi_1-\phi_2)\leq C$. 
\end{proof}

A particular case of interest of the maximum principle is when the unique Kantorovich potential $\phi_i$ is given by a constant function. This situation occurs for instance for the quadratic cost $c(x,y)=\abs{x-y}^2$ on a compact subset of $\R^n$, when one of the source measures coincides with the target measure, and \eqref{eq:unique-potential} holds.

\begin{corollary}[Maximum principle for Kantorovich potentials II]
    Suppose that $\Omegax=\Omegay$. Let $\mu,\nu\in\Pc(\Omega)$ with $\mu\neq\nu$ and let $\phi\in\Phi_c(\mu,\nu)$. Suppose that $\Phi_c(\nu,\nu)$ consists exactly of the constant functions over $\Omega$. 
    Then 
    \[
    \max_\Omega\phi=\max_{\supp(\mu-\nu)^+}\phi,\qquad\text{and}\qquad\min_\Omega\phi=\min_{\supp(\nu-\mu)^+}\phi.
    \]
\end{corollary}
\begin{proof}
    Take $\mu_1=\mu$, $\mu_2=\nu$, $\phi_1=\phi$ and $\phi_2=0$ in \cref{cor:max-principle}. 
\end{proof}

We conclude this series of observations by pointing to \Cref{sec:monge-ampere} for a discussion on how \Cref{thm:comp-principle-potentials} relates to the comparison principle for the Monge--Ampère equation.

We now set out to prove \Cref{thm:comp-principle-potentials}, and \cref{cor:max-principle}. 
Given $\nu\in\Mp(\Omegay)$, let us introduce the dual functional $J_0\colon \Mp(\Omegax)\times C(\Omegax)\to\R$ defined by 
\begin{equation}
    J_0(\mu,\phi)=\int_\Omegay \phi^c\,d\nu - \int_\Omegax\phi\,d\mu.
\end{equation}
We are then interested in minimizing $J_0(\mu,\cdot)$, since $\Tc_c(\mu,\nu)=-\inf_{\phi\in C(\Omegax)} J_0(\mu,\phi)$ and $\Phi_c(\mu,\nu)=\argmin_{\phi\in C(\Omegax)}J_0(\mu,\phi)$. Note that $J_0$ is the sum of a $c$-transform functional which only depends on $\phi$, and a bilinear pairing of $\mu$ and $\phi$. It turns out that the $c$-transform is submodular: this is proven in \cref{lemma:K-submodular} below. Our first lemma is a general result which proves comparison principles in the form given by \cref{thm:comp-principle-potentials} for energies consisting of a submodular term and a bilinear term. We will reuse this blueprint for both the entropic and the unbalanced optimal transport problems, which share the same structure.

\begin{lemma}\label{lemma:vcp-J}
    Let $\Omega$ be a compact metric space and let $K\colon C(\Omega)\to\Rinf$ be a submodular function. Define $J\colon \Mp(\Omega)\times C(\Omega)\to\Rinf$ by 
    \begin{equation*}\label{eq:def-general-J}
    J(\mu,\phi)=K(\phi)-\int_\Omega \phi\,d\mu.
    \end{equation*}
    Given $\mu\in\Mp(\Omega)$, consider the solution set 
    \[
    \Phi(\mu)=\argmin_{\phi\in C(\Omega)} J(\mu,\phi).
    \]
    Let $\mu_1,\mu_2\in\Mp(\Omega)$, and let $\phi_i\in\Phi(\mu_i)$.
    Let $U$ be a Borel subset of $\Omega$ and suppose that $\mu_1\leq\mu_2$ on $U$ and that $\phi_1\leq\phi_2$ on $\Omega\setminus U$. Then 
    \begin{equation}
        \phi_1\wedge\phi_2\in\Phi(\mu_1) 
        \quad\text{and}\quad
        \phi_1\vee\phi_2\in\Phi(\mu_2).
    \end{equation}
    Additionally, $\phi_1 \leq \phi_2$ on the support of $\mu_2-\mu_1$.
\end{lemma}
\begin{proof}
    By submodularity of $K$ we have
    \begin{equation}\label{eq:proof-lemma:vcp-J-1}
    K(\phi_1\wedge\phi_2)+
        K(\phi_1\vee\phi_2)
        \leq 
        K(\phi_1)+
        K(\phi_2).
    \end{equation}
    Note that some of these terms may be infinite. Using the identities $(\phi_1\wedge\phi_2)-\phi_1=-(\phi_1-\phi_2)^+$ and $(\phi_1\vee\phi_2)-\phi_2=(\phi_1-\phi_2)^+$, we have 
    \begin{equation}\label{eq:proof-lemma:vcp-J-2}
    \int_\Omega\phi_1\,d\mu_1+\int_\Omega\phi_2\,d\mu_2=\int_\Omega(\phi_1\wedge\phi_2)\,d\mu_1+\int_\Omega(\phi_1\vee\phi_2)\,d\mu_2-\int_\Omega(\phi_1-\phi_2)^+\,d(\mu_2-\mu_1).
    \end{equation}
    Decomposing $\mu_2-\mu_1=(\mu_2-\mu_1)^+ -(\mu_2-\mu_1)^-$, we observe that $\int_\Omegax (\phi_1-\phi_2)^+\,d(\mu_2-\mu_1)^-= 0$ since $ \mu_2-\mu_1 \geq 0$ on $U$ and $(\phi_1 - \phi_2)^+ = 0$ on $\Omegax\setminus U$. Therefore  
    \begin{equation}\label{eq:proof-lemma:vcp-J-3}
    \int_\Omegax (\phi_1-\phi_2)^+\,d(\mu_2-\mu_1)
    =\int_\Omegax (\phi_1-\phi_2)^+\,d\abs{\mu_2-\mu_1}.
    \end{equation}
    Combining \eqref{eq:proof-lemma:vcp-J-1}, \eqref{eq:proof-lemma:vcp-J-2} and \eqref{eq:proof-lemma:vcp-J-3} we obtain 
    \begin{equation*}
         J(\mu_1,\phi_1\wedge\phi_2)+
        J(\mu_2,\phi_1\vee\phi_2)
        +\int_\Omegax (\phi_1-\phi_2)^+\,d\abs{\mu_2-\mu_1}
        \leq 
    J(\mu_1,\phi_1)+
        J(\mu_2,\phi_2)    .
    \end{equation*}    
    Since $\phi_i$ realizes the minimum of $J(\mu_i,\cdot)$, we directly deduce: 
    \begin{gather*}
        \phi_1\wedge\phi_2\in \argmin J(\mu_1,\cdot), \\
        \phi_1\vee\phi_2\in \argmin J(\mu_2,\cdot),\\ 
        \int_\Omegax (\phi_1-\phi_2)^+\,d\abs{\mu_2-\mu_1} =0.    
    \end{gather*}
    We conclude that $\phi_1\wedge\phi_2\in\Phi(\mu_1)$, $\phi_1\vee\phi_2\in\Phi(\mu_2)$, and that $(\phi_1-\phi_2)^+$ vanishes on the support of $\abs{\mu_2-\mu_1}$. 
\end{proof}

We continue with an elementary but powerful result from Crandall and Tartar on order-preserving mappings \cite{CrandallTartar1980}. We restrict their setting to ours, only state part of their result, and produce a variant of their proof below.

\begin{lemma}[{\cite[Prop. 2]{CrandallTartar1980}}]\label{lemma:Crandall-Tartar}
    Consider a mapping $T\colon C(\Omegax)\to C(\Omegay)$ which satisfies
    \begin{enumerate}[(a)]
        \item $\phi_1\leq\phi_2$ implies $T(\phi_1)\leq T(\phi_2)$ for all $\phi_1,\phi_2\in C(\Omega)$;
        \item $T(\phi+C)=T(\phi)+C$ for all $\phi\in C(\Omega)$ and  $C\in\R$.
    \end{enumerate}
    Then $T$ is $1$-Lipschitz: $\norm{T(\phi_1)-T(\phi_2)}_\infty \leq \norm{\phi_1-\phi_2}_\infty$. 
\end{lemma}
\begin{proof}
    From $\phi_1\leq\phi_1\vee\phi_2$ we deduce $T(\phi_1)\leq T(\phi_1\vee\phi_2)$, and exchanging $\phi_1$ and $\phi_2$ gives $T(\phi_2)\leq T(\phi_1\vee\phi_2)$. Therefore 
    \begin{equation*}
        T(\phi_1)\vee T(\phi_2)\leq T(\phi_1\vee\phi_2).
    \end{equation*}
    A similar reasoning with $\wedge$ instead of $\vee$ gives 
    \begin{equation*}
        T(\phi_1)\wedge T(\phi_2)\geq T(\phi_1\wedge\phi_2).
    \end{equation*}
    By the identity $\abs{f_1-f_2}=(f_1\vee f_2)-(f_1\wedge f_2)$ we obtain 
    \begin{equation}\label{eq:proof-lemma:Crandall-Tartar-1}
        \abs{T(\phi_1)-T(\phi_2)}\leq T(\phi_1\vee\phi_2)-T(\phi_1\wedge\phi_2).
    \end{equation}
    By the same identity $\phi_1\vee\phi_2-\phi_1\wedge\phi_2=\abs{\phi_1-\phi_2}$ we deduce 
    \begin{equation*}
        \phi_1\vee\phi_2\leq\phi_1\wedge\phi_2+\norm{\phi_1-\phi_2}_\infty.
    \end{equation*}
    Using hypotheses (a) and (b) we find 
    \begin{equation}\label{eq:proof-lemma:Crandall-Tartar-2}
        T(\phi_1\vee\phi_2)-T(\phi_1\wedge\phi_2)\leq \norm{\phi_1-\phi_2}_\infty.
    \end{equation}
    Combining \eqref{eq:proof-lemma:Crandall-Tartar-1} and \eqref{eq:proof-lemma:Crandall-Tartar-2} allows us to conclude.
\end{proof}

Coming back to the optimal transport setting, we are now ready to prove that the $c$-transform is submodular. Given $\nu\in\Pc(\Omegay)$, we define $K_0\colon C(\Omegax)\to\R$ by 
\[
K_0(\phi)=\int_\Omegay  \phi^c\,d\nu.
\]

\begin{lemma}\label{lemma:K-submodular}
    $K_0$ is submodular, convex, and  $1$-Lipschitz.
\end{lemma}
\begin{proof}
    Given $x\in \Omegax$ and $y\in\Omegay$, let us introduce the set 
    \begin{equation}
        A_{xy}=\{(\phi,\psi)\in C(\Omegax)\times C(\Omegay) : \phi(x)-\psi(y)\leq c(x,y)\}.
    \end{equation}
    $A_{xy}$ is clearly a closed convex subset of $C(\Omegax)\times C(\Omegay)$. Additionally, given $(\phi_1,\psi_1), (\phi_2,\psi_2)\in A_{xy}$ we have $\min(\phi_1(x),\phi_2(x))\leq \phi_1(x)\leq \psi_1(y)+c(x,y)$ and $\min(\phi_1(x),\phi_2(x))\leq \psi_2(y)+c(x,y)$, so that 
    \begin{equation}
        \min(\phi_1(x),\phi_2(x)) \leq \min(\psi_1(y),\psi_2(y)) +c(x,y).
    \end{equation}
    In other words, $(\phi_1\wedge\phi_2,\psi_1\wedge\psi_2)\in A_{xy}$. 
    Similarly, $\max(\phi_1(x),\phi_2(x)) \leq \max(\psi_1(y),\psi_2(y)) +c(x,y)$ which shows that $(\phi_1\vee\phi_2,\psi_1\vee\psi_2)\in A_{xy}$ . Therefore $A_{xy}$ is a lattice. 

    Define now $A=\bigcap_{x\in \Omegax,y\in\Omegay}A_{xy}$, i.e.\ 
    \begin{equation*}
        A = \{(\phi,\psi)\in C(\Omegax)\times C(\Omegay) : \forall x\in \Omegax,y\in\Omegay,\quad \phi(x)-\psi(y)\leq c(x,y)\}.
    \end{equation*}
    $A$ is a closed convex lattice subset of $C(\Omegax)\times C(\Omegay)$ as an intersection of such sets. As a consequence, its indicator function $\iota_A$ taking the value $0$ on $A$ and $+\infty$ outside of $A$ is a submodular convex lower semicontinuous function over $C(\Omegax)\times C(\Omegay)$.

    Given $\nu\in \Mp(\Omegay)$, define $E(\phi,\psi)=\int_\Omegay \psi\,d\nu+\iota_A(\phi,\psi)$. $E$ is a jointly submodular convex function as a sum of such functions. Then 
    \[
    K_0(\phi)=\int_\Omegay\phi^c\,d\nu=\inf_{\psi\in C(\Omegay)} E(\phi,\psi)
    \]
    is convex, and submodular by \Cref{lemma:P-order-min}.

    It remains to show that $K_0$ is $1$-Lipschitz. The $c$-transform $\phi\mapsto \phi^c$ with $\phi^c(y)=\sup_{x\in\Omega}\phi(x)-c(x,y)$ is clearly order-preserving and satisfies $(\phi+C)^c=\phi^c+C$ for any constant $C$. The Crandall--Tartar lemma (\cref{lemma:Crandall-Tartar}) then gives 
    \[
    \abs{\phi_1^c(y)-\phi_2^c(y)} \leq \Vert \phi_1 - \phi_2 \Vert_\infty.
    \]
    The desired result follows after integrating against $\nu$.
\end{proof}

With the preceding results in hand, the proof of \cref{thm:comp-principle-potentials} is a direct combination of \cref{lemma:vcp-J} and \cref{lemma:K-submodular}. Note that the convexity and continuity of $K$ proved in \cref{lemma:K-submodular} are not needed to apply \cref{lemma:vcp-J}.

\subsubsection{Link with the Monge--Ampère equation}\label{sec:monge-ampere}

In this section, we show that  \cref{thm:comp-principle-potentials} recovers the comparison principle for the Monge--Ampère equation (\cref{cor:monge-ampere-comp}). This comparison principle is a standard result which plays an important role in the theory of these nonlinear PDEs. We refer to the book \cite{FigalliMongeAmpereBook} for some background on the Monge--Ampère equation, and in particular to \cite[Theorem 2.10]{FigalliMongeAmpereBook} for the comparison principle.

Let us recall some standard notions of convex analysis used in the theory of Monge--Ampère equations. Given a nonempty bounded and merely open set $U\subset\R^n$, we say that $u\in C(\Ub)$ is convex if it can be extended into a function from 
$\R^n$ to $\R\cup\{+\infty\}$
which is convex in the classical sense \cite{Rockafellar1970}; in particular the extended convex function is proper and can be chosen lower semicontinuous.
The convex conjugate of $u$ is the function $u^*\colon\R^n\to\R$ defined by 
\begin{equation}\label{eq:def-convex-conjugate}
    u^*(y)=\sup_{x\in \Ub}\bracket{x,y}-u(x).
\end{equation}
Note that $u^*$ takes only finite values since $u\in C(\Ub)$ and $\Ub$ is compact. For the same reasons, ``$\sup$'' in \eqref{eq:def-convex-conjugate} can be replaced by ``$\max$''. The subdifferential of $u$ at $x\in \Ub$ is defined by 
\[
\partial u(x)=\{y\in\R^n : \forall z\in \Ub, u(x)+\bracket{x-z,y}\leq u(z)\}.
\]
Note that
\begin{equation}\label{eq:subdiff-and-argmin}
    y\in\partial u(x) \iff x\in\argmin_{z\in\Ub} u(z)-\bracket{z,y}.
\end{equation}
With these ingredients in hand, we may define the \emph{Monge--Ampère measure} associated to a convex $u\in C(\Ub)$. This is a locally finite positive Borel measure $\theta_u$ on $U$, defined by 
\[
\theta_u(E)=\Leb(\partial u(E)),
\]
where $\Leb$ denotes the Lebesgue measure on $\R^n$ and $\partial u(E)=\bigcup_{x\in E}\partial u(x)$. 

The main ingredient relating \cref{thm:comp-principle-potentials} to the Monge--Ampère equation is expressing the Monge--Ampère measure as a convex mapping of the Lebesgue measure (\cref{prop:MAM-convex-mapping}). To derive this result we start with some elementary facts. 
The first one is standard in convex analysis, we give a proof for completeness.

\begin{lemma}\label{lemma:subdiff-standard-inclusion}
    Let $x\in \Ub$ and $y\in \partial u(x)$. Then $x\in \partial u^*(y)$.
\end{lemma}
\begin{proof}
    Fix $x\in \Ub$ and $y\in \partial u(x)$. By definition of the subdifferential, we have $u(x)+u^*(y)\leq\bracket{x,y}$. By definition of the convex conjugate, we have $\bracket{x,y'}\leq u(x)+u^*(y')$ for all $y'\in\R^n$. Summing both inequalities gives $u^*(y)+\bracket{x,y'-y}\leq u^*(y')$ for all $y'\in\R^n$, or in other words $x\in\partial u^*(y)$. 
\end{proof}

Let us denote by $D(u^*)\subset\R^n$ the set of points where $u^*$ is differentiable. Since $u^*$ is a convex function over $\R^n$ taking finite values, $\R^n\setminus D(u^*)$ has zero Lebesgue measure \cite[Theorem 25.5]{Rockafellar1970}. Recall also that $\partial u^*(y)=\{\nabla u^*(y)\}$ whenever $y\in D(u^*)$, where $\nabla u^*$ denotes the gradient of $u^*$ \cite[Theorem 25.1]{Rockafellar1970}.

\begin{lemma}\label{lemma:nabla-ustar-unique-argmin}
    For every $y\in D(u^*)$, the function $(z\in\Ub)\mapsto u(z)-\bracket{z,y}$ attains its minimum at a unique point $x\in \Ub$, given by $x=\nabla u^*(y)$.
\end{lemma}
\begin{proof}
    Let $y\in\R^n$. Combining \eqref{eq:subdiff-and-argmin} and \cref{lemma:subdiff-standard-inclusion} we find that 
    \begin{equation}
        \argmin_{z\in\Ub} u(z)-\bracket{z,y}\subset\partial u^*(y).
    \end{equation}
    Suppose that $u^*$ is differentiable at $y$. On the one hand the $\argmin$ is nonempty since $u-\bracket{\cdot,y}$ is continuous on the compact set $\Ub$. On the other hand $\partial u^*(y)=\{\nabla u^*(y)\}$. We conclude that $\argmin_{z\in\Ub} u(z)-\bracket{z,y}=\{\nabla u^*(y)\}$. 
\end{proof}

Using \cref{lemma:nabla-ustar-unique-argmin} we obtain a partial converse to \cref{lemma:subdiff-standard-inclusion}.

\begin{lemma}\label{lemma:converse-subdiff-inclusion}
    Let $y\in D(u^*)$ and let $x=\nabla u^*(y)$. Then $y\in \partial u(x)$. 
\end{lemma}
\begin{proof}
    Let $y\in D(u^*)$. By \cref{lemma:nabla-ustar-unique-argmin}, $x\coloneqq \nabla u^*(y)\in\argmin_{z\in\Ub} u(z)-\bracket{z,y}$. Thus $y\in\partial u(x)$ by \eqref{eq:subdiff-and-argmin}. 
\end{proof}

Combining \cref{lemma:subdiff-standard-inclusion} and \cref{lemma:converse-subdiff-inclusion} directly gives 

\begin{lemma}\label{lemma:subdiff-nabla-ustar}
    For any Borel set $E\subset \Ub$, $(\nabla u^*)^{-1}(E)=\partial u(E)\cap D(u^*)$. 
\end{lemma}

By \cref{lemma:nabla-ustar-unique-argmin}, $\nabla u^*$ defines a mapping from $D(u^*)$ to $\Ub$. Since the complement of $D(u^*)$ in $\R^n$ has zero Lebesgue measure, we can consider the measure 
$
(\nabla u^*)_\#\Leb,
$
which is a positive Borel measure on the closure $\Ub$. 

\begin{proposition}[Monge--Ampère measure as a convex mapping]\label{prop:MAM-convex-mapping}
    The Monge--Ampère measure of $u$ coincides with 
    $(\nabla u^*)_\#\Leb$ on $U$, in the sense that $\theta_u(E)=[(\nabla u^*)_\#\Leb](E)$ for every Borel set $E\subset U$. 
\end{proposition}
\begin{proof}
    Fix a Borel set $E\subset U$. We have $\theta_u(E)=\Leb(\partial u(E))$ and $[(\nabla u^*)_\#\Leb](E)=\Leb((\nabla u^*)^{-1}(E))$. 
    Since the complement of $D(u^*)$ is a null set, $\Leb(\partial u(E))=\Leb(\partial u(E)\cap D(u^*))$.
    \cref{lemma:subdiff-nabla-ustar} then allows us to conclude.
\end{proof}

Since we work on compact sets it will be interesting to know when the Lebesgue measure in the above characterization of the Monge--Ampère measure can be restricted to a bounded set without losing any mass in $U$ or a subdomain of $U$.

\begin{lemma}\label{lemma:restriction-Lebesgue}
    Let $V\subset U$ and $B\subset\R^n$ denote two Borel sets
    such that $\partial u(V)\subset B$. Then the measures $(\nabla u^*)_\#\Leb$ and $(\nabla u^*)_\#(\Leb|_B)$ coincide on $V$, where $\Leb|_B$ denotes the restriction of the Lebesgue measure to $B$. 
\end{lemma}
\begin{proof}
    \cref{lemma:subdiff-nabla-ustar} shows that $(\nabla u^*)^{-1}(V)\subset B$, which suffices to prove the claim.
\end{proof}

\begin{proposition}[Comparison principle for the Monge--Ampère equation]\label{cor:monge-ampere-comp}
    Let $U$ be a bounded open subset of $\R^n$ and let $u_1,u_2\in C(\Ub)$ denote two convex functions. Let $\theta_i$ denote the Monge--Ampère measure associated with $u_i$, $i=1,2$. Suppose that $u_2\leq u_1$ on $\partial U$ and that $\theta_1\leq \theta_2$ on $U$. Then $u_2\leq u_1$ on $U$.
\end{proposition}

Here is the informal idea of the proof: since $u_i$ is convex, $\phi_i\coloneqq \abs{x}^2/2-u_i$ is a ``Kantorovich potential'' for the quadratic cost optimal transport between the measures $(\nabla u_i^*)_\#\Leb$ on $\Ub$ and $\Leb$ on $\R^n$, and using that on $U$: $\theta_1=(\nabla u_1^*)_\#\Leb\leq \theta_2=(\nabla u_2^*)_\#\Leb$, and on $\partial U$: $\phi_1\leq\phi_2$, we deduce from our comparison principle in \cref{thm:comp-principle-potentials} that $\phi_1\leq\phi_2$ on $U$, yielding the desired result. However $\Leb$ has infinite mass, making the optimal transport problem not well-defined, and moreover our comparison principle is stated on compact spaces, making the above argument only formal. If $\partial u_1(U)$ and $\partial u_2(U)$ happen to be bounded subsets of $\R^n$ then the above argument can be made rigorous, but in general we need to reason by approximation.

\begin{proof}[Proof of \cref{cor:monge-ampere-comp}]
    Let $k\geq 1$ be a natural number. By compactness of $\Ub$, $u_1$ and $u_2$ are uniformly continuous over $\Ub$. Therefore there exists $\delta_k>0$ such that $\abs{u_i(x)-u_i(z)} \leq 1/k$ for every $x,z\in \Ub$ such that $\abs{x-z}\leq\delta_k$, and for $i=1,2$. 
    Let $\Omegax=\Ub$, $\phi_{1,k}=\abs{x}^2/2-u_1$ and $\phi_{2,k}=\abs{x}^2/2-u_2+2/k$, and $U_k = \{x\in U : \dist(x,\partial U)> \delta_k\}$. Then $-u_1\leq -u_2$ on $\partial U$ implies $-u_1-1/k\leq -u_2+1/k$ on $\Omegax\setminus U_k$, therefore
    \begin{equation}\label{eq:proof-cor:monge-ampere-comp-1}
    \phi_{1,k}\leq\phi_{2,k}\quad\text{on }\Omegax\setminus U_k.
    \end{equation}
    Let $B_k\subset \R^n$ denote the closed ball of radius $R_k$ centered at $0$, with 
    \[
    R_k= \frac{\max\big\{\norm{u_1}_{L^\infty(\Ub)}, \norm{u_2}_{L^\infty(\Ub)}\big\}}{\delta_k}.
    \]
    Let $i=1$ or $2$. By a standard estimate, $\partial u_i(U_k)\subset B_k$ (see e.g.\ \cite[Lemma A.22]{FigalliMongeAmpereBook}). Let $\nu_k$ denote the restriction of the Lebesgue measure to $B_k$. 
    % The preceding remarks show that $\mu_{i,k}\coloneqq (\nabla u_i^*)_\# \nu_k$ coincides on $U_k$ with $(\nabla u_i^*)_\#\Leb$, and therefore $\theta_i$. Therefore $\mu_{1,k}\leq\mu_{2,k}$ on $U_k$. 
    By \cref{lemma:restriction-Lebesgue}, $(\nabla u_i^*)_\#\Leb$ coincides with $\mu_{i,k}\coloneqq (\nabla u_i^*)_\# \nu_k$ on $U_k$. Then by \cref{prop:MAM-convex-mapping}, $\mu_{i,k}$ coincides with $\theta_i$ on $U_k\subset U$. Then $\theta_1\leq \theta_2$ implies 
    \begin{equation}\label{eq:proof-cor:monge-ampere-comp-2}
        \mu_{1,k}\leq\mu_{2,k}\quad\text{on } U_k.
    \end{equation}

    We have now reached a position where the measures $\mu_{i,k}$, $\nu_k$ are finite and the domains $\Omegax=\Ub$, $\Omegay\coloneqq B_k$ are compact. 
    Since $\nu_k$ is absolutely continuous and $u_i^*$ is convex, by Brenier's theorem $\nabla u_i^*$ is the unique optimal transport map from $\nu_k$ to $\mu_{i,k}$, see e.g.\ \cite[Theorem 1.17]{SantambrogioBook}; moreover since $B_k$ is connected, $\nabla u_i$ uniquely determines $u_i$ up an additive constant implying that $\phi_{i,k}$ is the unique Kantorovich potential up to an additive constant.  

    We may therefore use \cref{thm:comp-principle-potentials}, which is stated for probability measures but trivially extend to finite positive measures: combining \eqref{eq:proof-cor:monge-ampere-comp-1} and \eqref{eq:proof-cor:monge-ampere-comp-2} gives $\phi_{1,k}\leq \phi_{2,k}$ on $U_k$, i.e.\ 
    \[
    u_2\leq u_1+2/k\quad\text{on } U_k.
    \]
    Letting $k\to\infty$ gives us the desired result.

\end{proof}

\subsubsection{Entropic optimal transport}

In this section we prove comparison principles for entropic potentials by following the same strategy established for standard optimal transport.

Given $\eps>0$, two non-null reference measures $\alpha \in \Mp(\Omegax)$ and $\beta \in \Mp(\Omegay)$, and two probability measures $\mu\in\Pc(\Omegax)$ and $\nu\in \Pc(\Omegay)$, let us denote the set of solutions to the dual entropic transport problem by 
\begin{equation}
    \Phi_{c,\eps}(\mu,\nu)=\left\{\phi\in C(\Omegax) : \int_\Omegax \phi\,d\mu - \int_\Omegay\Lc_{c,\eps}(\phi)\,d\nu +\eps \KL(\nu \mid \beta)= \ET_{c,\eps}(\mu,\nu)\right\}.
\end{equation}
The soft $c$-transform $\Lc_{c,\eps}$ is defined in \eqref{eq:def-soft-c-transform}. This notation suppresses the dependence on $\alpha$ and $\beta$, which are considered fixed throughout. 

\begin{theorem}[Comparison principle for entropic potentials]\label{thm:entropic-ot-comp}
    Let $\Omegax,\Omegay$ be two compact metric spaces and let $c\in C(\Omegax\times\Omegay)$. 
    Let $\mu_1,\mu_2$ be two probability measures on $\Omegax$ and $\nu$ a probability measure on $\Omegay$. For $i=1,2$, let $\phi_i\in\Phi_{c,\eps}(\mu_i,\nu)$.
    Let $U$ be a Borel subset of $\Omegax$ and suppose that $\mu_1\leq\mu_2$ on $U$ and that $\phi_1\leq\phi_2$ on $\Omegax\setminus U$. 
    Then 
    \begin{equation}\label{eq:thm:comp-principle-eot-potentials}
        \phi_1\wedge\phi_2\in\Phi_{c,\eps}(\mu_1,\nu) 
        \quad\text{and}\quad
        \phi_1\vee\phi_2\in\Phi_{c,\eps}(\mu_2,\nu).
    \end{equation}
    Additionally, $\phi_1 \leq \phi_2$ on the support of $\mu_2-\mu_1$.
\end{theorem}

To prove \cref{thm:entropic-ot-comp}, we follow the same approach used in the previous section. Given $\nu\in\Mp(\Omegay)$, we introduce the functionals $K_\eps\colon C(\Omega)\to\R$ and $J_\eps\colon \Mp(\Omegax)\times C(\Omegax)\to\R$ defined by 
\begin{equation*}
    K_\eps(\phi)=\int_\Omegay \Lc_{c,\eps}(\phi)\,d\nu, 
\end{equation*}
\begin{equation*}
    J_\eps(\mu,\phi)=K_\eps(\phi)- \int_\Omegax\phi\,d\mu.
\end{equation*}
In particular note that 
\begin{equation*}
    \ET_{c,\eps}(\mu,\nu)=-\inf_{\phi\in C(\Omegax)} J_\eps(\mu,\phi),
\end{equation*}
\begin{equation*}    
    \Phi_{c,\eps}(\mu,\nu)=\argmin_{\phi\in C(\Omegax)}J_\eps(\mu,\phi).
\end{equation*}
Therefore \cref{thm:entropic-ot-comp} directly follows from \cref{lemma:vcp-J}, provided that $K_\eps$ is subdmodular.

\begin{lemma}\label{lemma:K-eps-submodular}
    $K_\eps$ is submodular, convex and $1$-Lipschitz.
\end{lemma}
\begin{proof}
    Let $x\in \Omegax$ and $y \in \Omegay$. By \cref{lemma:convexdiff_submodular}, convexity of the exponential implies submodularity of the function $(a,b)\mapsto e^{[a-b-c(x,y)]/\eps}$ over $\R\times\R$. As a consequence, we have for any $\phi_1,\phi_2 \in C(\Omegax)$ and $\psi_1,\psi_2 \in C(\Omegay)$ that
    \begin{multline*}
        e^{[\phi_1(x)\wedge\phi_2(x)-\psi_1(y)\wedge\psi_2(y)-c(x,y)]/\eps} + e^{[\phi_1(x)\vee\phi_2(x)-\psi_1(y)\vee\psi_2(y)-c(x,y)]/\eps}\\
        \leq e^{[\phi_1(x)-\psi_1(y)-c(x,y)]/\eps} + e^{[\phi_2(x)-\psi_2(y)-c(x,y)]/\eps}.
    \end{multline*}
    Integrating this inequality with respect to $\alpha$ and $\nu$ we obtain that 
    \begin{equation*}
    (\phi,\psi)\mapsto\iint_{\Omegax\times\Omegay} e^{[\phi(x)-\psi(y)-c(x,y)]/\eps} d\alpha(x)d\nu(y)
    \end{equation*}
    is a jointly submodular convex function. Therefore the function $E\colon C(\Omegax)\times C(\Omegay)\to\R$ defined by 
    \begin{equation*}
    E(\phi,\psi)=\int_\Omegay \psi\,d\nu+\eps\iint_{\Omegax\times\Omegay} e^{[\phi(x)-\psi(y)-c(x,y)]/\eps} d\alpha(x)d\nu(y)
    \end{equation*}
    is also jointly submodular and convex, as a sum of such functions. Notice that 
    \[
    \inf_{\psi\in C(\Omegay)} E(\phi,\psi)=\int_\Omegay\Lc_{c,\eps}(\phi)\,d\nu=K_\eps(\phi). 
    \]
    $K_\eps$ is therefore convex and, using \cref{lemma:P-order-min}, submodular.

    It remains to show that $K_\eps$ is $1$-Lipschitz. $\Lc_{c,\eps}$ is clearly order-preserving and satisfies $\Lc_{c,\eps}(\phi+C)=\Lc_{c,\eps}(\phi)+C$ for any constant $C$. The Crandall--Tartar lemma (\cref{lemma:Crandall-Tartar}) then allows us to conclude, as in the proof of \cref{lemma:K-submodular}.
\end{proof}

We conclude this section by pointing out that a similar maximum principle as the ones described in \cref{cor:max-principle} can be derived, since $\phi\mapsto J_\eps(\mu,\phi)$ is invariant under the addition of a constant.

\subsubsection{Unbalanced optimal transport}

Let $\Omegax,\Omegay$ be two compact metric spaces, let $c\in C(\Omegax\times\Omegay)$ and consider two finite-valued entropy functions $h_0, h_1\colon [0,+\infty)\to[0,+\infty)$. Given two positive measures $\mu\in\Mp(\Omegax)$ and $\nu\in\Mp(\Omegay)$, we denote the set of solutions to the dual unbalanced optimal transport problem by 
\begin{equation}
    \Phi_{h,c}(\mu,\nu)=\Big\{\phi\in C(\Omegax) : \int_\Omegax -h_0^*(-\phi(x))\,d\mu-\int_\Omegay h_1^*(\phi^c(y))\,d\nu=\UT_{h,c}(\mu,\nu)\Big\}.
\end{equation}
Under our assumptions,  $\Phi_{h,c}(\mu,\nu)$ is nonempty \cite[Theorem 4.14]{LieroMielkeSavare2018}.

\begin{theorem}[Comparison principle for unbalanced potentials]\label{thm:uot-pot}
    Let $\mu_i\in \Mp(\Omegax)$, $\nu\in\Mp(\Omegay)$, and $\phi_i\in\Phi_{h,c}(\mu_i,\nu)$.
    Let $U$ be a Borel subset of $\Omegax$ and suppose that $\mu_1\leq\mu_2$ on $U$ and that $\phi_1\leq\phi_2$ on $\Omegax\setminus U$. 
    Then 
    \begin{equation}\label{eq:thm:comp-principle-uot-potentials}
        \phi_1\wedge\phi_2\in\Phi_{h,c}(\mu_1,\nu) 
        \quad\text{and}\quad
        \phi_1\vee\phi_2\in\Phi_{h,c}(\mu_2,\nu).
    \end{equation}
    Additionally, $-h_0^*(-\phi_1(x)) \leq -h_0^*(-\phi_2(x))$ for $x\in\supp(\mu_2-\mu_1)$.
\end{theorem}

Before giving the proof of \cref{thm:uot-pot} let us state a consequence specific to the unbalanced problem. Since now $\mu_1$ and $\mu_2$ may have different total masses, the case $\mu_1\leq\mu_2$ over all of $\Omegax$ becomes more interesting. The following result is a direct consequence of \cref{thm:uot-pot} in the case of a perturbation with full support.

\begin{corollary}\label{cor:uot-pot}
    Suppose that the function $s\mapsto h_0^*(s)$ is strictly increasing. Let $\eta\in\Mp(\Omegax)$ be a positive measure with full support. Let $\mu_1\in\Mp(\Omegax)$, $\nu\in\Mp(\Omegay)$, and define $\mu_2=\mu_1+\eta$. Take $\phi_i\in\Phi_c(\mu_i,\nu)$. Then $\phi_1\leq\phi_2$ over $\Omegax$.
\end{corollary}

The conclusion of \cref{cor:uot-pot} gives a comparison between the solution sets $\Phi_{h,c}(\mu_1,\nu)$ and $\Phi_{h,c}(\mu_2,\nu)$ which is very strong: every element of the first set is smaller than every element of the second set. We now proceed with the proof of the theorem. 

\begin{proof}[Proof of \cref{thm:uot-pot}]
    The argument is a variant of the one used for the standard and the entropic optimal transport problems, and is based on the following observation: whenever $f\colon \R\to\R$ is a nondecreasing function, we have for any functions $\zeta_1,\zeta_2$ that 
    \begin{equation}\label{eq:monotone-meet-join}
        \begin{aligned}
        f\big(\zeta_1(x)\wedge\zeta_2(x)\big)&=f\big(\zeta_1(x)\big)\wedge f\big(\zeta_2(x)\big),\\
        f\big(\zeta_1(x)\vee\zeta_2(x)\big)&=f\big(\zeta_1(x)\big)\vee f\big(\zeta_2(x)\big).    
        \end{aligned}        
    \end{equation}
    Define $K_{h_1}\colon C(\Omegax)\to\R$ by 
    \begin{equation*}
        K_{h_1}(\phi)=\int_\Omegay h_1^*(\phi^c(y))\,d\nu.    
    \end{equation*}
    Then $K_{h_1}(\phi)=\inf_{\psi\in C(\Omegay)} \int_\Omegay h_1^*(\psi(y))\,d\nu + \iota_A(\phi,\psi)$, where $A$ denotes the same set as in the proof of \cref{lemma:K-submodular}. Using \eqref{eq:monotone-meet-join} we find that $E$ is jointly submodular, hence $K_{h_1}$ is submodular. Define now $J_h\colon \Mp(\Omegax)\times C(\Omegax)\to\R$ by     
    \begin{equation*}
    J_h(\mu,\phi)=K_{h_1}(\phi)+\int_\Omegax h_0^*(-\phi(x))\,d\mu.      
    \end{equation*}
    Given $\mu_1,\mu_2\in \Mp(\Omegax)$ and $\phi_1,\phi_2\in C(\Omegax)$, following the same arguments as in the proof of \cref{lemma:vcp-J}, and using \eqref{eq:monotone-meet-join} again, we have that 
    \begin{multline*}
        J_h(\mu_1,\phi_1\wedge\phi_2)+
        J_h(\mu_2,\phi_1\vee\phi_2)
        +\int_\Omegax \big(-h_0^*(-\phi_1(x))+h_0^*(-\phi_2(x))\big)^+\,d\abs{\mu_2-\mu_1}\\
        \leq 
    J_h(\mu_1,\phi_1)+
        J_h(\mu_2,\phi_2)    .
    \end{multline*}
    The desired result immediately follows.
\end{proof}

\subsection{Comparison principles for JKO-type problems} \label{sec:cp-ot}

Let $\Omegax$ and $\Omegay$ be two compact metric spaces. The main result of this section (\Cref{thm:ot-cp}) is a comparison principle for ``proximal'' variational problems of the form: given $\mu\in\Mp(\Omegax)$, solve 
\begin{equation}\label{eq:ot-min-problem}
    \min_{\nu\in\Mp(\Omegay)} \Tc(\mu,\nu)+E(\nu).
\end{equation}
Here, $\Tc$ is a cost between positive measures which can be chosen among a broad class of transport costs. Specifically, we consider the three types introduced in \cref{sec:cp-basic-setting}:
\begin{enumerate}[label=Type (\Roman*):,labelwidth =\widthof{(Type III):}, leftmargin = !]
    \item The standard optimal transport cost $\Tc(\mu,\nu)=\Tc_c(\mu,\nu)$. Here $c$ is a continuous function over $\Omegax\times\Omegay$.
    
    \item The entropic optimal transport cost $\Tc(\mu,\nu)=\ET_{c,\eps}(\mu,\nu)$. Here $c\in C(\Omegax\times\Omegay)$ and $\eps>0$.
    
    \item The unbalanced optimal transport cost $\Tc(\mu,\nu)=\UT_{h,c}(\mu,\nu)$. Here $c\in C(\Omegax\times\Omegay)$ and $h_0,h_1\colon[0,+\infty)\to[0,\infty)$ are two finite-valued entropy functions.
\end{enumerate}
The energies $E$ we consider are convex internal energies to which can be added a potential energy. More precisely, let $f\colon [0,+\infty)\to[0,+\infty)$ be an entropy function (proper, lower semicontinuous, convex and superlinear, see \cref{def:entropy-function}) which we assume in addition finite and strictly convex, and let $m\in \Mp(\Omegay)$ denote a non-null reference measure. Consider the associated internal energy $H_{f,m}$ (\cref{def:internal-energy}). We fix $V\in C(\Omegay)$, and define the function $E\colon \Mp(\Omegay)\to\R\cup\{+\infty\}$ by
\begin{equation}\label{eq:def-F-ot-cp}
    E(\nu)=H_{f,m}(\nu) + \int_\Omegay V(y)\,d\nu(y).
\end{equation}

Variational problems of the form \eqref{eq:ot-min-problem} appear in minimizing movement schemes such as the Jordan--Kinderlehrer--Otto (JKO) scheme \cite{Jordan1998}, for which the transport term is given by the squared Wasserstein distance. Variants of the JKO scheme with a modified version of the entropic cost $\Tc_{c,\eps}$ have also been recently studied in \cite{baradat2025usingsinkhornjkoscheme}. For this reason we will refer to \eqref{eq:ot-min-problem} as the \textbf{JKO problem}.

\begin{lemma}\label{lemma:unique-solution}
    Let $\Tc(\mu,\nu)$ denote a transport cost of type (I)--(III), and let $E$ be defined by \eqref{eq:def-F-ot-cp}. Then given $\mu\in \Mp(\Omegax)$, the JKO problem \eqref{eq:ot-min-problem} admits a unique solution
    \begin{equation}
    \nu\coloneqq \JKO(\mu)\in\Mp(\Omegay).
\end{equation}
\end{lemma}
Since the proof of \cref{lemma:unique-solution} uses standard arguments from calculus of variations, we defer it until the end of this section.  
The main result of this section then reads

\begin{theorem}[Comparison principle for JKO problems]\label{thm:ot-cp} 
    Let $\Tc(\mu,\nu)$ denote a transport cost of type (I)--(III), and let $E$ be defined by \eqref{eq:def-F-ot-cp}. 
    Let $\mu_i\in \Mp(\Omegax)$ for $i=1,2$ and let $\nu_i=\JKO(\mu_i)$. Then 
    \begin{equation}\label{eq:thm:ot-cp}
        \mu_1\leq\mu_2 \implies \nu_1\leq\nu_2.
    \end{equation}
\end{theorem}

We recall that $\mu_1\leq\mu_2$ means that $\mu_1(A)\leq\mu_2(A)$ for every Borel set $A\subset \Omegax$, and similarly for the statement $\nu_1\leq\nu_2$. We postpone the proof of \cref{thm:ot-cp} until the end of the section to make some observations and draw some consequences.

For the Type (I) transport cost $\Tc=\Tc_c$, a similar result was obtained by Jacobs, Kim and Tong \cite{Jacobs2020TheP} using regularity of the cost in order to define a transport map associated with the transport problem. They then use the optimality conditions associated with the JKO scheme and the optimal map to derive an $L^1$ contraction principle similar to \cref{cor:tv-contraction}. In order for the optimal transport map for the cost to exist their approach requires the cost to be $C^1_{\mathrm{loc}}$ and twisted. In the present theorem no such assumptions are required on the cost as the comparison principle is a consequence of the structure of the problem rather than the cost itself. The present argument also bypasses technicalities by not referring to any underlying PDE.

We now derive from \cref{thm:ot-cp} a maximum principle in the spatially homogeneous case $V(y)=0$. 
\begin{equation}\label{eq:stationary-setting}
    \begin{minipage}[c]{0.85\textwidth}
        $\Omegax=\Omegay$, $V(y)=0$, $\Tc(\mu,\nu)=\Tc_c(\mu,\nu)$ is nonnegative and vanishes when  $\mu=\nu$.
    \end{minipage}
\end{equation}
A simple setting where the assumptions on $\Tc_c$ are satisfied is when $c(x,y)$ is nonnegative and vanishes on the diagonal $x=y$. The next lemma shows that under \eqref{eq:stationary-setting} the measure $\mu= C dm$, where $C$ is a positive constant, is stationary for the JKO problem.

\begin{lemma}\label{lemma:Cdm-stationary}
    Suppose that \eqref{eq:stationary-setting} holds. Let $C> 0$ be a constant. Then $\JKO(C dm)=C dm$.
\end{lemma}
\begin{proof} 
    Let us start by showing that for any $\mu,\nu\in\Mp(\Omega)$, 
    \begin{equation}\label{eq:proof-lemma:Cdm-stationary-1}
        \Tc_c(\mu,\nu)+H_{f,m}(\nu) \geq m(\Omega)f\big(\mu(\Omega)/m(\Omega)\big). 
    \end{equation}
    Firstly, by Jensen's inequality \cite[Eq. (2.44)]{LieroMielkeSavare2018} we have 
    \begin{equation}\label{eq:proof-lemma:Cdm-stationary-2}
        H_{f,m}(\nu) \geq m(\Omega)f\big(\nu(\Omega)/m(\Omega)\big).
    \end{equation}
    Secondly, under the hypotheses of the lemma we have $\Tc_c(\mu,\nu)\geq 0$. Adding this inequality to \eqref{eq:proof-lemma:Cdm-stationary-2} we obtain \eqref{eq:proof-lemma:Cdm-stationary-1} after noticing that we may replace $\nu(\Omega)$ by $\mu(\Omega)$ in \eqref{eq:proof-lemma:Cdm-stationary-2} since $\Tc_c(\mu,\nu)=+\infty$ when $\mu(\Omega)\neq\nu(\Omega)$.

    Take now $\mu=Cdm$ for some constant $C>0$. By \eqref{eq:proof-lemma:Cdm-stationary-1} we have $\Tc_c(\mu,\nu)+H_{f,m}(\nu) \geq m(\Omega)f(C)$, and $\nu=\mu$ produces equality. The claim follows.
\end{proof}

As an immediate corollary to \cref{thm:ot-cp} and \cref{lemma:Cdm-stationary}, we obtain a maximum principle.

\begin{corollary}[Maximum principle for JKO problems]\label{cor:max-principle-jko}
    Suppose that \eqref{eq:stationary-setting} holds. Let $\mu\in\Mp(\Omega)$ and let $\nu=\JKO(\mu)$. Let $C_0, C_1$ denote two positive constants. Then 
    \begin{equation}\label{eq:cor:max-principle-jko}
        C_0\leq \frac{d\mu}{dm}\leq C_1 \implies C_0\leq \frac{d\nu}{dm}\leq C_1.
    \end{equation}    
\end{corollary}

In particular, \eqref{eq:cor:max-principle-jko} implies the $L^\infty$ bound 
\begin{equation*}
    \norm{d\nu / dm}_\infty \leq \norm{d\mu / dm}_\infty.
\end{equation*}
(Here we abuse notation and allow the $L^\infty$ norms to be infinite). Of particular interest is the case where $\Omega$ is a compact subset of $\R^n$ and $m$ is the restriction of the Lebesgue measure to $\Omega$.

We now turn our attention to a different consequence of the comparison principle: contraction in the total variation (TV) norm, called the $L^1$-contraction principle in \cite{Jacobs2020TheP}.
In the standard and entropic transport cases (I) and (II), a simple property of the costs is that they force the mapping $\JKO$ to be mass-preserving, in the sense that the solution $\nu$ for initial data $\mu$ satisfies $\nu(\Omegay)=\mu(\Omegax)$. By a simple but powerful result of Crandall and Tartar \cite{CrandallTartar1980} this fact, combined with the order-preserving property of $\JKO$, implies a contraction in total variation norm $\norm{\cdot}_{\mathrm{TV}}$. 

\begin{corollary}[TV contraction]\label{cor:tv-contraction}
    Let $\Tc(\mu,\nu)$ denote a transport cost of type (I) or (II). 
    Let $\mu_i\in \Mp(\Omegax)$ and $\nu_i=\JKO(\mu_i)$. Then 
    \begin{equation}
        \norm{\nu_1-\nu_2}_{\mathrm{TV}} \leq \norm{\mu_1-\mu_2}_{\mathrm{TV}}.
    \end{equation}
\end{corollary}
\begin{proof}
    This is a direct consequence of the Crandall--Tartar result contained in \cref{lemma:Crandall-Tartar-II}. 
\end{proof}

The TV contraction is interesting in and of itself; in \cite{Jacobs2020TheP} it was put forward to derive bounds in bounded variation under a translation-invariance assumption.

We now set out to prove \cref{thm:ot-cp}. The main ingredient to prove the comparison principle is a substitutability property of the transport costs.

\begin{lemma}\label{lemma:Tc-exchangeable}
    The function $F\colon (\eta,\tau)\mapsto \Tc_c(\eta,-\tau)$ is jointly substitutable over $\Mp(\Omegax)\times\Mm(\Omegay)$. 
\end{lemma}
\begin{proof}
    Define 
    \begin{equation*}
        A = \{(\phi,\psi)\in C(\Omegax)\times C(\Omegay) : \forall x\in \Omegax,\forall y\in\Omegay,\quad \phi(x)-\psi(y)\leq c(x,y)\}.
    \end{equation*}
    We established in the proof of \Cref{lemma:K-submodular} that $A\subset C(\Omegax)\times C(\Omegay)$ is a closed convex lattice and therefore that $\iota_A$ is a lower semicontinuous convex submodular function over $C(\Omegax)\times C(\Omegay)$.
    Let $F\colon \Mc(\Omegax)\times \Mc(\Omegay)\to\R\cup\{+\infty\}$ denote the convex conjugate of $\iota_A$,
    \[
    F(\eta,\tau)=\sup_{(\phi,\psi)\in A}\int_\Omegax\phi\,d\eta + \int_\Omegay \psi\,d\tau.
    \]
    By \Cref{cor:submodular-exchangeable}, $F$ is substitutable. Note that here $\eta$ and $\tau$ are signed measures. However, if $\eta$ is not a positive measure then there exists an open set $U\subset\Omegax$ such that $\eta(U)<0$ (by outer regularity). A function $\phi$ admissible in $A$ may then take arbitrarily small values on $U$ resulting in $F(\eta,\tau)=+\infty$. In other words, $F(\eta,\tau)<+\infty$ only when $\eta$ is a positive measure. Similarly, $F(\eta,\tau)<+\infty$ only when $\tau$ is a negative measure. 
    
    The domain of $F$ is therefore included in  $\Mp(\Omegax)\times\Mm(\Omegay)$ and we may as well restrict $F$ to this set. 
    Note also that adding the same constant to both $\phi$ and $\psi$ leaves the constraint set invariant, so that $F(\eta,\tau)$ is finite only when $\eta(\Omegax)+\tau(\Omegay)=0$.
\end{proof}

\begin{lemma}\label{lemma:exchangeability-entropic-cost}
    Let $\eps>0$ and $\alpha \in \Mp(\Omegax),\beta \in \Mp(\Omegay)$ be two nonvanishing reference measures. Then the function $F\colon (\eta,\tau)\mapsto \ET_{c,\eps}(\eta,-\tau)$ is jointly substitutable over $\Mp(\Omegax)\times\Mm(\Omegay)$. 
\end{lemma}
\begin{proof}
    Following the proof of \Cref{lemma:K-eps-submodular} we have that 
    \begin{equation*}
        (\phi,\psi) \mapsto \eps \int e^{[\phi(x)-\psi(y)-c(x,y)]/\eps} d\alpha(x)d\beta(y)
    \end{equation*}
    is submodular and convex over $C(\Omegax)\times C(\Omegay)$. It is also lower semicontinuous by positivity of the exponential using Fatou's lemma. Thus \Cref{thm:duality-submodular-exch} ensures that 
    \begin{equation*}
       \ET_{c,\eps}(\eta,-\tau) = F(\eta,\tau) = \sup_{(\phi,\psi)\in C(\Omegax)\times C(\Omegay)} \int \phi d\eta + \int \psi d\tau - \eps \int e^{[\phi(x)-\psi(y)-c(x,y)]/\eps} d\alpha(x)d\beta(y)
    \end{equation*}
    is substitutable. The proof is then the same as in \Cref{lemma:Tc-exchangeable} and the domain of $F$ is still included in $\Mp(\Omegax)\times\Mm(\Omegay)$.
\end{proof}

\begin{lemma}\label{lemma:exchangeability-unbalanced-cost}
    Let $h_0,h_1\colon [0,+\infty) \to [0,+\infty)$ be two entropy functions. Then the function $F\colon (\eta,\tau)\mapsto \UT_{h,c}(\eta,-\tau)$ is jointly substitutable over $\Mp(\Omegax)\times\Mm(\Omegay)$. 
\end{lemma}

\begin{proof}
    By definition of $\UT_{h,c}$ we have for $(\eta,\tau) \in \Mp(\Omegax)\times\Mm(\Omegay)$,
    \begin{equation*}
        F(\eta,-\tau) = \sup_{(\phi,\psi)\in A} \int_\Omegax -h_0^*(-\phi(x))\,d\eta(x) + \int_\Omegay h_1^*(\psi(y))\,d\tau(y),
    \end{equation*}
    where $A$ is the same set as in the proof of \Cref{lemma:Tc-exchangeable} and as such is a lattice. $F$ is almost the Legendre-Fenchel transform of the submodular convex function $\iota_A$, however the scalar product has been replaced by $\int_\Omegax -h_0^*(-\phi)\,d\eta + \int_\Omegay h_1^*(\psi)\,d\tau$. Let $(\eta_1,\tau_1),(\eta_2,\tau_2) \in \Mp(\Omegax)\times\Mm(\Omegay)$, that we denote by $m_i = (\eta_i,\tau_i)$. Let $(\phi_1,\psi_1),(\phi_2,\psi_2) \in A$ and set $f_i = (-h_0^\ast(-\phi_i),h_1^\ast(\psi_i))$. The Banach lattices $\Mc(\Omegax)\times\Mc(\Omegay)$ and $C(\Omegax)\times C(\Omegay)$ are in duality which allows for the use of \cref{lemma:duality-bracket}. Let $t_{21} \in [0,(m_2-m_1)^+]$ by \cref{lemma:duality-bracket} there is $t_{12} \in [0,(m_2-m_1)^-]$ such that
    \begin{equation*}
        \langle m_1 + t_{21}-t_{12},f_1 \rangle + \langle m_2 - t_{21}+t_{12},f_2 \rangle \leq \langle m_1, f_1 \wedge f_2\rangle + \langle m_2, f_1 \vee f_2 \rangle.
    \end{equation*}
    Since $h^\ast_0,h^\ast_1$ are non decreasing we have $f_1 \wedge f_2 = (-h_0^\ast(-(\phi_1\wedge \phi_2)),h_1^\ast(\psi_1\wedge \psi_2))$ and similarly for $f_1 \vee f_2$. Thus since $A$ is a lattice we obtain
    \begin{equation*}
        \langle m_1 + t_{21}-t_{12},f_1 \rangle + \langle m_2 - t_{21}+t_{12},f_2 \rangle \leq F(m_1) + F(m_2).
    \end{equation*}
    We will conclude in a similar fashion to \cref{thm:duality-submodular-exch} using a minimax theorem.
    Observe that $((\eta,\tau),(\phi,\psi)) \mapsto \langle (\eta,\tau) , (-h_0^\ast(-\phi),h_1^\ast(\psi))\rangle - \iota_A((\phi,\psi))$ is convex in $(\eta,\tau)$ and concave in $(\phi,\psi)$ by convexity of $h_i^\ast$ and because $\tau$ is nonpositive. Thus by \cite[Theorem 3.1]{SimonsMinimaxBook} we have the result.
\end{proof}
The proof of \Cref{thm:ot-cp} is now a direct combination of some facts established in \Cref{sec:submodularity-exchangeability}.

\begin{proof}[Proof of \Cref{thm:ot-cp}]
    
    Let $\mu_1\leq\mu_2 \in \Mp(\Omegax)$.
    Define $
    F(\mu,\eta)=\Tc(\mu,-\eta)$ on $\Mp(\Omegax)\times \Mm(\Omegay)$. By \Cref{lemma:Tc-exchangeable,lemma:exchangeability-entropic-cost,lemma:exchangeability-unbalanced-cost}, $F$ is jointly substitutable. Therefore $F(\mu_2,\cdot) \leQ F(\mu_1,\cdot)$ by \Cref{lemma:Q-order-joint-exch}. Due to the change of sign on the second variable it is clear that we obtain $\Tc(\mu_1,\cdot) \leQ \Tc(\mu_2,\cdot)$.  Moreover $E$ is totally substitutable by \Cref{lemma:internal_energy}. Then \Cref{prop:sum-total-exch} gives $\Tc(\mu_1,\cdot) + E(\cdot)\leQ \Tc(\mu_2,\cdot)+ E(\cdot)$. Using \Cref{prop:argmin-exch} we obtain $\argmin_\nu \Tc(\mu_1,\cdot) + E(\cdot)\leQ \argmin \Tc(\mu_2,\cdot) + E(\cdot)$. Since the minimizers are unique by \cref{lemma:unique-solution} we have by \Cref{lemma:Q-singletons} that $\JKO(\mu_1) \leq \JKO (\mu_2)$.
\end{proof}

% A similar result was obtained in \cite{Jacobs2020TheP} using regularity of the cost in order to be able to define a transport map associated with the transport problem. In the present theorem no such assumptions are required on the cost as the comparison principle is a consequence of the structure of the problem rather than the cost itself. The present argument also bypasses technicality by not referring to any underlying PDE.

\begin{proof}[Proof of \cref{lemma:unique-solution}]
    In order to prove the existence and uniqueness of the minimizers to \eqref{eq:ot-min-problem} we will need to detail some properties of $E$ defined by \eqref{eq:def-F-ot-cp}. First the strict convexity of $f$ ensures the strict convexity of $E$. Second we claim that the sublevel sets of $E$ are compact. They are closed by lower semicontinuity of $f$. Now let $t \in \mathbb{R}$ be such that $\{\nu \mid E(\nu) \leq t\}$ is non empty. Take $\nu$ such that $E(\nu) \leq t$, then by Jensen's inequality, $f\left(\frac{\nu(\Omegay)}{m(\Omegay)}\right) m(\Omegay)-\Vert V \Vert_\infty \nu(\Omegay) \leq t$. Since $f$ is superlinear we get that $\nu(\Omegay) \leq C$ for some constant $C$. Since the sublevel set of $E$ is contained in a set of measures with bounded mass, it is compact by Prokhorov's theorem. 
    
    Let $\mu \in \Mp(\Omegax)$. We now prove the existence and uniqueness of the minimizers to \eqref{eq:ot-min-problem}. Let $\nu_n \in \Mp(\Omegay)$ be a minimizing sequence. Since the sublevel sets of $E$ are compact and $\Tc$ is bounded from below, up to a not relabeled subsequence, $(\nu_n)_n$ converges to some $\nu \in \Mp(\Omegay)$. By lower semicontinuity of $\Tc$ and $E$, we have that $\nu$ is in fact a solution to \eqref{eq:ot-min-problem}. Uniqueness follows from the strict convexity of $E$ and the convexity of $\Tc$. 
\end{proof}

We conclude the section with another result of Crandall and Tartar, cousin to \cref{lemma:Crandall-Tartar}, which shows that an order-preserving function which conserves mass is a contraction. Again we slightly adapt their setting to ours and include a variant of their proof.

\begin{lemma}[{\cite[Prop. 1]{CrandallTartar1980}}]\label{lemma:Crandall-Tartar-II}
    Consider a mapping $T\colon \Mp(\Omegax)\to \Mp(\Omegay)$ which satisfies
    \begin{enumerate}[(a)]
        \item $\mu_1\leq\mu_2$ implies $T(\mu_1)\leq T(\mu_2)$ for all $\mu_1,\mu_2\in \Mp(\Omega)$;
        \item $\int_\Omegay T(\mu)=\int_\Omegax \mu$ for all $\mu\in\Mp(\Omegax)$.
    \end{enumerate}
    Then $T$ is $1$-Lipschitz in total variation norm: $\norm{T(\mu_1)-T(\mu_2)}_{\mathrm{TV}} \leq \norm{\mu_1-\mu_2}_{\mathrm{TV}}$.     
\end{lemma}
\begin{proof}
    Since $T$ is order-preserving, by following the same steps as in the proof of \cref{lemma:Crandall-Tartar} we obtain  that 
    \begin{equation*}
        \abs{T(\mu_1)-T(\mu_2)}\leq T(\mu_1\vee\mu_2)-T(\mu_1\wedge\mu_2).
    \end{equation*}
    Note that $\abs{\cdot}$ now denotes the total variation of a measure; here the advantage of the lattice formalism is apparent since the same computations can be carried out for continuous functions or measures. Integrating the measures over $\Omegay$, we obtain 
    \[
    \norm{T(\mu_1)-T(\mu_2)}_{\mathrm{TV}}=\int_\Omegay\abs{T(\mu_1)-T(\mu_2)} \leq \int_\Omegay T(\mu_1\vee\mu_2)-\int_\Omegay T(\mu_1\wedge\mu_2).
    \]
    Using hypothesis (b) we deduce that 
    \begin{equation*}
        \int_\Omegay T(\mu_1\vee\mu_2)-\int_\Omegay T(\mu_1\wedge\mu_2) = \int_\Omegax (\mu_1\vee\mu_2)-(\mu_1\wedge\mu_2)=\int_\Omegax \abs{\mu_1-\mu_2}.
    \end{equation*}
    This proves the desired inequality.
\end{proof}

\printbibliography

@article{PerthameQuirosVazquez2014,
  author       = {Perthame, Benoit and Quiros, Fernando and Vazquez, Juan Luis},
  title        = {The Hele-Shaw asymptotics for mechanical models of tumor growth},
  journaltitle = {Archive for Rational Mechanics and Analysis},
  shortjournal = {Arch. Ration. Mech. Anal.},
  volume       = {212},
  year         = {2014},
  number       = {1},
  pages        = {93--127},
  issn         = {0003-9527,1432-0673},
  doi          = {10.1007/s00205-013-0704-y},
  url          = {https://doi.org/10.1007/s00205-013-0704-y},
}

@misc{zhu2024inclusiveklminimizationwassersteinfisherrao,
      title={Inclusive KL Minimization: A Wasserstein-Fisher-Rao Gradient Flow Perspective}, 
      author={Jia-Jie Zhu},
      year={2024},
      eprint={2411.00214},
      archivePrefix={arXiv},
      primaryClass={stat.ML},
      url={https://arxiv.org/abs/2411.00214}, 
}

@article {MR4804827,
    AUTHOR = {Yan, Yuling and Wang, Kaizheng and Rigollet, Philippe},
     TITLE = {Learning {G}aussian mixtures using the
              {W}asserstein-{F}isher-{R}ao gradient flow},
   JOURNAL = {Ann. Statist.},
  FJOURNAL = {The Annals of Statistics},
    VOLUME = {52},
      YEAR = {2024},
    NUMBER = {4},
     PAGES = {1774--1795},
      ISSN = {0090-5364,2168-8966},
   MRCLASS = {62G05 (62H30)},
  MRNUMBER = {4804827},
MRREVIEWER = {Sreenivasan\ Ravi},
       DOI = {10.1214/24-aos2416},
       URL = {https://doi.org/10.1214/24-aos2416},
}

@article{ChizatPeyreSchmitzerVialard2018FoCM,
  title = {An Interpolating Distance Between Optimal Transport and {F}isher--{R}ao Metrics},
author = {Chizat, L{\'e}naic and Peyr{\'e}, Gabriel and Schmitzer, Bernhard and Vialard, Fran{\c{c}}ois-Xavier},
  year = 2018,
  month = feb,
  journal = {Foundations of Computational Mathematics},
  volume = {18},
  number = {1},
  pages = {1--44},
  publisher = {{Springer Science and Business Media LLC}},
  issn = {1615-3375, 1615-3383},
  doi = {10.1007/s10208-016-9331-y},
  urldate = {2025-07-16},
  copyright = {http://www.springer.com/tdm},
  langid = {english}
}

@article{ChizatPeyreSchmitzerVialard2018JFA,
  title = {Unbalanced Optimal Transport: {{Dynamic}} and {{Kantorovich}} Formulations},
  shorttitle = {Unbalanced Optimal Transport},
  author = {Chizat, L{\'e}naic and Peyr{\'e}, Gabriel and Schmitzer, Bernhard and Vialard, Fran{\c{c}}ois-Xavier},
  year = 2018,
  month = jun,
  journal = {Journal of Functional Analysis},
  volume = {274},
  number = {11},
  pages = {3090--3123},
  publisher = {Elsevier BV},
  issn = {0022-1236},
  doi = {10.1016/j.jfa.2018.03.008},
  urldate = {2025-07-22},
  copyright = {https://www.elsevier.com/tdm/userlicense/1.0/},
  langid = {english}
}

@article{LieroMielkeSavare2016,
author = {Liero, Matthias and Mielke, Alexander and Savar\'{e}, Giuseppe},
title = {Optimal Transport in Competition with Reaction: The {H}ellinger--{K}antorovich Distance and Geodesic Curves},
journal = {SIAM Journal on Mathematical Analysis},
volume = {48},
number = {4},
pages = {2869-2911},
year = {2016},
doi = {10.1137/15M1041420},
URL = {https://doi.org/10.1137/15M1041420},
eprint = {https://doi.org/10.1137/15M1041420}
}

@misc{rankin2024jkoschemesgeneraltransport,
      title={{JKO} schemes with general transport costs}, 
      author={Cale Rankin and Ting-Kam Leonard Wong},
      year={2024},
      eprint={2402.17681},
      archivePrefix={arXiv},
      primaryClass={math.AP},
      url={https://arxiv.org/abs/2402.17681}, 
}

@misc{agarwal2024iteratedschrodingerbridgeapproximation,
      title={Iterated {S}chr\"odinger bridge approximation to {W}asserstein Gradient Flows}, 
      author={Medha Agarwal and Zaid Harchaoui and Garrett Mulcahy and Soumik Pal},
      year={2024},
      eprint={2406.10823},
      archivePrefix={arXiv},
      primaryClass={math.PR},
      url={https://arxiv.org/abs/2406.10823}, 
}

@misc{hardion2025gradientflowspotentialenergies,
      title={Gradient Flows of Potential Energies in the Geometry of {S}inkhorn Divergences}, 
      author={Mathis Hardion and Hugo Lavenant},
      year={2025},
      eprint={2511.14278},
      archivePrefix={arXiv},
      primaryClass={math.AP},
      url={https://arxiv.org/abs/2511.14278}, 
}

@article {CarlierDuvalPeyreSchmitzer2017,
    AUTHOR = {Carlier, Guillaume and Duval, Vincent and Peyr\'e, Gabriel and
              Schmitzer, Bernhard},
     TITLE = {Convergence of entropic schemes for optimal transport and
              gradient flows},
   JOURNAL = {SIAM J. Math. Anal.},
  FJOURNAL = {SIAM Journal on Mathematical Analysis},
    VOLUME = {49},
      YEAR = {2017},
    NUMBER = {2},
     PAGES = {1385--1418},
      ISSN = {0036-1410,1095-7154},
   MRCLASS = {49Q20 (65K05 90C25)},
  MRNUMBER = {3635459},
MRREVIEWER = {Kathrin\ Welker},
       DOI = {10.1137/15M1050264},
       URL = {https://doi.org/10.1137/15M1050264},
}

@article {Peyre2015,
    AUTHOR = {Peyr\'e, Gabriel},
     TITLE = {Entropic approximation of {W}asserstein gradient flows},
   JOURNAL = {SIAM J. Imaging Sci.},
  FJOURNAL = {SIAM Journal on Imaging Sciences},
    VOLUME = {8},
      YEAR = {2015},
    NUMBER = {4},
     PAGES = {2323--2351},
      ISSN = {1936-4954},
   MRCLASS = {94A08 (49Q20 68U10 90C25 90C48)},
  MRNUMBER = {3413589},
       DOI = {10.1137/15M1010087},
       URL = {https://doi.org/10.1137/15M1010087},
}

@article {Agueh2005,
    AUTHOR = {Agueh, Martial},
     TITLE = {Existence of solutions to degenerate parabolic equations via
              the {M}onge-{K}antorovich theory},
   JOURNAL = {Adv. Differential Equations},
  FJOURNAL = {Advances in Differential Equations},
    VOLUME = {10},
      YEAR = {2005},
    NUMBER = {3},
     PAGES = {309--360},
      ISSN = {1079-9389},
   MRCLASS = {35K55 (35K60 35K65)},
  MRNUMBER = {2123134},
MRREVIEWER = {Ning\ Su},
}

@article {AlexanderKimYao2014,
    AUTHOR = {Alexander, Damon and Kim, Inwon and Yao, Yao},
     TITLE = {Quasi-static evolution and congested crowd transport},
   JOURNAL = {Nonlinearity},
  FJOURNAL = {Nonlinearity},
    VOLUME = {27},
      YEAR = {2014},
    NUMBER = {4},
     PAGES = {823--858},
      ISSN = {0951-7715,1361-6544},
   MRCLASS = {35Q35 (35B36 35R35 49J45)},
  MRNUMBER = {3190322},
       DOI = {10.1088/0951-7715/27/4/823},
       URL = {https://doi.org/10.1088/0951-7715/27/4/823},
}

@incollection {Lovasz1983,
    AUTHOR = {Lov\'asz, L.},
     TITLE = {Submodular functions and convexity},
 BOOKTITLE = {Mathematical programming: the state of the art ({B}onn, 1982)},
     PAGES = {235--257},
 PUBLISHER = {Springer, Berlin},
      YEAR = {1983},
      ISBN = {3-540-12082-3},
   MRCLASS = {90C27 (05B35 05C30 49D99)},
  MRNUMBER = {717403},
}

@book {FujishigeBook2005,
    AUTHOR = {Fujishige, Satoru},
     TITLE = {Submodular functions and optimization},
    SERIES = {Annals of Discrete Mathematics},
    VOLUME = {58},
   EDITION = {Second},
 PUBLISHER = {Elsevier B. V., Amsterdam},
      YEAR = {2005},
     PAGES = {xiv+395},
      ISBN = {0-444-52086-4},
   MRCLASS = {90C10 (05A18 05B35 52A41 90C27 90C35 90C57)},
  MRNUMBER = {2171629},
}

@book {TopkisBook1998,
    AUTHOR = {Topkis, Donald M.},
     TITLE = {Supermodularity and complementarity},
    SERIES = {Frontiers of Economic Research},
 PUBLISHER = {Princeton University Press, Princeton, NJ},
      YEAR = {1998},
     PAGES = {xii+272},
      ISBN = {0-691-03244-0},
   MRCLASS = {90A11 (90D10 90D12)},
  MRNUMBER = {1614637},
MRREVIEWER = {Hans\ Peters},
}

@article{MilgromShannon1994,
  title={Monotone comparative statics},
  author={Milgrom, Paul and Shannon, Chris},
  journal={Econometrica: Journal of the Econometric Society},
  pages={157--180},
  year={1994},
  publisher={JSTOR}
}

@article {KondratyevMonsaingeonVorotnikov2016,
    AUTHOR = {Kondratyev, Stanislav and Monsaingeon, L\'eonard and
              Vorotnikov, Dmitry},
     TITLE = {A new optimal transport distance on the space of finite
              {R}adon measures},
   JOURNAL = {Adv. Differential Equations},
  FJOURNAL = {Advances in Differential Equations},
    VOLUME = {21},
      YEAR = {2016},
    NUMBER = {11-12},
     PAGES = {1117--1164},
      ISSN = {1079-9389},
   MRCLASS = {49Q20 (28A33 35L60 35Q92 58B20)},
  MRNUMBER = {3556762},
MRREVIEWER = {Giandomenico\ Orlandi},
       URL = {http://projecteuclid.org/euclid.ade/1476369298},
}

@book {BogachevBook2007,
    AUTHOR = {Bogachev, V. I.},
     TITLE = {Measure theory. {V}ol. {I}, {II}},
 PUBLISHER = {Springer-Verlag, Berlin},
      YEAR = {2007},
     PAGES = {Vol. I: xviii+500 pp., Vol. II: xiv+575},
      ISBN = {978-3-540-34513-8; 3-540-34513-2},
   MRCLASS = {28-02 (28Axx 28Cxx 46G12 60G42 60G44)},
  MRNUMBER = {2267655},
MRREVIEWER = {Ren\'e\ L.\ Schilling},
       DOI = {10.1007/978-3-540-34514-5},
       URL = {https://doi.org/10.1007/978-3-540-34514-5},
}

@article {CrandallTartar1980,
    AUTHOR = {Crandall, Michael G. and Tartar, Luc},
     TITLE = {Some relations between nonexpansive and order preserving
              mappings},
   JOURNAL = {Proc. Amer. Math. Soc.},
  FJOURNAL = {Proceedings of the American Mathematical Society},
    VOLUME = {78},
      YEAR = {1980},
    NUMBER = {3},
     PAGES = {385--390},
      ISSN = {0002-9939,1088-6826},
   MRCLASS = {47H07},
  MRNUMBER = {553381},
MRREVIEWER = {A.\ J. B. Potter},
       DOI = {10.2307/2042330},
       URL = {https://doi.org/10.2307/2042330},
}

@book {AGS2008,
    AUTHOR = {Ambrosio, Luigi and Gigli, Nicola and Savar\'e, Giuseppe},
     TITLE = {Gradient flows in metric spaces and in the space of
              probability measures},
    SERIES = {Lectures in Mathematics ETH Z\"urich},
   EDITION = {Second},
 PUBLISHER = {Birkh\"auser Verlag, Basel},
      YEAR = {2008},
     PAGES = {x+334},
      ISBN = {978-3-7643-8721-1},
   MRCLASS = {49-02 (28A33 35K55 35K90 49Q20 60B05)},
  MRNUMBER = {2401600},
MRREVIEWER = {Pietro\ Celada},
}

@article{LieroMielkeSavare2018,
  title = {Optimal Entropy-Transport Problems and a New {{Hellinger}}--{{Kantorovich}} Distance between Positive Measures},
  author = {Liero, Matthias and Mielke, Alexander and Savar{\'e}, Giuseppe},
  year = {2018},
  month = mar,
  journal = {Inventiones mathematicae},
  volume = {211},
  number = {3},
  pages = {969--1117},
  publisher = {{Springer Science and Business Media LLC}},
  issn = {0020-9910, 1432-1297},
  doi = {10.1007/s00222-017-0759-8},
  urldate = {2025-07-16},
  copyright = {http://creativecommons.org/licenses/by/4.0},
  langid = {english}
}

@book {MeyerNiebergBook1991,
    AUTHOR = {Meyer-Nieberg, Peter},
     TITLE = {Banach lattices},
    SERIES = {Universitext},
 PUBLISHER = {Springer-Verlag, Berlin},
      YEAR = {1991},
     PAGES = {xvi+395},
      ISBN = {3-540-54201-9},
   MRCLASS = {46B42 (46A40 47B60)},
  MRNUMBER = {1128093},
MRREVIEWER = {Yu.\ A.\ Abramovich},
       DOI = {10.1007/978-3-642-76724-1},
       URL = {https://doi.org/10.1007/978-3-642-76724-1},
}

@book {SchaeferBook1974,
    AUTHOR = {Schaefer, Helmut H.},
     TITLE = {Banach lattices and positive operators},
    SERIES = {Die Grundlehren der mathematischen Wissenschaften},
    VOLUME = {Band 215},
 PUBLISHER = {Springer-Verlag, New York-Heidelberg},
      YEAR = {1974},
     PAGES = {xi+376},
   MRCLASS = {46A40 (47B55 47D20)},
  MRNUMBER = {423039},
MRREVIEWER = {A.\ C.\ Zaanen},
}

@book {Rockafellar1970,
    AUTHOR = {Rockafellar, R. Tyrrell},
     TITLE = {Convex analysis},
    SERIES = {Princeton Mathematical Series},
    VOLUME = {No. 28},
 PUBLISHER = {Princeton University Press, Princeton, NJ},
      YEAR = {1970},
     PAGES = {xviii+451},
   MRCLASS = {26.52 (46.00)},
  MRNUMBER = {274683},
MRREVIEWER = {Ky\ Fan},
}

@book {FigalliMongeAmpereBook,
    AUTHOR = {Figalli, Alessio},
     TITLE = {The {M}onge-{A}mp\`ere equation and its applications},
    SERIES = {Zurich Lectures in Advanced Mathematics},
 PUBLISHER = {European Mathematical Society (EMS), Z\"urich},
      YEAR = {2017},
     PAGES = {x+200},
      ISBN = {978-3-03719-170-5},
   MRCLASS = {35-02 (35B65 35J60 35J96 53C45)},
  MRNUMBER = {3617963},
MRREVIEWER = {Xiaobing\ Henry\ Feng},
       DOI = {10.4171/170},
       URL = {https://doi.org/10.4171/170},
}

@book {SantambrogioBook,
    AUTHOR = {Santambrogio, Filippo},
     TITLE = {Optimal transport for applied mathematicians},
    SERIES = {Progress in Nonlinear Differential Equations and their
              Applications},
    VOLUME = {87},
      NOTE = {Calculus of variations, PDEs, and modeling},
 PUBLISHER = {Birkh\"auser/Springer, Cham},
      YEAR = {2015},
     PAGES = {xxvii+353},
      ISBN = {978-3-319-20827-5; 978-3-319-20828-2},
   MRCLASS = {49-02 (35J96 49J45 49M29 58E50 90C05 90C48 91B02)},
  MRNUMBER = {3409718},
MRREVIEWER = {Luigi\ De Pascale},
       DOI = {10.1007/978-3-319-20828-2},
       URL = {https://doi.org/10.1007/978-3-319-20828-2},
}

@book {SimonsMinimaxBook,
    AUTHOR = {Simons, Stephen},
     TITLE = {Minimax and monotonicity},
    SERIES = {Lecture Notes in Mathematics},
    VOLUME = {1693},
 PUBLISHER = {Springer-Verlag, Berlin},
      YEAR = {1998},
     PAGES = {xii+172},
      ISBN = {3-540-64755-4},
   MRCLASS = {49-02 (47H04 47H05 47N10 49J35 49J53)},
  MRNUMBER = {1723737},
MRREVIEWER = {George\ Isac},
       DOI = {10.1007/BFb0093633},
       URL = {https://doi.org/10.1007/BFb0093633},
}

@book {EkelandTemamBook,
    AUTHOR = {Ekeland, Ivar and T\'emam, Roger},
     TITLE = {Convex analysis and variational problems},
    SERIES = {Classics in Applied Mathematics},
    VOLUME = {28},
   EDITION = {English},
      NOTE = {Translated from the French},
 PUBLISHER = {Society for Industrial and Applied Mathematics (SIAM),
              Philadelphia, PA},
      YEAR = {1999},
     PAGES = {xiv+402},
      ISBN = {0-89871-450-8},
   MRCLASS = {49-02 (01A75 49J53 90C46)},
  MRNUMBER = {1727362},
       DOI = {10.1137/1.9781611971088},
       URL = {https://doi.org/10.1137/1.9781611971088},
}

@book {VillaniBook2009,
    AUTHOR = {Villani, C.},
     TITLE = {Optimal transport},
    SERIES = {Grundlehren der mathematischen Wissenschaften [Fundamental Principles of Mathematical Sciences]},
    VOLUME = {338},
      NOTE = {Old and new},
 PUBLISHER = {Springer-Verlag, Berlin},
      YEAR = {2009},
     PAGES = {xxii+973},
      ISBN = {978-3-540-71049-3},
   MRCLASS = {49-02 (28A75 37J50 49Q20 53C23 58E30)},
  MRNUMBER = {2459454},
MRREVIEWER = {Dario\ Cordero-Erausquin},
       DOI = {10.1007/978-3-540-71050-9},
       URL = {https://doi.org/10.1007/978-3-540-71050-9},
}

@book {ZaanenBook1997,
    AUTHOR = {Zaanen, Adriaan C.},
     TITLE = {Introduction to operator theory in {R}iesz spaces},
 PUBLISHER = {Springer-Verlag, Berlin},
      YEAR = {1997},
     PAGES = {xii+312},
      ISBN = {3-540-61989-5},
   MRCLASS = {47B60 (46A40 47B65)},
  MRNUMBER = {1631533},
MRREVIEWER = {Yu.\ A.\ Abramovich and C.\ D.\ Aliprantis},
       DOI = {10.1007/978-3-642-60637-3},
       URL = {https://doi.org/10.1007/978-3-642-60637-3},
}

@misc{GalichonHsiehSylvestre2024,
      title={Monotone comparative statics for submodular functions, with an application to aggregated deferred acceptance}, 
      author={Alfred Galichon and Yu-Wei Hsieh and Maxime Sylvestre},
      year={2024},
      eprint={2304.12171},
      archivePrefix={arXiv},
      primaryClass={econ.TH},
      url={https://arxiv.org/abs/2304.12171}, 
}

@article{DINEZZA2012521,
	author = {Eleonora {Di Nezza} and Giampiero Palatucci and Enrico Valdinoci},
	journal = {Bulletin des Sciences Math{\'e}matiques},
	number = {5},
	pages = {521-573},
	title = {Hitchhiker's guide to the fractional Sobolev spaces},
	volume = {136},
	year = {2012}}

@misc{baradat2025usingsinkhornjkoscheme,
      title={Using Sinkhorn in the JKO scheme adds linear diffusion}, 
      author={Aymeric Baradat and Anastasiia Hraivoronska and Filippo Santambrogio},
      year={2025},
      eprint={2502.12666},
      archivePrefix={arXiv},
      primaryClass={math.AP},
      url={https://arxiv.org/abs/2502.12666}, 
}

@article{Jacobs2020TheP,
  title={The $L^1$-contraction principle in optimal transport},
  author={Matt Jacobs and Inwon C. Kim and Jiajun Tong},
  journal={ANNALI SCUOLA NORMALE SUPERIORE - CLASSE DI SCIENZE},
  year={2020},
  url={https://api.semanticscholar.org/CorpusID:219721053}
}

@article{Topkis1978,
  title = {Minimizing a Submodular Function on a Lattice},
  volume = {26},
  ISSN = {1526-5463},
  url = {http://dx.doi.org/10.1287/opre.26.2.305},
  DOI = {10.1287/opre.26.2.305},
  number = {2},
  journal = {Operations Research},
  publisher = {Institute for Operations Research and the Management Sciences (INFORMS)},
  author = {Topkis,  Donald M.},
  year = {1978},
  month = apr,
  pages = {305–321}
}

@book{Murota2003,
  title = {Discrete Convex Analysis},
  ISBN = {9780898718508},
  url = {http://dx.doi.org/10.1137/1.9780898718508},
  DOI = {10.1137/1.9780898718508},
  publisher = {Society for Industrial and Applied Mathematics},
  author = {Murota,  Kazuo},
  year = {2003},
  month = jan 
}

@article{Chen2020,
  title = {S-Convexity and Gross Substitutability},
  ISSN = {1556-5068},
  url = {http://dx.doi.org/10.2139/ssrn.3549632},
  DOI = {10.2139/ssrn.3549632},
  journal = {SSRN Electronic Journal},
  publisher = {Elsevier BV},
  author = {Chen,  Xin and Li,  Menglong},
  year = {2020}
}

@article{Chambolle2009,
  title = {On Total Variation Minimization and Surface Evolution Using Parametric Maximum Flows},
  volume = {84},
  ISSN = {1573-1405},
  url = {http://dx.doi.org/10.1007/s11263-009-0238-9},
  DOI = {10.1007/s11263-009-0238-9},
  number = {3},
  journal = {International Journal of Computer Vision},
  publisher = {Springer Science and Business Media LLC},
  author = {Chambolle,  Antonin and Darbon,  Jér\^ome},
  year = {2009},
  month = apr,
  pages = {288–307}
}

@article{Jordan1998,
  title = {The Variational Formulation of the Fokker--Planck Equation},
  volume = {29},
  ISSN = {1095-7154},
  url = {http://dx.doi.org/10.1137/S0036141096303359},
  DOI = {10.1137/s0036141096303359},
  number = {1},
  journal = {SIAM Journal on Mathematical Analysis},
  publisher = {Society for Industrial & Applied Mathematics (SIAM)},
  author = {Jordan,  Richard and Kinderlehrer,  David and Otto,  Felix},
  year = {1998},
  month = jan,
  pages = {1–17}
}

@article{Bach2013,
  title = {Learning with Submodular Functions: A Convex Optimization Perspective},
  volume = {6},
  ISSN = {1935-8245},
  url = {http://dx.doi.org/10.1561/2200000039},
  DOI = {10.1561/2200000039},
  number = {2–3},
  journal = {Foundations and Trends{\textregistered} in Machine Learning},
  publisher = {Now Publishers},
  author = {Bach,  Francis},
  year = {2013},
  pages = {145–373}
}

@book {Ziemer,
    AUTHOR = {Ziemer, William P.},
     TITLE = {Weakly differentiable functions},
    SERIES = {Graduate Texts in Mathematics},
    VOLUME = {120},
      NOTE = {Sobolev spaces and functions of bounded variation},
 PUBLISHER = {Springer-Verlag, New York},
      YEAR = {1989},
     PAGES = {xvi+308},
      ISBN = {0-387-97017-7},
   MRCLASS = {46E35},
  MRNUMBER = {1014685},
MRREVIEWER = {V.\ M.\ Gol\cprime dshte\u in},
       DOI = {10.1007/978-1-4612-1015-3},
       URL = {https://doi.org/10.1007/978-1-4612-1015-3},
}

@article{Leonard2013ASO,
  title={A survey of the Schr\"odinger problem and some of its connections with optimal transport},
  author={Christian L'eonard},
  journal={arXiv: Probability},
  year={2013},
  url={https://api.semanticscholar.org/CorpusID:14387241}
}

@article{nutz2021introduction,
  title={Introduction to entropic optimal transport},
  author={Nutz, Marcel},
  journal={Lecture notes, Columbia University},
  year={2021}
}

@article {choquet,
    AUTHOR = {Choquet, Gustave},
     TITLE = {Theory of capacities},
   JOURNAL = {Ann. Inst. Fourier (Grenoble)},
  FJOURNAL = {Universit\'e{} de Grenoble. Annales de l'Institut Fourier},
    VOLUME = {5},
      YEAR = {1953/54},
     PAGES = {131--295},
      ISSN = {0373-0956,1777-5310},
   MRCLASS = {27.5X},
  MRNUMBER = {80760},
MRREVIEWER = {Benjamin\ Lepson},
       URL = {http://www.numdam.org/item?id=AIF_1954__5__131_0},
}

@book {jon,
    AUTHOR = {Lee, Jon},
     TITLE = {A first course in combinatorial optimization},
    SERIES = {Cambridge Texts in Applied Mathematics},
 PUBLISHER = {Cambridge University Press, Cambridge},
      YEAR = {2004},
     PAGES = {xvi+211},
      ISBN = {0-521-81151-1; 0-521-01012-8},
   MRCLASS = {90-01 (90C27)},
  MRNUMBER = {2036957},
MRREVIEWER = {Peter\ Butkovi\v c},
       DOI = {10.1017/CBO9780511616655},
       URL = {https://doi.org/10.1017/CBO9780511616655},
}

\end{document}